\documentclass[10pt]{amsart}

\usepackage{amsfonts}
\usepackage{amsmath}
\usepackage{amsthm}
\usepackage{mathrsfs}
\usepackage{amssymb}
\usepackage[arrow,matrix]{xy}
\usepackage[all]{xypic}
\usepackage{graphicx}

\def\dd{\mathinner{\mkern1mu\raise8pt\hbox{.}\mkern2mu
\raise4pt\hbox{.}\mkern2mu\hbox{.}\mkern1mu}}
\def\rdd{\mathinner{\mkern1mu\hbox{.}\mkern2mu
\raise4pt\hbox{.}\mkern2mu\raise8pt\hbox{.}\mkern1mu}}

\newcommand{\mo}{\mathrm{mod}\,}
\newcommand{\ind}{\mathrm{ind}\,}
\newcommand{\Hom}{\mathrm{Hom}}
\newcommand{\Tr}{\mathrm{Tr}}
\newcommand{\Imi}{\mathrm{Im}}

\newcommand{\soc}{\mathrm{soc}}
\newcommand{\topp}{\mathrm{top}}
\newcommand{\supp}{\mathrm{supp}}
\newcommand{\su}{\mathrm{Supp}}
\newcommand{\End}{\mathrm{End}}
\newcommand{\Ext}{\mathrm{Ext}}
\newcommand{\rad}{\mathrm{rad}}
\newcommand{\add}{\mathrm{add}}

\newcommand{\ann}{\mathrm{ann}}
\newcommand{\gd}{\mathrm{gl.\,dim\,}}
\newcommand{\pd}{\mathrm{pd}}
\newcommand{\id}{\mathrm{id}}
\newcommand{\rep}{\mathrm{rep}}
\newcommand{\Fac}{\mathrm{Fac}}

\newtheorem{thm}{Theorem}[section]
\newtheorem{prop}[thm]{Proposition}
\newtheorem{lem}[thm]{Lemma}
\newtheorem{cor}[thm]{Corollary}

\theoremstyle{definition}
\newtheorem{ex}[thm]{Example}

\begin{document}

\title{Finite cycles of indecomposable modules}
%
\author[Malicki]{Piotr Malicki}
\address{Faculty of Mathematics and Computer Science, Nicolaus Copernicus University, Chopina 12/18, 87-100 Toru\'n, Poland}
\email{pmalicki@mat.uni.torun.pl}

\thanks{This work was completed with the support of the research grant DEC-2011/02/A/ST1/00216 of the Polish National Science Center
and the CIMAT Guanajuato, M\'exico.}
\author[de la Pe\~na]{Jos\'e Antonio de la Pe\~na}
\address{Centro de Investigaci\'on en Mathem\'aticas (CIMAT), Guanajuato, M\'exico}
\email{jap@cimat.mx}
%
\author[Skowro\'nski]{Andrzej Skowro\'nski}
\address{Faculty of Mathematics and Computer Science, Nicolaus Copernicus University, Chopina 12/18, 87-100 Toru\'n, Poland}
\email{skowron@mat.uni.torun.pl}


\dedicatory{Dedicated to Raymundo Bautista on the occasion of his 70th birthday}


\baselineskip 15pt

\begin{abstract}
We solve a long standing open problem concerning the structure of finite cycles in the category $\mo A$ of finitely generated modules
over an arbitrary artin algebra $A$, that is, the chains of homomorphisms
$M_0 \buildrel {f_1}\over {\hbox to 6mm{\rightarrowfill}} M_1 \to \cdots \to M_{r-1} \buildrel {f_r}\over {\hbox to 6mm{\rightarrowfill}} M_r=M_0$
between indecomposable modules in $\mo A$ which do not belong to the infinite radical of $\mo A$.
In particular, we describe completely the structure of an arbitrary module category $\mo A$ whose all cycles are finite.
The main structural results of the paper
allow to derive several interesting combinatorial and homological properties of indecomposable modules lying on finite cycles.
For example, we prove that for all but finitely many isomorphism classes of indecomposable modules $M$ lying on finite cycles of a~module
category $\mo A$ the Euler characteristic of $M$ is well defined and nonnegative.
As an another application of these results we obtain a characterization of all cycle-finite module categories $\mo A$ having only
a~finite number of functorially finite torsion classes.
Moreover, new types of examples illustrating the main results of the paper are presented.
\end{abstract}

\maketitle
\begin{center}
{\sc 0. Introduction}
\end{center}

Throughout the paper, by an algebra is meant an artin algebra over a fixed commutative artin ring $K$, which we shall assume (without loss of generality)
to be basic and indecomposable. For an algebra $A$, we denote by $\mo A$ the category of finitely generated right $A$-modules and by $\ind A$ the full
subcategory of $\mo A$ formed by the indecomposable modules. The Jacobson radical $\rad_A$ of $\mo A$ is the ideal generated by all nonisomorphisms between modules
in $\ind A$, and the infinite radical $\rad^{\infty}_A$ of $\mo A$ is the intersection of all powers $\rad^i_A$, $i\geq 1$, of $\rad_A$.
By a result of Auslander \cite{Au}, $\rad_A^{\infty}=0$ if and only if $A$ is of finite representation type, that is, $\ind A$ admits only a finite
number of pairwise nonisomorphic modules (see also \cite{KS0} for an alternative proof of this result).
On the other hand, if $A$ is of infinite representation type then $(\rad_A^{\infty})^2\neq 0$, by a result
proved in \cite{CMMS}.

An important combinatorial and homological invariant of the module category $\mo A$ of an algebra $A$ is its Auslander-Reiten quiver $\Gamma_A$.
Recall that $\Gamma_A$ is a valued translation quiver whose vertices are the isomorphism classes $\{X\}$ of modules $X$ in $\ind A$, the arrows correspond
to irreducible homomorphisms between modules in $\ind A$, and the translation is the Auslander-Reiten translation $\tau_A=D\Tr$.
We shall not distinguish between a module in $X$ in $\ind A$ and the corresponding vertex $\{X\}$ of $\Gamma_A$. If $A$ is an algebra of finite
representation type, then every nonzero nonisomorphism in $\ind A$ is a finite sum of composition of irreducible homomorphisms between modules in $\ind A$,
and hence we may recover $\mo A$ from the translation quiver $\Gamma_A$. In general, $\Gamma_A$ describes only the quotient category $\mo A/\rad^{\infty}_A$.

Let $A$ be an algebra and $M$ a module in $\ind A$. An important information concerning the structure of $M$ is coded in the structure and properties of its
support algebra $\su(M)$ defined as follows. Consider a decomposition $A=P_M\oplus Q_M$ of $A$ in $\mo A$ such that the simple summands of the semisimple module
$P_M/\rad P_M$ are exactly the simple composition factors of $M$. Then $\su(M)=A/t_A(M)$, where $t_A(M)$ is the ideal in $A$ generated by the images of all
homomorphisms from $Q_{M}$ to $A$ in $\mo A$. We note that $M$ is an indecomposable module over $\su(M)$. Clearly, we may realistically hope to describe the
structure of $\su(M)$ only for modules $M$ having some distinguished properties.

A prominent role in the representation theory of algebras is played by cycles of indecomposable modules (see \cite{MPS1}, \cite{MS4}, \cite{Ri1}, \cite{Sk7}).
Recall that a \textit{cycle} in $\ind A$ is a sequence
\[ M_0 \buildrel {f_1}\over {\hbox to 6mm{\rightarrowfill}} M_1 \to \cdots \to M_{r-1} \buildrel {f_r}\over {\hbox to 6mm{\rightarrowfill}} M_r=M_0 \]
of nonzero nonisomorphisms in $\ind A$ \cite{Ri1}, and such a cycle is said to be \textit{finite} if the homomorphisms $f_1,\ldots, f_r$
do not belong to $\rad_A^{\infty}$ (see \cite{AS1}, \cite{AS2}). Following Ringel \cite{Ri1}, a module $M$ in $\ind A$ which does not lie on a cycle in
$\ind A$ is called \textit{directing}. The following two important results on directing modules were established by Ringel in \cite{Ri1}.
Firstly, if $A$ is an algebra with all modules in $\ind A$ being directing, then $A$ is of finite representation type. Secondly, the support algebra $\su(M)$
of a directing module $M$ over an algebra $A$ is a tilted algebra $\End_H(T)$, for a hereditary algebra $H$ and a tilting module $T$ in $\mo H$, and $M$ is
isomorphic to the image $\Hom_H(T,I)$ of an indecomposable injective module $I$ in $\mo H$ via the functor $\Hom_H(T,-): \mo H \to \mo\End_H(T)$.
In particular, it follows that, if $A$ is an algebra of infinite representation type, then $\ind A$ always contains a cycle. Moreover, it has been proved
independently by Peng and Xiao \cite{PX} and Skowro\'nski \cite{Sk4} that the Auslander-Reiten quiver $\Gamma_A$ of an
algebra $A$ admits at most finitely many $\tau_A$-orbits containing directing modules. Hence, in order to obtain information on the support algebras $\su(M)$
of nondirecting modules in $\ind A$, it is natural to study properties of cycles in $\ind A$ containing $M$.
A module $M$ in $\ind A$ is said to be \textit{cycle-finite} if $M$ is nondirecting and every cycle in $\ind A$ passing through $M$ is finite.
Obviously, every indecomposable module over an algebra of finite representation type is cycle-finite. Examples of cycle-finite indecomposable modules
over algebras of infinite representation type are provided by all indecomposable modules in the stable tubes of tame hereditary algebras \cite{DlRi},
canonical algebras \cite{Ri1}, \cite{Ri2}, or more generally concealed canonical algebras \cite{LP}. Following Assem and Skowro\'nski \cite{AS1}, \cite{AS2},
an algebra $A$ is said to be \textit{cycle-finite} if all cycles in $\ind A$ are finite. The class of cycle-finite algebras is wide and contains the
following distinguished classes of algebras: the algebras of finite representation type, the tame tilted algebras \cite{HR}, \cite{Ke}, \cite{Ri1},
the tame double tilted algebras \cite{RS1},
the tame generalized double tilted algebras \cite{RS2},
the tubular algebras \cite{Ri1}, \cite{Ri2}, the iterated tubular algebras \cite{PTo}, the tame quasi-tilted algebras \cite{LS}, \cite{Sk10},
the tame generalized multicoil algebras \cite{MS2}, the algebras with cycle-finite derived categories \cite{AS1}, and the strongly simply connected
algebras of polynomial growth \cite{Sk9}.
We also mention that a selfinjective algebra $A$ is cycle-finite if and only if $A$ is of finite representation type \cite{JS}.
On the other hand, frequently an algebra $A$ admits a Galois covering $R\to R/G = A$, where $R$ is a cycle-finite locally bounded category and $G$
is an admissible group of automorphisms of $R$, which allows to reduce the representation theory of $A$ to the representation theory of cycle-finite
algebras being finite convex subcategories of $R$ (see \cite{PS3} and \cite{Sk9-5} for some general results).
For example, every finite dimensional selfinjective algebra $A$ of polynomial growth over an
algebraically closed field $K$ admits a canonical standard form $\overline A$ (geometric socle deformation of $A$) such that $\overline A$
has a Galois covering $R\to R/G = \overline A$, where $R$ is a cycle-finite selfinjective locally bounded category and $G$ is an admissible infinite
cyclic group of automorphisms of $R$, the Auslander-Reiten quiver $\Gamma_{\overline A}$ of $\overline A$ is the orbit quiver $\Gamma_R/G$ of $\Gamma_R$,
and the stable Auslander-Reiten quivers of $A$ and $\overline A$ are isomorphic (see \cite{Sk1}, \cite{Sk13}).
We refer to \cite{BSSW}, \cite{MPS1}, \cite{Sk8} for some general results on the structure of cycle-finite algebras and their module categories.

In the paper we are concerned with the problem of describing the support algebras of cycle-finite modules over arbitrary (artin) algebras.
We note that this may be considered as a natural extension of the problem concerning the structure of support algebras of directing modules,
solved by Ringel in \cite{Ri1}. Namely, the directing modules in $\ind A$ may be viewed as modules $M$ in $\ind A$ for which every oriented cycle
of nonzero homomorphisms in $\ind A$ containing $M$ consists entirely of isomorphisms. The considered problem, initiated more than 25 years ago in
\cite{AS1}, turned out to be very difficult, and many researchers involved to its solution resigned. The main obstacle for solution of this problem was
the large complexity of finite cycles of indecomposable modules and the fact that all cycles of indecomposable modules over algebras of finite representation
type are finite. The main results of the paper show that new classes of algebras and complete understanding of the structure of their module categories
were necessary for the solution of the considered problem. We will outline now our approach towards solution of the problem.

Let $A$ be an algebra and $M$ be a cycle-finite module in $\ind A$. Then every cycle in $\ind A$ passing through $M$ has a refinement to a cycle
of irreducible homomorphisms in $\ind A$ containing $M$ and consequently $M$ lies on an oriented cycle in the Auslander-Reiten quiver $\Gamma_A$ of $A$.
Following Malicki and Skowro\'nski \cite{MS1}, we denote by $_c\Gamma_A$ the \textit{cyclic quiver} of $A$ obtained from $\Gamma_A$ by removing all acyclic vertices
(vertices not lying on oriented cycles in $\Gamma_A$) and the arrows attached to them. Then the connected components of the translation quiver
$_c\Gamma_A$ are said to be \textit{cyclic components} of $\Gamma_A$. It has been proved in \cite{MS1} that two modules $X$ and $Y$ in $\ind A$
belong to the same cyclic component of $\Gamma_A$ if and only if there is an oriented cycle in $\Gamma_A$ passing through $X$ and $Y$. For a cyclic
component $\Gamma$ of $_c\Gamma_A$, we consider a decomposition $A=P_{\Gamma}\oplus Q_{\Gamma}$ of $A$ in $\mo A$ such that the simple summands
of the semisimple module $P_{\Gamma}/\rad P_{\Gamma}$ are exactly the simple composition factors of indecomposable modules in $\Gamma$, the ideal $t_A(\Gamma)$
in $A$ generated by the images of all homomorphisms from $Q_{\Gamma}$ to $A$ in $\mo A$, and call the quotient algebra $\su(\Gamma)=A/t_A(\Gamma)$ the
\textit{support algebra} of $\Gamma$. Observe now that $M$ belongs to a unique cyclic component $\Gamma(M)$ of $\Gamma_A$ consisting entirely of cycle-finite
indecomposable modules, and the support algebra $\su(M)$ of $M$ is a quotient algebra of the support algebra $\su(\Gamma(M))$ of $\Gamma(M)$. A cyclic component
$\Gamma$ of $\Gamma_A$ containing a cycle-finite module is said to be a cycle-finite cyclic component of $\Gamma_A$.
We will prove that the support algebra $\su(\Gamma)$ of a cycle-finite cyclic component $\Gamma$ of $\Gamma_A$ is isomorphic to an algebra of the form
$e_{\Gamma}Ae_{\Gamma}$ for an idempotent $e_{\Gamma}$ of $A$ whose primitive summands correspond to the vertices of a convex subquiver of the valued quiver
$Q_A$ of $A$. On the other hand, the support algebra $\su(M)$ of a cycle-finite module $M$ in $\ind A$ is not necessarily an algebra of the form
$eAe$ for an idempotent $e$ of $A$ (see Section \ref{exs-fin}).

The main results of the paper provide
a conceptual description of the support algebras of cycle-finite cyclic components of $\Gamma_A$. The description splits into two cases.
In the case when a cycle-finite cyclic component $\Gamma$ of $\Gamma_A$ is infinite, we prove that $\su(\Gamma)$ is a suitable gluing  of finitely many
generalized multicoil algebras (introduced by Malicki and Skowro\'nski in \cite{MS2}) and algebras of finite representation type, and $\Gamma$ is the
corresponding gluing of the associated cyclic generalized multicoils via finite translation quivers. In the second case when a cycle-finite cyclic component
$\Gamma$ is finite, we prove that $\su(\Gamma)$ is a generalized double tilted algebra (in the sense of Reiten and Skowro\'nski \cite{RS2}) and
$\Gamma$ is the core of the connecting component of this algebra.

We would like to mention that the generalized multicoil algebras form a prominent class of algebras of global dimension at most 3, containing the class
of quasitilted algebras of canonical type, and are obtained by sophisticated gluings of concealed canonical algebras using admissible algebra operations,
generalizing the coil operations proposed by Assem and Skowro\'nski in \cite{AS2}. The generalized double tilted algebras form a distinguished
class of algebras, containing all tilted algebras and all algebras of finite representation type, and can be viewed as two-sided gluings of tilted
algebras. The tilted algebras and quasitilted algebras of canonical type were under intensive investigation over the last two decades by many
representation theory algebraists. Hence, the main results of the paper give a good understanding of the support algebras of cycle-finite cyclic components.
On the other hand, the results and examples presented in the paper create new interesting open problems and research directions (see Section \ref{main}).

The paper is organized as follows. In Section \ref{main} we present the main results of the paper and related background.
In Section \ref{cyccomp} we describe properties of cyclic components of the Auslander-Reiten quivers of algebras, applied in the proofs of the main
theorems. Sections \ref{pfth1}, \ref{pfth2} and \ref{pfthm10} are devoted to the proofs of Theorems \ref{thm1}, \ref{thm2} and \ref{thm10}, respectively.
In Sections \ref{exs-inf} and \ref{exs-fin} we present new types of examples, illustrating the main results of the paper.

For basic background on the representation theory applied here we refer to \cite{ASS}, \cite{ARS}, \cite{Ri1}, \cite{SS1}, \cite{SS2}, \cite{SY}.

The main results of the paper have been proved during the visit of P. Malicki and A. Skowro\'nski at the Centro de Investigaci\'on en Mathem\'aticas (CIMAT)
in Guanajuato (November 2012), who would like to thank J. A. de la Pe\~na and CIMAT for the warm hospitality and wonderful conditions for
the successful realization of this joint research project. The results were presented by the first named author during the conferences
"Advances in Representation Theory of Algebras" (Guanajuato, December 2012) and "Perspectives of Representation Theory of Algebras" (Nagoya, November 2013).

\section{Main results and related background}
\label{main}





In order to formulate the main results of the paper we need special types of components of the
Auslander-Reiten quivers of algebras and distinguished classes of algebras with separating families of Auslander-Reiten components.

Let $A$ be an algebra. For a subquiver $\Gamma$ of $\Gamma_A$, we denote by $\ann_A(\Gamma)$ the intersection of the annihilators $\ann_A(X)=\{a\in A\mid Xa=0\}$
of all indecomposable modules $X$ in $\Gamma$, and call the quotient algebra $B(\Gamma)=A/\ann_A(\Gamma)$ the \textit{faithful algebra} of $\Gamma$.
By a component of $\Gamma_A$ we mean a connected component of the translation quiver $\Gamma_A$.
A component $\mathcal C$ of $\Gamma_A$ is called \textit{regular} if $\mathcal C$ contains neither a projective module nor an
injective module, and \textit{semiregular} if $\mathcal C$ does not contain both a projective and an injective module. It has been shown in \cite{Li1}
and \cite{Zh} that a regular component $\mathcal C$ of $\Gamma_A$ contains an oriented cycle if and only if $\mathcal C$ is a \textit{stable tube}
(is of the form ${\Bbb Z}{\Bbb A}_{\infty}/(\tau^{r})$, for a positive integer $r$). Moreover, Liu proved in \cite{Li2} that a semiregular component
$\mathcal C$ of $\Gamma_A$ contains an oriented cycle if and only if $\mathcal C$ is a \textit{ray tube} (obtained from a stable tube by a finite number
(possibly zero) of ray insertions) or a \textit{coray tube} (obtained from a stable tube by a finite number (possibly zero) of coray insertions).
A component $\mathcal C$ of $\Gamma_A$ is said to be \textit{coherent} \cite{MS1} (see also \cite{DR}) if the following two conditions are satisfied:

(C1) For each projective module $P$ in $\mathcal C$ there is an infinite sectional path

$\hspace{2cm}P~=~X_1 \rightarrow X_2 \rightarrow \cdots \rightarrow X_i\rightarrow X_{i+1} \rightarrow X_{i+2} \rightarrow \cdots .$ 

(C2) For each injective module $I$ in $\mathcal C$ there is  an infinite sectional path

$\hspace{2cm}\cdots \rightarrow Y_{j+2} \rightarrow Y_{j+1} \rightarrow Y_j \rightarrow \cdots \rightarrow Y_2 \rightarrow Y_1 = I.$ 

\noindent Further, a component $\mathcal C$ of $\Gamma_A$ is said to be \textit{almost cyclic} if its cyclic part $_c{\mathcal C}$ is a cofinite subquiver of $\mathcal C$.
%
We note that the stable tubes, ray tubes and coray tubes of $\Gamma_A$ are special types of
almost cyclic coherent components. In general, it has been proved by Malicki and Skowro\'nski in \cite{MS1} that a component
$\mathcal C$ of $\Gamma_A$ is almost cyclic and coherent
if and only if $\mathcal C$ is a \textit{generalized multicoil}, obtained from a finite family of stable tubes by a sequence of admissible operations (ad~1)-(ad~5)
and their duals (ad~1$^*$)-(ad~5$^*$). On the other hand, a component $\mathcal C$ of $\Gamma_A$ is said to be \textit{almost acyclic} if all but finitely many
modules of $\mathcal C$ are acyclic. It has been proved by Reiten and Skowro\'nski in \cite{RS2} that a component $\mathcal C$ of $\Gamma_A$ is almost acyclic
if and only if $\mathcal C$ admits a multisection $\Delta$. Moreover, for an almost acyclic component $\mathcal C$ of $\Gamma_A$, there exists a finite convex
subquiver $c(\mathcal C)$ of $\mathcal C$ (possibly empty), called the \textit{core} of $\mathcal C$, containing all modules lying on oriented cycles in $\mathcal C$
(see \cite{RS2} for details).
A family $\mathcal C=({\mathcal C}_i)_{i\in I}$ of components of $\Gamma_A$ is said to be \textit{generalized standard} if $\rad_A^{\infty}(X,Y)=0$ for all modules
$X$ and $Y$ in $\mathcal C$ \cite{Sk3}, and \textit{sincere} if every simple module in $\mo A$ occurs as a composition factor of a module in $\mathcal C$.
Finally, following Assem, Skowro\'nski and Tom\'e \cite{AST}, a family ${\mathcal C}$ = $({\mathcal C}_{i})_{i \in I}$ of components of $\Gamma_A$ is said to be \textit{separating} if the
components in $\Gamma_A$ split into three disjoint families ${\mathcal P}^A$, ${\mathcal C}^A={\mathcal C}$ and ${\mathcal Q}^A$ such that:

(S1) ${\mathcal C}^A$ is a sincere generalized standard family of components;

(S2) $\Hom_{A}({\mathcal Q}^A,{\mathcal P}^A) = 0$, $\Hom_{A}({\mathcal Q}^A,{\mathcal C}^A)=0$, $\Hom_{A}({\mathcal C}^A,{\mathcal P}^A) = 0$;

(S3) any morphism from ${\mathcal P}^A$ to ${\mathcal Q}^A$ in $\mo A$ factors through the additive
category $\add({\mathcal C}^A)$ of ${\mathcal C}^A$. \\
We then say that ${\mathcal C}^A$ separates ${\mathcal P}^A$ from ${\mathcal Q}^A$ and write
$$\Gamma_A={\mathcal P}^A \cup {\mathcal C}^A \cup {\mathcal Q}^A.$$
We mention that then the families ${\mathcal P}^A$ and ${\mathcal Q}^A$ are uniquely determined by the separating family ${\mathcal C}^A$, and
${\mathcal C}^A$ is a faithful family of components in $\Gamma_A$, that is, $\ann_A({\mathcal C}^A)=0$.

In the representation theory of algebras an important role is played by the canonical algebras introduced by Ringel in \cite{Ri1} and \cite{Ri2}.
Every canonical algebra $\Lambda$ is of global dimension at most $2$ and its Auslander-Reiten quiver $\Gamma_{\Lambda}$ admits a canonical separating
family ${\mathcal T}^{\Lambda}$ of stable tubes, so $\Gamma_{\Lambda}$ admits a disjoint union decomposition
$\Gamma_{\Lambda}={\mathcal P}^{\Lambda} \cup {\mathcal T}^{\Lambda} \cup {\mathcal Q}^{\Lambda}$. Then an algebra $C$ of the form $\End_{\Lambda}(T)$,
with $T$ a tilting module in the additive category $\add({\mathcal P}^{\Lambda})$ of ${\mathcal P}^{\Lambda}$ is called a \textit{concealed canonical algebra}
of type $\Lambda$, and ${\mathcal T}^{C}=\Hom_{\Lambda}(T,{\mathcal T}^{\Lambda})$ is a separating family of stable tubes in $\Gamma_C$, so we have
a disjoint union decomposition $\Gamma_C={\mathcal P}^C \cup {\mathcal T}^C \cup {\mathcal Q}^C$. It has been proved by Lenzing and de la Pe\~na in
\cite{LP} that an algebra $A$ is a concealed canonical algebra if and only if $\Gamma_A$ admits a separating family ${\mathcal T}^A$ of stable tubes.
The concealed canonical algebras form a distinguished class of \textit{quasitilted algebras}, which are the endomorphism algebras $\End_{\mathscr H}(T)$ of tilting
objects $T$ in abelian hereditary $K$-categories ${\mathscr H}$ \cite{HRS}. By a result due to  Happel, Reiten and Smal{\o} proved in \cite{HRS}, an algebra $A$ is
a quasitilted algebra if and only if $\gd A\leq 2$ and every module $X$ in $\ind A$ satisfies $\pd_AX\leq 1$ or $\id_AX\leq 1$. Further, it has been
proved by Happel and Reiten in \cite{HRe} that the class of quasitilted algebras consists of the \textit{tilted algebras} (the endomorphism algebras
$\End_{H}(T)$ of tilting modules $T$ over hereditary algebras $H$) and the \textit{quasitilted algebras of canonical type} (the endomorphism
algebras $\End_{\mathscr H}(T)$ of tilting objects $T$ in abelian hereditary categories ${\mathscr H}$ whose derived category $D^b({\mathscr H})$
is equivalent to the derived category $D^b(\mo\Lambda)$ of the module category $\mo\Lambda$ of a canonical algebra $\Lambda$). Moreover, it has been
proved by Lenzing and Skowro\'nski in \cite{LS} (see also \cite{Sk10}) that an algebra $A$ is a quasitilted algebra of canonical type if and only if
$\Gamma_A$ admits a separating family ${\mathcal T}^{A}$ of semiregular tubes (ray and coray tubes), and if and only if $A$ is a semiregular branch
enlargement of a concealed canonical algebra $C$. We are now in position to introduce the class of generalized multicoil algebras \cite{MS2}, being
sophisticated gluings of quasitilted algebras of canonical type, playing the fundamental role in first main result of the paper. It has been proved
by Malicki and Skowro\'nski in \cite{MS2} that the Auslander-Reiten quiver $\Gamma_A$ of an algebra $A$ admits a separating family of almost cyclic
coherent components if and only if $A$ is a \textit{generalized multicoil algebra}, that is, a generalized multicoil enlargement of a product
$C=C_1\times\ldots\times C_m$ of concealed canonical algebras $C_1, \ldots, C_m$ using modules from the separating families
${\mathcal T}^{C_1}, \ldots, {\mathcal T}^{C_m}$ of stable tubes of $\Gamma_{C_1}, \ldots, \Gamma_{C_m}$ and a sequence of admissible operations of types
(ad~1)-(ad~5) and their duals (ad~1$^*$)-(ad~5$^*$). For a generalized multicoil algebra $A$, there is a unique quotient algebra $A^{(l)}$ of $A$ which
is a product of quasitilted algebras of canonical type having separating families of coray tubes (the \textit{left quasitilted algebra} of $A$) and
a unique quotient algebra $A^{(r)}$ of $A$ which is a product of quasitilted algebras of canonical type having separating families of ray tubes
(the \textit{right quasitilted algebra} of $A$) such that $\Gamma_A$ has a disjoint union decomposition (see \cite[Theorems C and E]{MS2})
$$\Gamma_A={\mathcal P}^A \cup {\mathcal C}^A \cup {\mathcal Q}^A,$$
where
\begin{itemize}
\item ${\mathcal P}^A$ is the left part ${\mathcal P}^{A^{(l)}}$ in a decomposition
$\Gamma_{A^{(l)}}={\mathcal P}^{A^{(l)}}\cup {\mathcal T}^{A^{(l)}}\cup {\mathcal Q}^{A^{(l)}}$ of the Auslander-Reiten quiver $\Gamma_{A^{(l)}}$ of
the left quasitilted algebra $A^{(l)}$ of $A$, with ${\mathcal T}^{A^{(l)}}$ a family of coray tubes separating ${\mathcal P}^{A^{(l)}}$ from
${\mathcal Q}^{A^{(l)}}$;
\item ${\mathcal Q}^A$ is the right part ${\mathcal Q}^{A^{(r)}}$ in a decomposition
$\Gamma_{A^{(r)}}={\mathcal P}^{A^{(r)}}\cup {\mathcal T}^{A^{(r)}}\cup {\mathcal Q}^{A^{(r)}}$ of the Auslander-Reiten quiver $\Gamma_{A^{(r)}}$ of
the right quasitilted algebra $A^{(r)}$ of $A$, with ${\mathcal T}^{A^{(r)}}$ a family of ray tubes separating ${\mathcal P}^{A^{(r)}}$ from
${\mathcal Q}^{A^{(r)}}$;
\item ${\mathcal C}^A$ is a family of generalized multicoils separating ${\mathcal P}^A$ from ${\mathcal Q}^A$, obtained from stable tubes in the
separating families ${\mathcal T}^{C_1}, \ldots, {\mathcal T}^{C_m}$ of stable tubes of the Auslander-Reiten quivers $\Gamma_{C_1}, \ldots, \Gamma_{C_m}$
of the concealed canonical algebras $C_1, \ldots, C_m$ by a sequence of admissible operations of types (ad~1)-(ad~5) and their duals (ad~1$^*$)-(ad~5$^*$),
corresponding to the admissible operations leading from $C=C_1\times\ldots\times C_m$ to $A$;
\item ${\mathcal C}^A$ consists of cycle-finite modules and contains all indecomposable modules of ${\mathcal T}^{A^{(l)}}$ and ${\mathcal T}^{A^{(r)}}$;
\item ${\mathcal P}^A$ contains all indecomposable modules of ${\mathcal P}^{A^{(r)}}$;
\item ${\mathcal Q}^A$ contains all indecomposable modules of ${\mathcal Q}^{A^{(l)}}$.
\end{itemize}
Moreover, in the above notation, we have
\begin{itemize}
\item $\gd A\leq 3$;
\item $\pd_AX\leq 1$ for any indecomposable module $X$ in ${\mathcal P}^A$;
\item $\id_AY\leq 1$ for any indecomposable module $Y$ in ${\mathcal Q}^A$;
\item $\pd_AM\leq 2$ and $\id_AM\leq 2$ for any indecomposable module $M$ in ${\mathcal C}^A$.
\end{itemize}
A generalized multicoil algebra $A$ is said to be \textit{tame} if $A^{(l)}$ and $A^{(r)}$ are products of tilted algebras of Euclidean types or tubular algebras.
We also note that every tame generalized multicoil algebra is a cycle-finite algebra.

The following theorem is the first main result of the paper.
\begin{thm} \label{thm1}
Let $A$ be an algebra and $\Gamma$ be a cycle-finite infinite component of $_c\Gamma_A$. Then there exist infinite
full translation subquivers $\Gamma_1, \ldots, \Gamma_r$ of $\Gamma$ such that the following statements hold.
\begin{enumerate}
\renewcommand{\labelenumi}{\rm(\roman{enumi})}
\item For each $i\in\{1,\ldots,r\}$, $\Gamma_i$ is a cyclic coherent full translation subquiver of $\Gamma_A$.
\item For each $i\in\{1,\ldots,r\}$, $\su(\Gamma_i)=B(\Gamma_i)$ and is a generalized multicoil algebra.
\item $\Gamma_1, \ldots, \Gamma_r$ are pairwise disjoint full translation subquivers of $\Gamma$ and
$\Gamma^{cc}=\Gamma_1\cup\ldots\cup\Gamma_r$ is a maximal cyclic coherent and cofinite full translation subquiver of $\Gamma$.
\item $B(\Gamma\setminus\Gamma^{cc})$ is of finite representation type.
\item $\su(\Gamma)=B(\Gamma)$.
\end{enumerate}
\end{thm}

It follows from the above theorem that all but finitely many modules lying in an infinite cycle-finite component $\Gamma$
of $_c\Gamma_A$ can be obtained from indecomposable modules in stable tubes of concealed canonical algebras by a finite sequence
of admissible operations of types (ad~1)-(ad~5) and their duals (ad~1$^*$)-(ad~5$^*$) (see \cite[Section 3]{MS2} for details). We refer also
to \cite{KS} and \cite{Sk5} for some results on the composition factors of indecomposable modules lying in stable tubes of the Auslander-Reiten
quivers of concealed canonical algebras, and to \cite{MST} for the structure of indecomposable modules lying in coils.
We would like to stress that the cycle-finiteness assumption imposed on the infinite component $\Gamma$ of $_c\Gamma_A$ is essential for the validity
of the above theorem. Namely, it has been proved in \cite{Sk11}, \cite{Sk12} that, for an arbitrary finite dimensional algebra $B$ over a field
$K$, a module $M$ in $\mo B$, and a positive integer $r$, there exists a finite dimensional algebra $A$ over $K$ such that $B$ is a quotient
algebra of $A$, $\Gamma_A$ admits a faithful generalized standard stable tube $\mathcal T$ of rank $r$, $\mathcal T$ is not cycle-finite, and $M$
is a subfactor of all but finitely many indecomposable modules in $\mathcal T$. This shows that in general the problem of describing the support
algebras of infinite cyclic components (even stable tubes) of Auslander-Reiten quivers is difficult.

In order to present the second main result of the paper, we need the class of generalized double tilted algebras introduced by Reiten and Skowro\'nski
in \cite{RS2} (see also \cite{AC}, \cite{CL} and \cite{RS1}). A \textit{generalized double tilted algebra} is an algebra $B$ for which $\Gamma_B$ admits a separating
almost acyclic component $\mathcal C$.

For a generalized double tilted algebra $B$, the Auslander-Reiten quiver $\Gamma_B$ has a disjoint union decomposition (see \cite[Section 3]{RS2})
$$\Gamma_B={\mathcal P}^B \cup {\mathcal C}^B \cup {\mathcal Q}^B,$$
where
\begin{itemize}
\item ${\mathcal C}^B$ is an almost acyclic component separating ${\mathcal P}^B$ from ${\mathcal Q}^B$, called a \textit{connecting component} of $\Gamma_B$;
\item There exist hereditary algebras $H_1^{(l)}, \ldots, H_m^{(l)}$ and tilting modules $T_1^{(l)}\in\mo H_1^{(l)}, \ldots,$ $T_m^{(l)}\in\mo H_m^{(l)}$
such that the tilted algebras
$B_1^{(l)}=\End_{H_1^{(l)}}(T_1^{(l)}), \ldots, B_m^{(l)}=\End_{H_m^{(l)}}(T_m^{(l)})$ are quotient algebras of $B$ and ${\mathcal P}^B$ is the
disjoint union of all components of $\Gamma_{B_1^{(l)}}, \ldots, \Gamma_{B_m^{(l)}}$ contained entirely in the torsion-free parts
${\mathscr Y}(T_1^{(l)}), \ldots,$ ${\mathscr Y}(T_m^{(l)})$ of $\mo B_1^{(l)}, \ldots,$ $\mo B_m^{(l)}$ determined by $T_1^{(l)}, \ldots, T_m^{(l)}$;
\item There exist hereditary algebras $H_1^{(r)}, \ldots, H_n^{(r)}$ and tilting modules
$T_1^{(r)}\in\mo H_1^{(r)}, \ldots,$ $T_n^{(r)}\in\mo H_n^{(r)}$ such that the tilted algebras
$B_1^{(r)}=\End_{H_1^{(r)}}(T_1^{(r)}), \ldots, B_n^{(r)}=\End_{H_n^{(r)}}(T_n^{(r)})$ are quotient algebras of $B$ and ${\mathcal Q}^B$ is the
disjoint union of all components of $\Gamma_{B_1^{(r)}}, \ldots, \Gamma_{B_n^{(r)}}$ contained entirely in the torsion parts
${\mathscr X}(T_1^{(r)}), \ldots,$ ${\mathscr X}(T_n^{(r)})$ of $\mo B_1^{(r)}, \ldots,$ $\mo B_n^{(r)}$ determined by $T_1^{(r)}, \ldots, T_n^{(r)}$;
\item every indecomposable module in ${\mathcal C}^B$ not lying in the core $c({\mathcal C}^B)$ of ${\mathcal C}^B$ is an indecomposable module
over one of the tilted algebras $B_1^{(l)}, \ldots, B_m^{(l)}$, $B_1^{(r)}, \ldots,$ $B_n^{(r)}$;
\item every nondirecting indecomposable module in ${\mathcal C}^B$ is cycle-finite and lies in $c({\mathcal C}^B)$;
\item $\pd_BX\leq 1$ for all indecomposable modules $X$ in ${\mathcal P}^B$;
\item $\id_BY\leq 1$ for all indecomposable modules $Y$ in ${\mathcal Q}^B$;
\item for all but finitely many indecomposable modules $M$ in ${\mathcal C}^B$, we have $\pd_BM\leq 1$ or $\id_BM\leq 1$.
\end{itemize}
Then $B^{(l)}=B_1^{(l)}\times\ldots\times B_m^{(l)}$ is called the \textit{left tilted algebra} of $B$ and
$B^{(r)}=B_1^{(r)}\times\ldots\times B_n^{(r)}$ is called the \textit{right tilted algebra} of $B$.
We note that the class of algebras of finite representation type coincides with the class of generalized double tilted algebras $B$ with $\Gamma_B$
being the connecting component ${\mathcal C}^B$ (equivalently, with the tilted algebras $B^{(l)}$ and $B^{(r)}$ being of finite representation type (possibly empty)).
Finally, a generalized double tilted algebra is said to be \textit{tame} if the tilted algebras $B^{(l)}$ and $B^{(r)}$ are generically tame in
the sense of Crawley-Boevey \cite{CB1}, \cite{CB2}. We note that every tame
generalized double tilted algebra is a cycle-finite algebra. We would like to mention that there exist generalized double tilted algebras of infinite
representation type of arbitrary global dimension $d\in\mathbb{N}\cup\{\infty\}$.
We refer also to \cite{JMS1}, \cite{Li3}, \cite{Sk2} for useful characterizations of tilted algebras.

The following theorem is the second main result of the paper.
\begin{thm} \label{thm2}
Let $A$ be an algebra and $\Gamma$ be a cycle-finite finite component of $_c\Gamma_A$. Then the following statements hold.
\begin{enumerate}
\renewcommand{\labelenumi}{\rm(\roman{enumi})}
\item $\su(\Gamma)$ is a generalized double tilted algebra.
\item $\Gamma$ is the core $c({\mathcal C}^{B(\Gamma)})$ of a unique almost acyclic connecting component ${\mathcal C}^{B(\Gamma)}$  of $\Gamma_{B(\Gamma)}$.
\item $\su(\Gamma)=B(\Gamma)$.
\end{enumerate}
\end{thm}

We would like to point that every finite cyclic component $\Gamma$ of an Auslander-Reiten quiver $\Gamma_A$ contains both a projective module and
an injective module (see Corollary \ref{cor25}), and hence $\Gamma_A$ admits at most finitely many finite cyclic components.
We refer also to \cite{LaSm}, \cite{Sm1}, \cite{Sm2} for some results concerning double tilted algebras with connecting components containing
nondirecting indecomposable modules.

An idempotent $e$ of an algebra $A$ is said to be \textit{convex} provided $e$ is a sum of pairwise orthogonal primitive idempotents of $A$
corresponding to the vertices of a convex valued subquiver of the quiver $Q_A$ of $A$ (see Section 2 for definition). The following direct
consequence of Theorems \ref{thm1}, \ref{thm2} and Propositions \ref{prop22}, \ref{prop23nn} provides a handy description of the faithful
algebra of a cycle-finite component of $_c\Gamma_A$.

\begin{cor} \label{cor3}
Let $A$ be an algebra and $\Gamma$ be a cycle-finite component of $_c\Gamma_A$. Then there exists
a convex idempotent $e_{\Gamma}$ of $A$ such that $\su(\Gamma)$ is isomorphic to the algebra $e_{\Gamma}Ae_{\Gamma}$.
\end{cor}

The third main result of the paper is a consequence of Theorems \ref{thm1} and \ref{thm2}, and the results established in \cite[Theorem 1.3]{MS3}.
\begin{thm} \label{thm4}
Let $A$ be an algebra. Then, for all but finitely many isomorphism classes of cycle-finite modules $M$ in $\ind A$,
the following statements hold.
\begin{enumerate}
\renewcommand{\labelenumi}{\rm(\roman{enumi})}
\item $|\Ext_A^1(M,M)|\leq |\End_A(M)|$ and $\Ext_A^r(M,M)=0$ for $r\geq 2$.
\item $|\Ext_A^1(M,M)| = |\End_A(M)|$ if and only if there is a quotient concealed canonical algebra $C$ of $A$ and a stable tube $\mathcal T$ of $\Gamma_C$
such that $M$ is an indecomposable $C$-module in $\mathcal T$ of quasi-length divisible by the rank of $\mathcal T$.
\end{enumerate}
\end{thm}

Here, $|V|$ denotes the length of a module $V$ in $\mo K$. In particular, the above theorem shows that, for all but finitely many
isomorphism classes of cycle-finite modules $M$ in a module category $\ind A$, the Euler characteristic
$$\chi_A(M)=\sum_{i=0}^{\infty}(-1)^i|\Ext_A^i(M,M)|$$
of $M$ is well defined and nonnegative. We would like to mention that there are cycle-finite algebras $A$ with indecomposable modules $M$ lying in
infinite cyclic components of $\Gamma_A$ and the Euler characteristic $\chi_A(M)$ being an arbitrary given positive integer (see \cite{PS2}).

Let $A$ be an algebra and $K_0(A)$ the Grothendieck group of $A$. For a module $M$ in $\mo A$, we denote by $[M]$ the image of $M$ in $K_0(A)$.
Then $K_0(A)$ is a free abelian group with a ${\Bbb Z}$-basis given by $[S_1], \ldots, [S_n]$ for a complete family $S_1, \ldots, S_n$ of pairwise
nonisomorphic simple modules in $\mo A$. Thus, for modules $M$ and $N$ in $\mo A$, we have $[M]=[N]$ if and only if the modules $M$ and $N$ have
the same composition factors including the multiplicities. In particular, it would be interesting to find sufficient conditions for a module $M$
in $\ind A$ to be uniquely determined (up to isomorphism) by its composition factors (see \cite{RSS} for a general result in this direction).

The next theorem provides information on the composition factors of cycle-finite modules, and is a direct consequence
of Theorems \ref{thm1}, \ref{thm2}, \ref{thm4} and the results established in \cite[Theorems A and B]{Ma}.
\begin{thm} \label{thm5}
Let $A$ be an algebra. The following statements hold.
\begin{enumerate}
\renewcommand{\labelenumi}{\rm(\roman{enumi})}
\item There is a positive integer $m$ such that, for any cycle-finite module $M$ in $\ind A$ with $|\End_A(M)|$ $\neq |\Ext_A^1(M,M)|$, the number
of isomorphism classes of modules $X$ in $\ind A$ with $[X]=[M]$ is bounded by $m$.
\item For all but finitely many isomorphism classes of cycle-finite modules $M$ in $\ind A$ with $|\End_A(M)|$ $= |\Ext_A^1(M,M)|$, there are infinitely
many pairwise nonisomorphic modules $X$ in $\ind A$ with $[X]=[M]$.
\item The number of isomorphism classes of cycle-finite modules $M$ in $\ind A$ with $\Ext_A^1(M,M)$ $= 0$ is finite.
\end{enumerate}
\end{thm}

A module $M$ is in $\mo A$ with $\Ext_A^1(M,M) = 0$ is frequently called \textit{rigid}. Following Adachi, Iyama and Reiten \cite{AIR}, a module $M$
in a module category $\mo A$ is said to be $\tau_A$-\textit{rigid} if $\Hom_A(M,\tau_AM)=0$. We note that every $\tau_A$-rigid module in $\mo A$ is
rigid and all directing modules in $\ind A$ are $\tau_A$-rigid. The following fact is a direct consequence of the part (iii) of Theorem \ref{thm5}.

\begin{cor} \label{cor16}
Let $A$ be an algebra. Then the number of isomorphism classes of cycle-finite $\tau_A$-rigid modules in $\ind A$ is finite.
\end{cor}

Following Auslander and Reiten \cite{AR}, one associates with each nonprojective module $X$ in a module category $\ind A$ the number $\alpha(X)$
of indecomposable direct summands in the middle term
\[ 0\to \tau_AX\to Y\to X\to 0 \]
of the almost split sequence with the right term $X$. It has been proved by Bautista and Brenner \cite{BaBr} that, if $A$ is an algebra of finite
representation type and $X$ a nonprojective module in $\ind A$, then $\alpha(X)\leq 4$, and if $\alpha(X)=4$ then $Y$ admits a projective-injective
indecomposable direct summand $P$, and hence $X=P/\soc(P)$.
In \cite{Li4} Liu proved that the same is true for any indecomposable nonprojective module $X$ lying on an oriented cycle of the Auslander-Reiten
quiver $\Gamma_A$ of any algebra $A$, and consequently for any nonprojective cycle-finite module in $\ind A$.

The following theorem is a direct consequence of Theorems \ref{thm1} and \ref{thm2}, and \cite[Corollary B]{MS1},
and provides more information on almost split sequences of cycle-finite modules.

\begin{thm} \label{thm6}
Let $A$ be an algebra. Then, for all but finitely many isomorphism classes of nonprojective cycle-finite modules $M$
in $\ind A$, we have $\alpha(M)\leq 2$.
\end{thm}

In connection to Theorem \ref{thm6}, we would like to mention that, for a cycle-finite algebra $A$ and a nonprojective module $M$ in $\ind A$, we have
$\alpha(M)\leq 5$, and if $\alpha(M)=5$ then the middle term of the almost split sequence in $\mo A$ with the right term $M$ admits a projective-injective
indecomposable direct summand $P$, and hence $M=P/\soc(P)$ (see \cite[Conjecture 1]{BrBu}, \cite{MPS2} and \cite{PTa}).

The next theorem describe the structure of the module category $\ind A$ of an arbitrary cycle-finite algebra $A$, and
is a direct consequence of Theorems \ref{thm1} and \ref{thm2} as well as \cite[Theorem 2.2]{MPS1} and its dual.

\begin{thm} \label{thm7}
Let $A$ be a cycle-finite algebra. Then there exist tame generalized multicoil algebras $B_1, \ldots, B_p$ and tame
generalized double tilted algebras $B_{p+1}, \ldots, B_q$ which are quotient algebras of $A$ and the following statements hold.
\begin{enumerate}
\renewcommand{\labelenumi}{\rm(\roman{enumi})}
\item $\ind A = \bigcup_{i=1}^q\ind B_i$.
\item All but finitely many isomorphism classes of modules in $\ind A$ belong to $\bigcup_{i=1}^p\ind B_i$.
\item All but finitely many isomorphism classes of nondirecting modules in $\ind A$ belong to generalized multicoils of $\Gamma_{B_1}, \ldots, \Gamma_{B_p}$.
\end{enumerate}
\end{thm}

The next theorem extends the homological characterization of strongly simply connected algebras of polynomial growth established in \cite{PS1} to arbitrary
cycle-finite algebras, and is a direct consequence of Theorem \ref{thm4} and the properties of directing modules described in \cite[2.4(8)]{Ri1}.

\begin{thm} \label{thm8}
Let $A$ be a cycle-finite algebra. Then, for all but finitely many isomorphism classes of modules $M$ in $\ind A$,
we have $|\Ext_A^1(M,M)|\leq |\End_A(M)|$ and $\Ext_A^r(M,M)$ $=0$ for $r\geq 2$.
\end{thm}

Recently, the problem of describing all algebras $A$ having only finitely many isomorphism classes of $\tau_A$-rigid modules in $\ind A$ was raised to be
important. Namely, it has been proved by Iyama, Reiten, Thomas and Todorov (talk by Reiten at the conference in Nagoya) that the functorially finite torsion
classes in a module category $\mo A$ form a complete lattice if and only if there is only a finite number of pairwise nonisomorphic $\tau_A$-rigid
modules in $\ind A$. Recall that a \textit{torsion class} in $\mo A$ is a full subcategory $\mathcal T$ of $\mo A$ which is closed under
factor modules and extensions, and $\mathcal T$ is \textit{functorially finite} if ${\mathcal T}=\Fac(M)$ for a module $M$ in $\mo A$,
where $\Fac(M)$ is the category of all factor modules of finite direct sums of copies of $M$.
Furthermore, by a recent result of Iyama and Jasso (talk by Jasso at the conference in Nagoya), the finiteness of the
isomorphism classes of $\tau_A$-rigid modules in a module category $\ind A$ implies that the geometric realization of the simplicial complex of support
$\tau_A$-tilting pairs is homeomorphic to the $(n-1)$-dimensional sphere, where $n$ is the rank of the $K_0(A)$.

The following theorem provides a complete characterization of cycle-finite algebras $A$ having only finitely many isomorphism classes
of $\tau_A$-rigid indecomposable modules.

\begin{thm} \label{thm10}
Let $A$ be a cycle-finite algebra. The following statements are equivalent.
\begin{enumerate}
\renewcommand{\labelenumi}{\rm(\roman{enumi})}
\item The number of isomorphism classes of $\tau_A$-rigid modules in $\ind A$ is finite.
\item The number of isomorphism classes of rigid modules in $\ind A$ is finite.
\item The number of isomorphism classes of directing modules in $\ind A$ is finite.
\item $A$ is of finite representation type.
\end{enumerate}
\end{thm}

Then we have the following consequence of Theorem \ref{thm10}, \cite[Theorem 2.7]{AIR} and a recent result by Iyama and Jasso.

\begin{cor} \label{corr11}
Let $A$ be a cycle-finite algebra. The following statements are equivalent.
\begin{enumerate}
\renewcommand{\labelenumi}{\rm(\roman{enumi})}
\item The number of functorially finite torsion classes in $\mo A$ is finite.
\item Every torsion class in $\mo A$ is functorially finite.
\item $A$ is of finite representation type.
\end{enumerate}
\end{cor}

We mention also that, by a recent result of Mizuno \cite[Theorem 2.21]{Mi}, for any preprojective algebra $P(\Delta)$ of Dynkin type
$\Delta\in\{{\Bbb A}_n, {\Bbb D}_n, {\Bbb E}_6, {\Bbb E}_7, {\Bbb E}_8\}$, the category $\ind P(\Delta)$ has only finitely many pairwise nonisomorphic
$\tau_{P(\Delta)}$-rigid modules. Moreover, it is well known that the algebras $P(\Delta)$ are of infinite representation type, except $\Delta={\Bbb A}_n$
with $n\leq 4$. Therefore, there are algebras of infinite representation type having only finitely many isomorphism classes of $\tau$-rigid
indecomposable modules.

We end this section with some questions related to the results described above.

In \cite{Li1}, \cite{Li2} Liu introduced the notions of left and right degrees of irreducible homomorphisms of modules and showed their importance
for describing the shapes of the components of the Auslander-Reiten quivers of algebras. In particular, Liu pointed out in \cite{Li1} that every cycle
of irreducible homomorphisms between indecomposable modules in a module category $\mo A$ contains an irreducible homomorphism of finite left degree and
an irreducible homomorphism of finite right degree. It would be interesting to describe the degrees of irreducible homomorphisms occurring in cycles
of cycle-finite modules (see \cite{C}, \cite{CMS}, \cite{CMT}, \cite{CT} for some results in this direction).

In \cite{Pe} de la Pe\~na proved that the support algebra of a directing module over a tame algebra over an algebraically closed field is a tilted algebra
being a gluing of at most two representation-infinite tilted algebras of Euclidean type. It would be interesting to know if the support algebra $\su(\Gamma)$
of a cycle-finite finite component $\Gamma$ in the cyclic quiver $_c\Gamma_A$ of a cycle-finite algebra is a gluing of at most two representation-infinite
tilted algebras of Euclidean type. In general, it is not clear how many tilted algebras may occur in the decompositions of the left tilted algebra and
the right tilted algebra of the support algebra $\su(\Gamma)$ of a cycle-finite component $\Gamma$ of the cyclic quiver $_c\Gamma_A$ of an algebra $A$
(see Examples \ref{ex6-1} and \ref{ex6-2}).
\section{Cyclic components}
\label{cyccomp}

In this section we recall some concepts and describe some properties of cyclic components of the Auslander-Reiten quivers of algebras.
Let $A$ be an algebra (basic, indecomposable) and $e_{1},\dots,e_{n}$ be a set of pairwise orthogonal primitive idempotents of $A$ with $1_{A}=e_{1}+\dots+e_{n}$.
Then
\begin{itemize}
\item $P_{i}=e_{i}A$, $i\in\{1,\dots,n\}$, is a complete set of pairwise nonisomorphic indecomposable projective modules in $\mo A$;
\item $I_{i}=D(Ae_{i})$, $i\in\{1,\dots,n\}$, is a complete set of pairwise nonisomorphic indecomposable injective modules in $\mo A$;
\item $S_{i}=\topp(P_{i})=e_{i}A/e_{i} \rad A$, $i\in\{1,\dots,n\}$, is a complete set of pairwise nonisomorphic simple modules in $\mo A$;
\item $S_{i}=\soc(I_{i})$, for any $i\in\{1,\dots,n\}$.
\end{itemize}
Moreover, $F_{i}=\End_{A}(S_{i})\cong e_{i}Ae_{i}/e_{i}(\rad A)e_{i}$, for $i\in\{1,\dots,n\}$, are division algebras.
The \textit{quiver} $Q_{A}$ of $A$ is the valued quiver defined as follows:
\begin{itemize}
\item the vertices of $Q_{A}$ are the indices $1,\dots,n$ of the chosen set $e_{1},\dots,e_{n}$ of primitive idempotents of $A$;
\item for two vertices $i$ and $j$ in $Q_{A}$, there is an arrow $i\rightarrow j$ from $i$ to $j$ in $Q_{A}$ if and only if $e_{i}(\rad A)e_{j}/e_{i}(\rad A)^{2}e_{j}\neq 0$.
Moreover, one associates to an arrow $i\rightarrow j$ in $Q_{A}$ the valuation $(d_{ij},d_{ij}')$, so we have in $Q_{A}$ the valued arrow
$$
i \buildrel {\left(d_{ij},d_{ij}'\right)}\over {\hbox to 16mm{\rightarrowfill}} j,
$$
with the valuation numbers are $d_{ij}=\dim_{F_{j}} e_{i}(\rad A)e_{j}/e_{i}(\rad A)^{2}e_{j}$ and

$d_{ij}'= \dim_{F_{i}} e_{i}(\rad A)e_{j}/e_{i}(\rad A)^{2}e_{j}$.
\end{itemize}
It is known that $Q_{A}$ coincides with the $\Ext$-quiver of $A$. Namely, $Q_A$ contains a valued arrow
$ i \buildrel {\left(d_{ij},d_{ij}'\right)}\over {\hbox to 16mm{\rightarrowfill}} j$
iff $\Ext_A^{1}(S_{i},S_{j})\neq 0$ and $d_{ij}=\dim_{F_{j}}\Ext_{A}^{1}(S_{i},S_{j})$, $d_{ij}'=\dim_{F_{i}}\Ext_{A}^{1}(S_{i},S_{j})$.
An algebra $A$ is called \textit{triangular} provided its quiver $Q_{A}$ is acyclic (there is no oriented cycle in $Q_{A}$).
We shall identify an algebra $A$ with the associated category $A^{*}$ whose objects are the vertices $1,\dots,n$ of $Q_{A}$, $\Hom_{A^{*}}(i,j)=e_{j}Ae_{i}$
for any objects $i$ and $j$ of $A^{*}$, and the composition of morphisms in $A^{*}$ is given by the multiplication in $A$.
For a module $M$ in $\mo A$, we denote by $\supp(M)$ the full subcategory of $A=A^{*}$ given by all objects $i$ such that $Me_{i}\neq 0$, and call the \textit{support} of $M$.
More generally, for a translation subquiver $\mathcal{C}$ of $\Gamma_{A}$, we denote by $\supp(\mathcal{C})$ the full subcategory of $A$
given by all objects $i$ such that $Xe_{i}\neq 0$ for some indecomposable module $X$ in $\mathcal{C}$, and call it the \textit{support} of $\mathcal{C}$.
Then a module $M$ in $\mo A$ (respectively, a family of components $\mathcal{C}$ in $\Gamma_{A}$) is said to be \textit{sincere} if $\supp(M)=A$
(respectively, if $\supp(\mathcal{C})=A$). Finally, a full subcategory $B$ of $A$ is said to be a \textit{convex subcategory} of $A$ if every path in $Q_{A}$
with source and target in $B$ has all vertices in $B$. Observe that, for a convex subcategory $B$ of $A$, there is a fully faithful embedding of $\mo B$ into $\mo A$
such that $\mo B$ is the full subcategory of $\mo A$ consisting of the modules $M$ with $Me_{i}=0$ for all objects $i$ of $A$ which are not objects of $B$.

An essential role in further considerations will be played by the following result proved in \cite[Proposition 5.1]{MS1}.

\begin{prop} \label{prop21}
Let $A$ be an algebra and $X, Y$ be modules in $\ind A$. Then $X$ and $Y$ belong to the same component of $_c\Gamma_A$ if and only if there is an oriented
cycle in $\Gamma_A$ passing through $X$ and $Y$.
\end{prop}

We prove now the following property of cycle-finite cyclic components.

\begin{prop} \label{prop22}
Let $A$ be an algebra and $\Gamma$ be a cycle-finite component of $_c\Gamma_A$. Then $\supp(\Gamma)$ is a convex subcategory of $A$.
\end{prop}
\begin{proof}
Let $C = \supp(\Gamma)$. Assume to the contrary that $C$ is not a convex subcategory of $A$. Then $Q_A$ contains a path
\[
 i = i_0 \buildrel {\left(d_{i_0 i_1},d_{i_0 i_1}'\right)}\over {\hbox to 20mm{\rightarrowfill}} i_1
 \buildrel {\left(d_{i_1 i_2},d_{i_1 i_2}'\right)}\over {\hbox to 20mm{\rightarrowfill}} i_2
    \to \cdots \to i_{s-1}
    \buildrel {\left(d_{i_{s-1} i_s},d_{i_{s-1} i_s}'\right)}\over {\hbox to 26mm{\rightarrowfill}} i_s = j,
\]
with $s \geqslant 2$, $i,j$ in $C$ and $i_1,\dots,i_{s-1}$ not in $C$. Since $Q_A$ coincides with the $\Ext$-quiver of $A$, we have
$\Ext_A^1(S_{i_{t-1}}, S_{i_t}) \neq 0$ for $t \in \{ 1,\dots,s \}$. Then there exist in $\mo A$ nonsplitable exact sequences
\[
  0 \to
  S_{i_t} \to
  L_t \to
  S_{i_{t-1}} \to
  0,
\]
for $t \in \{ 1,\dots,s \}$. Clearly, $L_1, \dots, L_s$ are indecomposable modules in $\mo A$ of length 2. In particular, we obtain nonzero nonisomorphisms
$f_r : L_r \to L_{r-1}$ with $\Imi f_r = S_{i_{r-1}}$, for $r \in \{ 2,\dots,s \}$. Consider now the ideal $J$ in $A$ of the form
\[
  J = A e_i(\rad A) e_{i_1} (\rad A)
      + (\rad A) e_{i_{s-1}} (\rad A) e_j A
\]
and the quotient algebra $B = A/J$. Since $i_1$ and $i_{s-1}$ do not belong to $C = \supp(\Gamma)$, for any module $M$ in $\Gamma$,
we have $M e_{i_1} = 0$ and $M e_{i_{s-1}} = 0$, and consequently $M J = 0$. This shows that $\Gamma$ is a cyclic component of $\Gamma_B$.
Moreover, it follows from the definition of $J$ that $S_{i_1}$ is a direct summand of the radical $\rad P_i^{*}$ of the projective cover $P_i^{*} = e_i B$
of $S_i$ in $\mo B$ and $S_{i_{s-1}}$ is a direct summand of the socle factor $I_j^{*}/S_j$ of the injective envelope $I_j^{*} = D(B e_j)$ of $S_j$ in $\mo B$.
Further, since $i$ and $j$ are in $C$, there exist indecomposable modules $X$ and $Y$ in $\Gamma$ such that $S_i$ is a composition factor of $X$
and $S_j$ is a composition factor of $Y$. Then we infer that $\Hom_B(P_i^{*},X) \neq 0$ and $\Hom_B(Y,I_j^{*}) \neq 0$, because $\Gamma$ consists of $C$-modules,
and hence $B$-modules. It follows from Proposition \ref{prop21} that we have in $\Gamma$ a path from $X$ to $Y$. Therefore, we obtain in $\ind A$ a cycle
of the form
\[
 X \to \cdots
   \to Y
   \to I_j^{*}
   \to S_{i_{s-1}}
   \to L_{s-1}
   \to \cdots
   \to L_2
   \to S_{i_1}
   \to P_i^{*}
   \to X ,
\]
which is an infinite cycle, because $X$ and $Y$ belong to $\Gamma$ but $S_{i_1}$ and $S_{i_{s-1}}$ are not in $\Gamma$.
This contradicts the cycle-finiteness of $\Gamma$. Hence $C = \supp(\Gamma)$ is indeed a convex subcategory of $A$.
\end{proof}

Let $A$ be an algebra, $\Gamma$ be a component of $_c\Gamma_A$, and $A=P_{\Gamma}\oplus Q_{\Gamma}$ a decomposition of $A$ in $\ind A$
such that the simple summands of $P_{\Gamma}/\rad P_{\Gamma}$ are exactly the simple composition factors of the indecomposable modules in $\Gamma$.
Then there exists an idempotent $e_{\Gamma}$ of $A$ such that $P_{\Gamma}=e_{\Gamma}A$, $Q_{\Gamma}=(1-e_{\Gamma})A$, $t_A({\Gamma})=A(1-e_{\Gamma})A$,
and $e_{\Gamma}Ae_{\Gamma}$ is isomorphic to the endomorphism algebra $\End_A(P_{\Gamma})$. In follows from Proposition \ref{prop22} that $e_{\Gamma}$
is a convex idempotent of $A$. Observe also that $\End_A(P_{\Gamma})$ is the algebra of the support category $\supp(\Gamma)$ of $\Gamma$. The next result gives
another description of $\End_A(P_{\Gamma})$ in case the component $\Gamma$ of $_c\Gamma_A$ is cycle-finite.

\begin{prop} \label{prop23nn}
Let $A$ be an algebra and $\Gamma$ be a cycle-finite component of $_c\Gamma_A$. Consider a decomposition $A=P_{\Gamma}\oplus Q_{\Gamma}$ of $A$ in $\mo A$
such that the simple summands of $P_{\Gamma}/\rad P_{\Gamma}$ are exactly the simple composition factors of the indecomposable modules in $\Gamma$.
Then the algebras $\su(\Gamma)$ and $\End_A(P_{\Gamma})$ are isomorphic.
\end{prop}
\begin{proof}
Observe that the support algebra $\su(\Gamma)=A/t_A(\Gamma)$ is isomorphic to the endomorphism algebra $\End_A(P_{\Gamma}/P_{\Gamma}t_A(\Gamma))$.
Moreover, $P_{\Gamma}t_A(\Gamma)$ is the right $A$-submodule of $P_{\Gamma}$ generated by the images of all homomorphisms from $Q_{\Gamma}$ to $P_{\Gamma}$
in $\mo A$. For any homomorphism $f\in\End_A(P_{\Gamma})$ we have the canonical commutative diagram in $\mo A$ of the form
\[
\xymatrix{
0\ar[r]&P_{\Gamma}t_A(\Gamma)\ar[r]\ar[d]^{f'}&P_{\Gamma}\ar[r]\ar[d]^{f}&P_{\Gamma}/P_{\Gamma}t_A(\Gamma)\ar[r]\ar[d]^{\overline{f}}&0\cr
0\ar[r]&P_{\Gamma}t_A(\Gamma)\ar[r]&P_{\Gamma}\ar[r]&P_{\Gamma}/P_{\Gamma}t_A(\Gamma)\ar[r]&0,\cr
}
\]
where $f'$ is the restriction of $f$ to $P_{\Gamma}t_A(\Gamma)$ and $\overline{f}$ is induced by $f$. Clearly, by the projectivity of $P_{\Gamma}$ in $\mo A$,
every homomorphism $g\in \End_A(P_{\Gamma}/P_{\Gamma}t_A(\Gamma))$ is of the form $\overline{f}$ for some homomorphism $f\in\End_A(P_{\Gamma})$.
This shows that the assignment $f\to\overline{f}$ induces an epimorphism $\End_A(P_{\Gamma})\to\End_A(P_{\Gamma}/P_{\Gamma}t_A(\Gamma))$ of algebras.
Assume now that $\overline{f}=0$ for a homomorphism $f\in\End_A(P_{\Gamma})$. Then $\Imi f\subseteq P_{\Gamma}t_A(\Gamma)$. On the other hand, it follows from the
definition of $t_A(\Gamma)$ that there is an epimorphism $v: Q_{\Gamma}^m\to P_{\Gamma}t_A(\Gamma)$ in $\mo A$ for some positive integer $m$. Using the
projectivity of $P_{\Gamma}$ in $\mo A$, we conclude that there is a homomorphism $u: P_{\Gamma}\to Q_{\Gamma}^m$ such that $f=vu$. But $f\neq 0$ implies
that $u\neq 0$ and $v\neq 0$, and then a contradiction with the convexity of $\su(\Gamma)$ in $A=A^*$ established in Proposition \ref{prop22}. Hence $f=0$.
Therefore, the canonical epimorphism of algebras $\End_A(P_{\Gamma})\to\End_A(P_{\Gamma}/P_{\Gamma}t_A(\Gamma))$ is an isomorphism, and so the algebras
$\End_A(P_{\Gamma})$ and $\su(\Gamma)$ are isomorphic.
\end{proof}

The following fact proved by Bautista and Smal{\o} in \cite{BaSm} (see also \cite[Corollary III. 11.3]{SY}) will be essential for our considerations.
\begin{prop} \label{prop23n}
Let $A$ be an algebra and
$$X=X_0\to X_1\to \cdots\to X_{r-1}\to X_r=X$$
a cycle in $\Gamma_A$. Then there exists $i\in \{2, \ldots, r\}$ such that $\tau_AX_i\cong X_{i-2}$.
\end{prop}
\begin{lem} \label{lem24}
Let $A$ be an algebra and $\Gamma$ be a cyclic component of $\Gamma_A$. Assume that
$$X=X_0\to X_1\to \cdots\to X_{r-1}\to X_r=X$$
is a cycle in $\Gamma$. Then the following statements hold.
\begin{enumerate}
\renewcommand{\labelenumi}{\rm(\roman{enumi})}
\item If all modules $X_i$, $i\in \{1, \ldots, r\}$, are nonprojective, then $\Gamma$ contains a cycle of the form
$$\tau_AX=\tau_AX_0\to \tau_AX_1\to \cdots\to \tau_AX_{r-1}\to \tau_AX_r=\tau_AX.$$
\item If all modules $X_i$, $i\in \{1, \ldots, r\}$, are noninjective, then $\Gamma$ contains a cycle of the form
$$\tau_A^{-1}X=\tau_A^{-1}X_0\to \tau_A^{-1}X_1\to \cdots\to \tau_A^{-1}X_{r-1}\to \tau_A^{-1}X_r=\tau_A^{-1}X.$$
\end{enumerate}
\end{lem}
\begin{proof}
It follows from Proposition \ref{prop23n} that there exists $i\in \{2, \ldots, r\}$ such that $\tau_AX_i$ $= X_{i-2}$, or equivalently,
$X_i = \tau_A^{-1}X_{i-2}$. Hence, if all modules $X_i$, $i\in \{1, \ldots, r\}$, are nonprojective, then we have in $\Gamma_A$ a cycle
$$\tau_AX=\tau_AX_0\to \tau_AX_1\to \cdots\to \tau_AX_{i} \to\cdots\to \tau_AX_{r-1}\to \tau_AX_r=\tau_AX$$
with $\tau_AX_i = X_{i-2}$, and hence all modules of this cycle belong to the cyclic component $\Gamma$ containing $X_{i-2}$.
Similarly, if all modules $X_i$, $i\in \{1, \ldots, r\}$, are noninjective, then we have in $\Gamma_A$ a cycle
$$\tau_A^{-1}X=\tau_A^{-1}X_0\to \tau_A^{-1}X_1 \to\cdots\to \tau_A^{-1}X_{i-2} \to\cdots\to \tau_A^{-1}X_{r-1}\to \tau_A^{-1}X_r=\tau_A^{-1}X$$
with $X_i = \tau_A^{-1}X_{i-2}$, and hence all modules of this cycle belong to the cyclic component $\Gamma$ containing $X_{i}$.
\end{proof}
\begin{cor} \label{cor25}
Let $A$ be an algebra and $\Gamma$ a finite cyclic component of $\Gamma_A$. Then $\Gamma$ contains a projective and an injective module.
\end{cor}
\begin{proof}
Assume $\Gamma$ does not contain a projective module. Then it follows from Lemma \ref{lem24} that, for any indecomposable module $X$ in $\Gamma$, $\tau_AX$
is also a module in $\Gamma$. Since $\Gamma$ is a finite translation quiver, this implies that $\Gamma = \tau_A\Gamma$, and hence $\Gamma$ is a component
of $\Gamma_A$. Then there exists an indecomposable algebra $B$ (a block of $A$) such that $\Gamma$ is a component of $\Gamma_B$, and consequently
$\Gamma = \Gamma_B$, by the well known theorem of Auslander (see \cite[Theorem III. 10.2]{SY}). But this is a contradiction, because $\Gamma_B$ contains
projective modules. Therefore, $\Gamma$ contains a projective module. The proof that $\Gamma$ contains an injective module is similar.
\end{proof}

Let $A$ be an algebra and $\mathcal C$ a component of $\Gamma_A$. We denote by ${}_l\mathcal{C}$ the \textit{left stable part} of $\mathcal C$ obtained by removing in $\mathcal C$
the $\tau_A$-orbits of projective modules and the arrows attached to them, and by ${}_r\mathcal{C}$ the \textit{right stable part} of $\mathcal C$ obtained
by removing in $\mathcal C$ the $\tau_A$-orbits of injective modules and the arrows attached to them. We note that, if $\mathcal C$ is infinite, then ${}_l\mathcal{C}$
or ${}_r\mathcal{C}$ is nonempty.

The following proposition will be applied in the proofs of our main theorems.
\begin{prop} \label{prop23}
Let $A$ be an algebra, $\mathcal C$ a component of $\Gamma_A$, and $\Sigma$ an infinite family of cycle-finite modules in $\mathcal C$.
Then one of the following statements hold.
\begin{enumerate}
\renewcommand{\labelenumi}{\rm(\roman{enumi})}
\item The stable part ${}_s\mathcal{C}$ of $\mathcal{C}$ contains a stable tube $\mathcal{D}$ having infinitely many modules from $\Sigma$.
\item The left stable part ${}_l\mathcal{C}$ of $\mathcal{C}$ contains a component $\mathcal{D}$ with an oriented cycle and an injective module
such that the cyclic part ${}_c\mathcal{D}$ of $\mathcal{D}$ contains infinitely many modules from $\Sigma$.
\item The right stable part ${}_r\mathcal{C}$ of $\mathcal{C}$ contains a component $\mathcal{D}$ with an oriented cycle and a projective module
such that the cyclic part ${}_c\mathcal{D}$ of $\mathcal{D}$ contains infinitely many modules from $\Sigma$.
\end{enumerate}
\end{prop}
\begin{proof}
(1) Assume first that there is a $\tau_A$-orbit $\mathcal O$ in $\mathcal C$ containing infinitely many modules from $\Sigma$. Consider the case
when $\mathcal O$ contains infinitely many left stable modules from $\Sigma$. Then there exist a module $M$ in ${\mathcal O}\cap\Sigma$ and an infinite
sequence $0=r_0<r_1<r_2<\ldots$ of integers such that the modules $\tau_A^{r_i}M$, $i\in{\Bbb N}$, belong to ${\mathcal O}\cap\Sigma$. Let $\mathcal D$
be the component of ${}_l\mathcal{C}$ containing the modules $\tau_A^{r_i}M$, $i\in{\Bbb N}$. We have two cases to consider.

Assume $\mathcal D$ contains an oriented cycle. Observe that $\mathcal D$ is not a stable tube, and hence does not contain a $\tau_A$-periodic module,
because $\mathcal D$ contains infinitely many modules from the $\tau_A$-orbit $\mathcal O$. Hence, applying \cite[Lemma 2.2 and Theorem 2.3]{Li2}, we
conclude that $\mathcal D$ contains an infinite sectional path
$$\cdots\to\tau_A^tX_s\to\cdots\to\tau_A^tX_2\to\tau_A^tX_1\to X_s\to\cdots\to X_2\to X_1,$$
where $t>s\geq 1$, $X_i$ is an injective module for some $i\in\{1, \ldots, s\}$, and each module in $\mathcal D$ belongs to the $\tau_A$-orbit of one
of the modules $X_i$. Clearly, then there is a nonnegative integer $m$ such that all modules $\tau_A^rM$, $r\geq m$, belong to the cyclic part
${}_c\mathcal{D}$ of $\mathcal{D}$. Therefore, the statement (ii) holds.

Assume $\mathcal{D}$ is acyclic. Then it follows from \cite[Theorem 3.4]{Li2} that there is an acyclic locally finite valued quiver $\Delta$
such that $\mathcal{D}$ is isomorphic to a full translation subquiver of ${\Bbb Z}\Delta$, which is closed under predecessors. But then there exists
a positive integer $i$ such that $\tau_A^{r_i}M$ is not a successor of a projective module in $\mathcal{C}$, and consequently does not lie on an oriented
cycle in $\mathcal{C}$. On the other hand, $\tau_A^{r_i}M$ belongs to $\Sigma$, and then is a cycle-finite indecomposable module, so lying on a cycle in
$\mathcal{C}$, a contradiction.

Similarly, if $\mathcal{O}$ contains infinitely many right stable modules from $\Sigma$, then the statement (iii) holds.

(2) Assume now that every $\tau_A$-orbit in $\mathcal C$ contains at most finitely many modules from $\Sigma$. Since $\Sigma$ is an infinite family
of modules, we infer that there is an infinite component $\mathcal{D}$ of the stable part ${}_s\mathcal{C}$ of $\mathcal{C}$ containing infinitely
many modules from $\Sigma$. We have two cases to consider.

Assume $\mathcal D$ contains an oriented cycle. Then it follows from \cite[Corollary]{Zh} (see also \cite[Theorems 2.5 and 2.7]{Li1}) that $\mathcal D$
is a stable tube. Thus the statement (i) holds.

Assume $\mathcal{D}$ is acyclic. Applying \cite[Corollary]{Zh} again, we conclude that there exists an infinite locally finite acyclic valued quiver
$\Delta$ such that $\mathcal{D}$ is isomorphic to the translation quiver ${\Bbb Z}\Delta$. Let $n$ be the rank of the Grothendieck group $K_0(A)$ of $A$.
Then there is a module $M$ in ${\mathcal D}\cap\Sigma$ such that the length of any walk in ${\mathcal C}$ from a nonstable module in ${\mathcal C}$ to
a module in the $\tau_A$-orbit ${\mathcal O}(M)$ of $M$ is at least $2n$. Then it follows from \cite[Lemma 1.5]{CS} (see also \cite[Lemma 4]{Sk4})
that, for each positive integer $s$, there exists a path
$$M=X_0\to X_1\to \cdots\to X_t=\tau_A^sM$$
in $\ind A$ with all modules $X_i$ in ${\mathcal C}$, and consequently a cycle in $\ind A$ passing through $M$ and $\tau_A^sM$, because there is a path
$$\tau_A^sM=Y_0\to Y_1\to \cdots\to Y_r=M$$
of irreducible homomorphisms in $\ind A$. Moreover, $M$ is a cycle-finite module, as a module from $\Sigma$. This shows that ${\mathcal C}$ contains
oriented cycles passing through $M$ and any module $\tau_A^sM$, $s\geq 1$. We also note that there is a component $\mathcal{D}'$ of the
left stable part ${}_l\mathcal{C}$ of $\mathcal{C}$ containing all $\tau_A$-orbits of $\mathcal{D}$. Then there is an infinite locally finite acyclic
valued subquiver $\Delta'$ containing $\Delta$ as a full valued subquiver, such that $\mathcal{D}'$ is isomorphic to a full translation subquiver of ${\Bbb Z}\Delta'$,
which is closed under predecessors. Then there exists a positive integer $m$ such that the module $\tau_A^mM$ is not a successor of a projective module
in ${\mathcal C}$, and then $\tau_A^mM$ does not lie on an oriented cycle in ${\mathcal C}$, a contradiction.
\end{proof}
\begin{cor} \label{cor28}
Let $A$ be an algebra and $\Gamma$ be a cycle-finite infinite component of $_c\Gamma_A$. Then ${}_l\Gamma$ or ${}_r\Gamma$ admits
a component $\mathcal{D}$ containing an oriented cycle and infinitely many modules of $\Gamma$.
\end{cor}
\section{Proof of Theorem \ref{thm1}}
\label{pfth1}

Let $A$ be an algebra and $\Gamma$ be a cycle-finite infinite component of $_c\Gamma_A$. Consider the component $\mathcal C$ of $\Gamma_A$
containing the translation quiver $\Gamma$. Since $\Gamma$ is infinite and cyclic, we conclude from Corollary \ref{cor28}
that ${}_l\mathcal{C}$ or ${}_r\mathcal{C}$ contains a connected component $\Sigma$ containing an oriented cycle and infinitely many modules of $\Gamma$.
We claim that there exists a cyclic coherent full translation subquiver $\Omega$ of $\Gamma$ containing all modules of the cyclic part ${}_c\Sigma$
of $\Sigma$. We have three cases to consider:
\begin{enumerate}
\renewcommand{\labelenumi}{\rm(\arabic{enumi})}
\item Assume $\Sigma$ is contained in the stable part $_s\mathcal C = {_l\mathcal C} \cap {_r\mathcal C}$ of $\mathcal C$. Then $\Sigma$ is an
infinite stable translation quiver containing an oriented cycle, and hence $\Sigma$ is a stable tube, by the main result of \cite{Zh}. Clearly, the
stable tube $\Sigma$ is a cyclic and coherent translation quiver. Since $\Sigma$ is a component of $_l\mathcal C$ and a component of $_r\mathcal C$,
we conclude that $\Gamma$ contains a cyclic coherent full translation subquiver $\Omega$ such that $\Sigma$ is obtained from $\Omega$ by removing all
finite $\tau_A$-orbits without $\tau_A$-periodic modules.
\item Assume $\Sigma$ is a component of $_l\mathcal C$ containing at least one injective module. Then it follows from \cite[Lemma 2.2 and Theorem 2.3]{Li2} that
$\Sigma$ contains an infinite sectional path
$$\cdots\to\tau_A^rX_s\to\cdots\to\tau_A^rX_2\to\tau_A^rX_1\to X_s\to\cdots\to X_2\to X_1,$$
where $r>s\geq 1$, $X_i$ is an injective module for some $i\in\{1, \ldots, s\}$, and each module in $\Sigma$ belongs to the $\tau_A$-orbit of one
of the modules $X_i$.
Observe that there exists an infinite sectional path in $\Sigma$
$$X_s\to \tau_A^{r-1}X_1\to\cdots$$
starting from $X_s$. Let $p$ be the minimal element in $\{1, \ldots, s\}$ such that there exists an infinite sectional path in $\Sigma$ starting from $X_p$.
Then $\Gamma$ contains a cyclic coherent full translation subquiver $\Omega$ such that $\Sigma$ is obtained from $\Omega$ by removing the $\tau_A$-orbits
of projective modules $P$ lying on infinite sectional paths in $\Sigma$ of the forms
$$P\to\cdots\to X_j\to\cdots\to \tau_A^{r-j+p-1}X_1\to\cdots$$
for some $j\in\{p, \ldots, s\}$, or
$$P\to\cdots\to \tau_A^{mr}X_i\to \tau_A^{mr-1}X_{i+1}\to\cdots$$
for some $m\geq 1$ and $i\in\{1, \ldots, s\}$.
\item Assume $\Sigma$ is a component of $_r\mathcal C$ containing at least one projective module. Then it follows from \cite[duals of Lemma 2.2 and Theorem 2.3]{Li2} that
$\Sigma$ contains an infinite sectional path
$$X_1\to X_2\to\cdots\to X_t\to\tau_A^{-m}X_1\to\tau_A^{-m}X_2\to\cdots\to\tau_A^{-m}X_t\to\cdots,$$
where $m>t\geq 1$, $X_j$ is an projective module for some $j\in\{1, \ldots, t\}$, and each module in $\Sigma$ belongs to the $\tau_A$-orbit of one
of the modules $X_j$.
Observe that there exists an infinite sectional path in $\Sigma$
$$\cdots\to \tau_A^{-m+1}X_1\to X_t$$
ending in $X_t$. Let $q$ be the minimal element in $\{1, \ldots, t\}$ such that there exists an infinite sectional path in $\Sigma$ ending in $X_q$.
Then $\Gamma$ contains a cyclic coherent full translation subquiver $\Omega$ such that $\Sigma$ is obtained from $\Omega$ by removing the $\tau_A$-orbits
of injective modules $I$ lying on infinite sectional paths in $\Sigma$ of the forms
$$\cdots\to\tau_A^{-m+t-i+1}X_1\to\cdots\to X_i\to\cdots\to I$$
for some $i\in\{q, \ldots, t\}$, or
$$\cdots\to\tau_A^{-ms+1}X_{j+1}\to \tau_A^{-ms}X_{j}\to\cdots\to I$$
for some $s\geq 1$ and $j\in\{1, \ldots, t\}$.
\end{enumerate}
Let $\Gamma_1, \ldots, \Gamma_t$ be all maximal cyclic coherent pairwise different full translation subquivers of $\Gamma$. Clearly,
$\Gamma_1, \ldots, \Gamma_t$ are pairwise disjoint. For each $i\in\{1, \ldots, t\}$, consider the support algebra $B^{(i)}=\su(\Gamma_i)$ of $\Gamma_i$.

Fix $i\in\{1, \ldots, t\}$. We shall prove that $B^{(i)}$ is a generalized multicoil algebra and $\Gamma_i$ is the cyclic part of a generalized
multicoil $\Gamma_i^*$ of $\Gamma_{B^{(i)}}$, and consequently $\Gamma_i$ is a cyclic generalized multicoil full translation subquiver of $\Gamma_{B^{(i)}}$.
Since $\Gamma_i$ is a cyclic coherent full translation subquiver of the component $\mathcal C$ of $\Gamma_A$ and of $\Gamma$, it follows from the proofs
of Theorems A and F in \cite{MS1} that $\Gamma_i$, considered as a translation quiver, is a generalized multicoil, and consequently can be obtained
from a finite family ${\mathcal T}_1^{(i)}, \ldots, {\mathcal T}_{p_i}^{(i)}$ of stable tubes by an iterated application of admissible operations
of types (ad~1)-(ad~5) and their duals (ad~1$^*$)-(ad~5$^*$). We note that all vertices of the stable tubes
${\mathcal T}_1^{(i)}, \ldots, {\mathcal T}_{p_i}^{(i)}$ are indecomposable modules of $\Gamma$, and the stable tubes
${\mathcal T}_1^{(i)}, \ldots, {\mathcal T}_{p_i}^{(i)}$ can be obtained from $\Gamma$ by removing the modules of
$\Gamma\setminus({\mathcal T}_1^{(i)}\cup \ldots\cup {\mathcal T}_{p_i}^{(i)})$ and shrinking the corresponding sectional paths in $\Gamma$ with the
ends at the modules in ${\mathcal T}_1^{(i)}\cup \ldots\cup {\mathcal T}_{p_i}^{(i)}$ into the arrows. We claim now that $\Gamma_i$ is a generalized
standard full translation subquiver of $\Gamma_A$. Suppose that $\rad^{\infty}(X,Y)\neq 0$ for some indecomposable $A$-modules $X$ and $Y$ lying in
$\Gamma_i$. Then, applying Proposition \ref{prop21}, we conclude that there is in $\ind A$ an infinite cycle
$$X \buildrel {f}\over {\hbox to 6mm{\rightarrowfill}} Y \buildrel {f_1}\over {\hbox to 6mm{\rightarrowfill}} Z_1
\buildrel {f_2}\over {\hbox to 6mm{\rightarrowfill}} Z_2 \cdots \to Z_{t-1}
\buildrel {f_t}\over {\hbox to 6mm{\rightarrowfill}} Z_t=X$$
where $Z_1, \ldots, Z_t=X, Y$ are modules in $\Gamma_i$, $f_1, \ldots, f_t$ are irreducible homomorphisms and $0\neq f\in\rad^{\infty}(X,Y)$,
a contradiction with the cycle-finiteness of $\Gamma$.
Similarly, there is no path in $\ind B^{(i)}$ of the form
$$X \buildrel {g}\over {\hbox to 6mm{\rightarrowfill}} Z \buildrel {h}\over {\hbox to 6mm{\rightarrowfill}} Y$$
with $X$ and $Y$ in $\Gamma_i$ and $Z$ not in $\Gamma_i$ (external short path of $\Gamma_i$ in $\ind B$ in the sense of \cite{RS0}).
Since $\Gamma_i$ is a sincere cyclic coherent full translation subquiver of $\Gamma_{B^{(i)}}$,
applying \cite[Theorem A]{MS2} (and its proof), we conclude that ${B^{(i)}}$ is a generalized multicoil algebra,
$\Gamma_i$ is the cyclic part of a generalized multicoil $\Gamma_i^*$ of $\Gamma_{B^{(i)}}$, and $\ann_{B^{(i)}}(\Gamma_i)=\ann_{B^{(i)}}(\Gamma_i^*)=0$, and hence
$B^{(i)}=B(\Gamma_i)=B(\Gamma_i^*)$.
For each $j\in\{1, \ldots, p_i\}$, consider the quotient
algebra $C_j^{(i)} = A/\ann_A({\mathcal T}_{j}^{(i)})$ of $A$ by the annihilator $\ann_A({\mathcal T}_{j}^{(i)})$ of the family of indecomposable $A$-modules
forming ${\mathcal T}_{j}^{(i)}$. Then $C_j^{(i)}$ is a concealed canonical algebra and ${\mathcal T}_{j}^{(i)}$ is a stable tube of $\Gamma_{C_j^{(i)}}$.
We note that we may have $C_j^{(i)}=C_k^{(i)}$ for $j\neq k$ in $\{1, \ldots, p_i\}$. Then denoting by $C^{(i)}$ the product of pairwise different
algebras in the family $C_1^{(i)}, \ldots, C_{p_i}^{(i)}$, with respect to the annihilators $\ann_A({\mathcal T}_1^{(i)}), \ldots, \ann_A({\mathcal T}_{p_i}^{(i)})$
of ${\mathcal T}_1^{(i)}, \ldots, {\mathcal T}_{p_i}^{(i)}$, we obtain that $B^{(i)}$ is a generalized multicoil enlargement of $C^{(i)}$ involving the stable
tubes ${\mathcal T}_1^{(i)}, \ldots, {\mathcal T}_{p_i}^{(i)}$ and admissible operations of types (ad~1)-(ad~5) and (ad~1$^*$)-(ad~5$^*$) corresponding
to the translation quiver operations leading from the stable tubes ${\mathcal T}_1^{(i)}, \ldots, {\mathcal T}_{p_i}^{(i)}$ to the generalized multicoil
$\Gamma_i^*$. Further, by \cite[Theorem C]{MS2}, we have the following additional properties of $B^{(i)}$:
\begin{enumerate}
\renewcommand{\labelenumi}{\rm(\arabic{enumi})}
\item There is a unique factor algebra (not necessarily connected) $B_l^{(i)}$ of $B^{(i)}$ (the left part of $B^{(i)}$) obtained from $C^{(i)}$ by an
iteration of admissible operations of type (ad~1$^*$) and a family ${\widehat{\mathcal T}}_1^{(i)}, \ldots, {\widehat{\mathcal T}}_{p_i}^{(i)}$
of coray tubes in $\Gamma_{B_l^{(i)}}$, obtained from the stable tubes ${\mathcal T}_1^{(i)}, \ldots, {\mathcal T}_{p_i}^{(i)}$ by the corresponding
coray insertions, such that $B^{(i)}$ is obtained from $B_l^{(i)}$ by an iteration of admissible operations of types (ad~1)-(ad~5) and $\Gamma_i^*$
is obtained from the family ${\widehat{\mathcal T}}_1^{(i)}, \ldots, {\widehat{\mathcal T}}_{p_i}^{(i)}$ by an iteration of admissible operations
of types (ad~1)-(ad~5) corresponding to those leading from $B_l^{(i)}$ to $B^{(i)}$.
\item There is a unique factor algebra (not necessarily connected) $B_r^{(i)}$ of $B^{(i)}$ (the right part of $B^{(i)}$) obtained from $C^{(i)}$ by an
iteration of admissible operations of type (ad~1) and a family ${\widetilde{\mathcal T}}_1^{(i)}, \ldots, {\widetilde{\mathcal T}}_{p_i}^{(i)}$
of ray tubes in $\Gamma_{B_r^{(i)}}$, obtained from the stable tubes ${\mathcal T}_1^{(i)}, \ldots, {\mathcal T}_{p_i}^{(i)}$ by the corresponding
ray insertions, such that $B^{(i)}$ is obtained from $B_r^{(i)}$ by an iteration of admissible operations of types (ad~1$^*$)-(ad~5$^*$) and $\Gamma_i^*$
is obtained from the family ${\widetilde{\mathcal T}}_1^{(i)}, \ldots, {\widetilde{\mathcal T}}_{p_i}^{(i)}$ by an iteration of admissible operations
of types (ad~1$^*$)-(ad~5$^*$) corresponding to those leading from $B_r^{(i)}$ to $B^{(i)}$.
\end{enumerate}
As a consequence, the generalized multicoil $\Gamma_i^*$ of $\Gamma_{B^{(i)}}$ admits a left border $\Delta_l^{(i)}$ and a right border $\Delta_r^{(i)}$
having the following properties:
\begin{enumerate}
\renewcommand{\labelenumi}{\rm(\alph{enumi})}
\item $\Delta_l^{(i)}$ and $\Delta_r^{(i)}$ are disjoint and unions of finite sectional paths of $\Gamma_i$;
\item $\Gamma_i$ is the full translation subquiver of $\Gamma_i^*$ consisting of all modules which are both successors of modules lying in
$\Delta_l^{(i)}$ and predecessors of modules lying in $\Delta_r^{(i)}$;
\item $\Gamma_i^*\setminus\Gamma_i$ consists of a finite number of directing $B^{(i)}$-modules;
\item Every module in $\Gamma\setminus\Gamma_i$ which is a predecessor of a module in $\Gamma_i$ is a predecessor of a module in $\Delta_l^{(i)}$;
\item Every module in $\Gamma\setminus\Gamma_i$ which is a successor of a module in $\Gamma_i$ is a successor of a module in $\Delta_r^{(i)}$;
\item $B(\Delta_l^{(i)})=\su(\Delta_l^{(i)})$ and is a product of tilted algebras of equioriented Dynkin types $\Bbb{A}_n$ and $\Delta_l^{(i)}$ is the union of sections of the
connecting components of the indecomposable parts of $B(\Delta_l^{(i)})$;
\item $B(\Delta_r^{(i)})=\su(\Delta_r^{(i)})$ and is a product of tilted algebras of equioriented Dynkin types $\Bbb{A}_n$ and $\Delta_r^{(i)}$ is the union of sections of the
connecting components of the indecomposable parts of $B(\Delta_r^{(i)})$.
\end{enumerate}
We denote by $\Gamma^{cc}$ the union of the translation subquivers $\Gamma_1, \ldots, \Gamma_t$. We claim that $\Gamma\setminus\Gamma^{cc}$ consists
of finitely modules and $\Gamma^{cc}$ is a maximal cyclic coherent full translation subquiver of $\Gamma$. Suppose that infinitely many modules of
$\Gamma$ are not contained in $\Gamma^{cc}$. We have the following properties of modules in $\Gamma\setminus\Gamma^{cc}$. Since $\Gamma$ is a connected
component of $_c\Gamma_A$, by Proposition \ref{prop21}, for any modules $M$ in $\Gamma\setminus\Gamma^{cc}$ and $N$ in $\Gamma^{cc}$, there is an
oriented cycle in $\Gamma$ passing through $M$ and $N$. Moreover, if $N$ belongs to $\Gamma_i$, then every such a cycle is of the form
$$M\to\cdots\to X\to\cdots\to N\to\cdots\to Y\to\cdots\to M$$
with $X$ in $\Delta_l^{(i)}$ and $Y$ in $\Delta_r^{(i)}$.
Applying Proposition \ref{prop23} to the infinite family $\Sigma=\Gamma\setminus\Gamma^{cc}$ of cycle-finite modules, we obtain that
the left stable part $_l{\mathcal C}$ or the right stable part $_r{\mathcal C}$ of $\mathcal C$ admits an infinite component $\Sigma'$ containing
an oriented cycle and infinitely many modules from $\Gamma\setminus\Gamma^{cc}$. Then, as in the first part of the proof, we infer that there exists
a cyclic coherent full translation subquiver $\Omega'$ of $\Gamma$ containing all modules of $\Sigma'$. Obviously, $\Omega'$ is disjoint with
$\Gamma_1, \ldots, \Gamma_t$, and this contradicts to our choice of $\Gamma_1, \ldots, \Gamma_t$. Therefore, indeed, $\Gamma\setminus\Gamma^{cc}$
consists of finitely many modules.

Our next aim is to show that the algebra $B(\Gamma\setminus\Gamma^{cc})=A/\ann_A(\Gamma\setminus\Gamma^{cc})$ is of finite representation type. We abbreviate
$D=B(\Gamma\setminus\Gamma^{cc})$.
Observe that, if every indecomposable module from $\mo D$ lies in $\Gamma\setminus\Gamma^{cc}$, then $D$ is of finite representation type. Therefore,
assume that $\mo D$ admits an indecomposable module $Z$ which is not in $\Gamma\setminus\Gamma^{cc}$. Let $M$ be the direct sum of all indecomposable
$A$-modules lying in $\Gamma\setminus\Gamma^{cc}$. Moreover, let $D=P'\oplus P''$ be a decomposition of $D$ in $\mo D$, where $P'$ is the direct sum
of all indecomposable projective $D$-modules lying in $\Gamma\setminus\Gamma^{cc}$ and $P''$ is the direct sum of the remaining indecomposable
projective $D$-modules. Observe that $M$ is a faithful module in $\mo D$ and hence we have a monomorphism of right $D$-modules $P''\to M^t$,
which then factors through a direct sum of modules
lying on the sum $\Delta_r^{(1)}\cup\ldots\cup\Delta_r^{(t)}$ of the right parts $\Delta_r^{(1)}, \ldots, \Delta_r^{(t)}$ of $\Gamma_1, \ldots, \Gamma_t$,
and consequently $P''$ is a module over the algebra $B(\Delta_r^{(1)})\times\ldots\times B(\Delta_r^{(t)})$.
Consider also a projective cover $\pi: P_D(Z)\to Z$ of $Z$ in $\mo D$. Let $P_D(Z) = P'_D(Z)\oplus P''_D(Z)$, where
$P'_D(Z)$ is a direct sum of direct summands of $P'$ and $P''_D(Z)$ is a direct sum of direct summands of $P''$, and denote by $\pi': P'_D(Z)\to Z$
and $\pi'': P''_D(Z)\to Z$ the restrictions of $\pi$ to $P'_D(Z)$ and $P''_D(Z)$, respectively.
Then $\pi': P'_D(Z)\to Z$ factors through a direct sum of modules lying on the sum $\Delta_l^{(1)}\cup\ldots\cup\Delta_l^{(t)}$
of the left parts $\Delta_l^{(1)}, \ldots, \Delta_l^{(t)}$ of $\Gamma_1, \ldots, \Gamma_t$, because $Z$ does not belong to $\Gamma\setminus\Gamma^{cc}$.
In particular, we obtain that $\pi'(P'_D(Z))$ is a module  over the algebra $B(\Delta_l^{(1)})\times\ldots\times B(\Delta_l^{(t)})$. Summing up, we
conclude that $Z=\pi'(P'_D(Z))+\pi''(P''_D(Z))$ is a module over the quotient algebra
$$\Lambda = B(\Delta_l^{(1)})\times\ldots\times B(\Delta_l^{(t)})\times B(\Delta_r^{(1)})\times\ldots\times B(\Delta_r^{(t)}),$$
of $A$, which is an algebra of finite representation type as a product of tilted algebras of Dynkin types $\Bbb{A}_n$.
Therefore, we obtain that every module from $\ind D$ which is not in $\Gamma\setminus\Gamma^{cc}$ is an indecomposable module in $\ind\Lambda$.
Since $\Gamma\setminus\Gamma^{cc}$ is finite, we conclude that $D$ is of finite representation type.

Finally, let $B=\su(\Gamma)=A/t_A(\Gamma)$. Then $\Gamma$ is a sincere cycle-finite component of $_c\Gamma_B$ and $\ann_B(\Gamma)=\ann_A(\Gamma)/t_A(\Gamma)$.
Hence, in order to show that $\su(\Gamma)=B(\Gamma)$, it is enough to prove that $\Gamma$ is a faithful translation subquiver of $\Gamma_B$. Let $B=P\oplus Q$
be a decomposition in $\mo B$ such that $Q$ is the direct sum of all indecomposable projective modules lying in $\Gamma$ and $P$ the direct sum of the remaining
indecomposable projective right $B$-modules. Then $P$ is a direct sum of indecomposable projective modules over the product
$B^{(1)}\times\ldots\times B^{(t)}$ of generalized multicoil algebras $B^{(1)}, \ldots, B^{(t)}$. Since $B^{(i)}=\su(\Gamma_i)=B(\Gamma_i)$ for any $i\in\{1, \ldots, t\}$,
we conclude that there is a monomorphism $P\to N^m$ for a module $N$ in $\mo B$ being a direct sum of indecomposable modules lying in
$\Gamma^{cc}=\Gamma_1\cup\ldots\cup\Gamma_t$ and a positive integer $m$. Clearly, then there is a monomorphism in $\mo B$ of the form
$B=P\oplus Q \to (N\oplus Q)^m$, and consequently $\Gamma$ is a faithful component of $\Gamma_B$. Therefore, we obtain the equality $\su(\Gamma)=B(\Gamma)$.
\section{Proof of Theorem \ref{thm2}}
\label{pfth2}

Let $A$ be an algebra and $\Gamma$ be a cycle-finite finite component of $_c\Gamma_A$. Moreover, let $B=A/t_A(\Gamma)$ be the support algebra of $\Gamma$.
Observe that $\Gamma$ is a sincere cycle-finite component of $_c\Gamma_B$. We will show that $B$ is a generalized double tilted algebra, applying \cite[Theorem]{Sk11-5}.
Since $\Gamma$ is a finite component of $_c\Gamma_B$, it follows from Corollary \ref{cor25} that $\Gamma$ contains a projective module and an injective module.
Hence, applying Proposition \ref{prop21}, we conclude that there exists in $\Gamma$ a path from an injective module to a projective module. Let
\[ I=X_0 \buildrel {f_1}\over {\hbox to 10mm{\rightarrowfill}} X_1 \buildrel {f_2}\over {\hbox to 10mm{\rightarrowfill}} \cdots
\buildrel {f_{m-1}}\over {\hbox to 10mm{\rightarrowfill}} X_{m-1} \buildrel {f_m}\over {\hbox to 10mm{\rightarrowfill}} X_m=P \]
be an arbitrary path in $\ind B$ from an indecomposable injective module $I$ to an indecomposable projective module $P$. Since $\Gamma$ is a sincere translation
subquiver of $\Gamma_B$, there exist indecomposable modules $M$ and $N$ in $\Gamma$ such that $\Hom_B(P,M)\neq 0$ and $\Hom_B(N,I)\neq 0$. Further, it follows
from Proposition \ref{prop21} that there exists a path in $\ind B$ from $M$ to $N$. Therefore, we obtain in $\ind B$ a cycle of the form
\[ M\to\cdots\to N\to X_0 \buildrel {f_1}\over {\hbox to 9mm{\rightarrowfill}} X_1 \buildrel {f_2}\over {\hbox to 9mm{\rightarrowfill}} \cdots
\buildrel {f_{m-1}}\over {\hbox to 10mm{\rightarrowfill}} X_{m-1} \buildrel {f_m}\over {\hbox to 9mm{\rightarrowfill}} X_m\to M, \]
and this is a finite cycle, because $M$ and $N$ belong to the cycle-finite component $\Gamma$ of $_c\Gamma_B$. This shows that all the modules
$X_0, X_1, \ldots, X_{m-1}, X_m$
belong to the finite translation quiver $\Gamma$ of $\Gamma_B$. Then it follows from \cite[Theorem]{Sk11-5} that $B$ is a quasitilted algebra or a generalized
double tilted algebra. Furthermore, by \cite[Corollary (E)]{CS}, the Auslander-Reiten quiver of a quasitilted algebra which is not a tilted algebra consists
of semiregular components. Clearly, every tilted algebra is a generalized double tilted algebra \cite{RS2}. Since the cyclic component $\Gamma$ of $\Gamma_B$
contains a path from an injective module to a projective module, we obtain that $B$ is a generalized double tilted algebra. Hence, it follows from
\cite[Section 3]{RS2} that $\Gamma_B$ admits an almost acyclic component $\mathcal C$ with a faithful multisection $\Delta$. Recall that, following \cite[Section 2]{RS2},
a full connected subquiver $\Delta$ of $\mathcal C$ is called a {\it multisection} if the following conditions are satisfied:
\begin{enumerate}
\renewcommand{\labelenumi}{\rm(\roman{enumi})}
\item $\Delta$ is almost acyclic.
\item $\Delta$ is convex in $\mathcal C$.
\item For each $\tau_B$-orbit $\mathcal O$ in $\mathcal C$, we have $1 \leq | \Delta \cap {\mathcal O} | < \infty$.
\item $| \Delta \cap {\mathcal O} | = 1$ for all but finitely many $\tau_B$-orbits $\mathcal O$ in $\mathcal C$.
\item No proper full convex subquiver of $\Delta$ satisfies \rm{(i)}--\rm{(iv)}.
\end{enumerate}

Moreover, for a multisection $\Delta$ of a component $\mathcal C$, the following full subquivers of $\mathcal C$ were defined in \cite{RS2}:
$$\! \Delta^{\prime}_{l} = \{ X\! \in \!\Delta; \text{there is a nonsectional path in $\mathcal C$ from $X$ to a projective module $P$}\},$$
$$\! \Delta^{\prime}_{r} = \{ X\! \in \!\Delta; \text{there is a nonsectional path in $\mathcal C$ from an injective module $I$ to $X$}\},$$
$$
  \Delta^{\prime\prime}_{l} =
  \{ X \in \Delta^{\prime}_{l};
     \tau_A^{-1} X \notin \Delta^{\prime}_{l}
  \} , \qquad
  \Delta^{\prime\prime}_{r} =
  \{ X \in \Delta^{\prime}_{r};
     \tau_A X \notin \Delta^{\prime}_{r}
  \} ,
$$
$$
  \Delta_l =
  (\Delta \setminus \Delta^{\prime}_r)
       \cup \tau_A \Delta^{\prime\prime}_r
    , \quad
  \Delta_c = \Delta^{\prime}_l \cap \Delta^{\prime}_r
    , \quad
  \Delta_r = (\Delta \setminus \Delta^{\prime}_l)
        \cup \tau_A^{-1} \Delta^{\prime\prime}_l
  .
$$

Then $\Delta_l$ is called the {\it left part} of $\Delta$, $\Delta_r$ the {\it right part} of $\Delta$, and $\Delta_c$ the {\it core} of $\Delta$.

The following basic properties of $\Delta$ have been established in \cite[Proposition 2.4]{RS2}:
\begin{enumerate}
\renewcommand{\labelenumi}{\rm(\alph{enumi})}
\item Every cycle of $\mathcal C$ lies in $\Delta_c$.
\item $\Delta_c$ is finite.
\item Every indecomposable module $X$ in $\mathcal C$ is in $\Delta_c$, or a predecessor of $\Delta_l$ or a successor of $\Delta_r$ in $\mathcal C$.
\end{enumerate}
It follows also from \cite[Theorem 3.4, Corollary 3.5]{RS2} and the known structure of the Auslander-Reiten quivers of tilted algebras
(see \cite{HR}, \cite{Ke}, \cite{Ri1}, \cite{SS2}) that every component of $\Gamma_B$ different from $\mathcal C$ is a semiregular component.
Hence the cyclic component $\Gamma$ is a translation subquiver of $\mathcal C$, and consequently is contained in the core $\Delta_c$ of $\Delta$.
We also know from \cite[Proposition 2.11]{RS2} that, for another multisection $\Sigma$ of $\mathcal C$, we have $\Sigma_c=\Delta_c$. Thus $\Delta_c$
is a uniquely defined core $c(\mathcal C)$ of the connecting component $\mathcal C$ of $\Gamma_B$. We claim that $\Gamma=c({\mathcal C})$.
Let $X$ be a module in $\Delta_c=\Delta'_l\cap\Delta'_r$. Then there are nonsectional paths in $\mathcal C$ from $X$ to an indecomposable projective module
$P$ and from an indecomposable injective module $I$ to $X$. Moreover, there exist indecomposable modules $Y$ and $Z$ in $\Gamma$ such that
$\Hom_B(P,Y)\neq 0$ and $\Hom_B(Z,I)\neq 0$, because $\Gamma$ is a sincere translation subquiver of $\Gamma_B$. Further, by Proposition \ref{prop21},
we have in $\Gamma$ a path from $Y$ to $Z$. Hence we obtain in $\ind B$ a cycle of the form
\[ X\to\cdots\to P\to Y\to\cdots\to Z\to I\to\cdots\to X, \]
which is a finite cycle because $Y$ and $Z$ belong to the cycle-finite component $\Gamma$ of $_c\Gamma_B$. Therefore, there is in $\mathcal C$ a cycle passing through
the modules $X$, $Y$ and $Z$, and so $X$ belongs to $\Gamma$. This shows that $\Gamma=\Delta_c=c({\mathcal C})$.

Let $B^{(l)}=\su(\Delta_l)$ be the support algebra of the left part $\Delta_l$ of $\Delta$ (if $\Delta_l$ is nonempty) and $B^{(r)}=\su(\Delta_r)$
be the support algebra of the right part $\Delta_r$ of $\Delta$ (if $\Delta_r$ is nonempty). Then the following description of $\ind B$ follows from the results
established in \cite[Section 3]{RS2}:
\begin{enumerate}
\renewcommand{\labelenumi}{\rm(\arabic{enumi})}
\item $B^{(l)}$ is a tilted algebra (not necessarily indecomposable) such that $\Delta_l$ is a disjoint union of sections
of the connecting components of the indecomposable parts of $B^{(l)}$ and the category of all predecessors of $\Delta_l$ in $\ind B$ coincides with
the category of all predecessors of $\Delta_l$ in $\ind B^{(l)}$, or $B^{(l)}$ is empty in case $\Delta_l$ is empty.
\item $B^{(r)}$ is a tilted algebra (not necessarily indecomposable) such that $\Delta_r$ is a disjoint union of sections
of the connecting components of the indecomposable parts of $B^{(r)}$ and the category of all successors of $\Delta_r$ in $\ind B$ coincides with
the category of all successors of $\Delta_r$ in $\ind B^{(r)}$, or $B^{(r)}$ is empty in case $\Delta_r$ is empty.
\item Every indecomposable module in $\ind B$ is either in $\Gamma=c({\mathcal C})$, a predecessor of $\Delta_l$ in $\ind B$, or a successor of $\Delta_r$ in $\ind B$.
\item If $\Delta_l$ is nonempty, then $\Delta_l$ is a faithful subquiver of $\Gamma_{B^{(l)}}$, and hence $B^{(l)}$ is the faithful algebra
$B(\Delta_l)=A/\ann_A(\Delta_l)$ of $\Delta_l$.
\item If $\Delta_r$ is nonempty, then $\Delta_r$ is a faithful subquiver of $\Gamma_{B^{(r)}}$, and hence $B^{(r)}$ is the faithful algebra
$B(\Delta_r)=A/\ann_A(\Delta_r)$ of $\Delta_r$.
\end{enumerate}

We will prove now that $B$ coincides with the faithful algebra $B(\Gamma)=A/\ann_A(\Gamma)$ of $\Gamma$. Observe that $t_A(\Gamma)\subseteq\ann_A(\Gamma)$
and $\ann_B(\Gamma)=\ann_A(\Gamma)/t_A(\Gamma)$. Therefore, it is sufficient to show that $\Gamma$ is a faithful subquiver of $\Gamma_B$. Let $M_{\Gamma}$
be the direct sum of all indecomposable $B$-modules lying in $\Gamma$. Then $M_{\Gamma}$ is a sincere module in $\mo B$, by definition of $B$. In order
to show that $M_{\Gamma}$ is a faithful $B$-module, it is enough to prove that there is a monomorphism $B\to (M_{\Gamma})^n$ for some positive integer
$n$. Let $B=P^{(l)}\oplus P^{(c)}$ be a decomposition of $B$ in $\mo B$ such that the indecomposable direct summands of $P^{(c)}$ are exactly the
indecomposable projective $B$-modules lying in the core $\Gamma=c({\mathcal C})$ of $\mathcal C$. Clearly, $P^{(c)}$ is a direct summands of $M_{\Gamma}$,
and hence there is a monomorphism $P^{(c)}\to M_{\Gamma}$ in $\mo B$. On the other hand, the indecomposable direct summands of $P^{(l)}$ form a complete family
of pairwise nonisomorphic indecomposable projective right modules over the left tilted algebra $B^{(l)}$ of $B$. Hence, if $P^{(l)}=0$, or equivalently the left
part $\Delta_l$ of $\Delta$ is empty, then $B=P^{(l)}$ and $M_{\Gamma}$ is a faithful module in $\mo B$, as required. Therefore, assume that $\Delta_l$ is nonempty.

Let $\Gamma^{(l)}$ be the family of all indecomposable modules $X$ in $\Gamma$ such that there is an arrow $Y\to X$ in $\mathcal C$ with $Y$ from $\Delta_l$.
We claim that, for any module $X$ in $\Gamma^{(l)}$, there exists an indecomposable projective module $P$ in $\Gamma$ such that $\Hom_B(P,X)\neq 0$. We may
assume that $X$ is not projective. Then $\tau_BX$ is an indecomposable module not lying in $\Gamma$, because we have a path $\tau_BX\to Y\to X$ in $\mathcal C$,
with $X$ in the cyclic component $\Gamma$ of $\Gamma_B$ and $Y$ not in $\Gamma$. Observe that then $\tau_BX\in\Delta_l$ because $X$ is in
$\Gamma=\Delta_c=\Delta'_l\cap\Delta'_r$. Consider now an oriented cycle in $\Gamma$
$$X=X_0\to X_1\to \cdots\to X_{r-1}\to X_r=X$$
passing through $X$. It follows from Proposition \ref{prop23n} that there exists $i\in \{2, \ldots, r\}$ such that $\tau_BX_i\cong X_{i-2}$. Since $\tau_BX$
does not belong to $\Gamma$, we then conclude that there is in $\Gamma$ a sectional path
$$X_s\to X_{s+1}\to \cdots\to X_{r-1}\to X_r=X$$
with $X_s=P$ an indecomposable projective module. Hence we obtain that $\Hom_B(P,X)\neq 0$, because the composition of irreducible homomorphisms in $\mo B$
corresponding to arrows of a sectional path in $\Gamma_B$ is nonzero, by a theorem of Bautista and Smal{\o} \cite{BaSm} (see also \cite[Theorem III.11.2]{SY}).
Observe also that, for any module $Y$ lying on $\Delta_l$, we have $\Hom_B(P^{(c)},Y) = 0$, because $Y$ is a module in $\mo B^{(l)}$. This leads to the following
property of modules in $\Gamma^{(l)}$: any irreducible homomorphism $f: Y\to X$ with $X$ in $\Gamma^{(l)}$ and $Y$ in $\Delta_l$ is a monomorphism.

Consider now the family $\Omega^{(l)}$ of all indecomposable modules $Y$ in $\Delta_l$ such that there is an arrow in $\mathcal C$ from $Y$ to a module $X$
in $\Gamma$, and hence in $\Gamma^{(l)}$. Moreover, let $M^{(l)}$ be the direct sum of all indecomposable modules in $\Omega^{(l)}$. Observe that $M^{(l)}$
is a right $B^{(l)}$-module. Moreover, for any module $Y$ in $\Omega^{(l)}$, there is an irreducible monomorphism $Y\to X$ in $\mo B$ with $X$
lying in $\Gamma^{(l)}$. This implies that there is a monomorphism in $\mo B$ of the form $M^{(l)}\to (M_{\Gamma})^m$ for some positive integer $m$.
We will prove that $M^{(l)}$ is a faithful right $B^{(l)}$-module.

Let $P$ be an indecomposable projective module in $\mo B^{(l)}$, or equivalently, an indecomposable direct summand of $P^{(l)}$. Since $M_{\Gamma}$
is a sincere module in $\mo B$, we conclude that there is an indecomposable module $Z$ in $\Gamma$ such that $\Hom_B(P,Z)\neq 0$.
Further, the radical $\rad\End_B(M_{\Gamma})$ of the endomorphism algebra $\End_B(M_{\Gamma})$ in nilpotent. Then there exist a path of irreducible
homomorphisms
\[ Z_{t+1} \buildrel {g_{t+1}}\over {\hbox to 8mm{\rightarrowfill}} Z_t \buildrel {g_{t}}\over {\hbox to 8mm{\rightarrowfill}} Z_{t-1}\to \cdots \to
Z_{2} \buildrel {g_2}\over {\hbox to 8mm{\rightarrowfill}} Z_1 \buildrel {g_{1}}\over {\hbox to 8mm{\rightarrowfill}} Z_0=Z \]
and a homomorphism $v_{t+1}: P\to Z_{t+1}$ in $\mo B$ with $g_1g_2\ldots g_tg_{t+1}v_{t+1}\neq 0$, $Z_0, Z_1, \ldots, Z_t$ indecomposable modules
in $\Gamma$ and $Z_{t+1}$ an indecomposable module in $\Delta_l$ (see \cite[Proposition III.10.1]{SY}). This implies that $\Hom_B(P,M^{(l)})$ $\neq 0$
because $Z_{t+1}$ is a direct summand of $M^{(l)}$. Therefore, $M^{(l)}$ is a sincere right $B^{(l)}$-module. We know also that $B^{(l)}$ is a tilted
algebra and $\Delta_l$ is a disjoint union of sections of the connecting components of the indecomposable parts of $B^{(l)}$ and the category of all
predecessors of $\Delta_l$ in $\ind B$ coincides with the category of all predecessors of $\Delta_l$ in $\ind B^{(l)}$. Then we conclude that,
for any indecomposable module $L$ in $\ind B^{(l)}$, we have
\[ \Hom_{B^{(l)}}(L,M^{(l)})=0 \,\,\,{\rm or}\,\,\, \Hom_{B^{(l)}}(M^{(l)},\tau_{B^{(l)}}L)=0. \]
Summing up, we proved that $M^{(l)}$ is a sincere module in $\mo B^{(l)}$ which is not the middle of a short chain in the sense of \cite{RSS} (see also \cite{AR2}).
Then it follows from \cite[Corollary 3.2]{RSS} that $M^{(l)}$ is a faithful module in $\mo B^{(l)}$. Hence, there exists a monomorphism $B^{(l)}\to (M^{(l)})^s$
in $\mo B^{(l)}$ for some positive integer $s$.

Finally, since there exist monomorphisms $M^{(l)}\to (M_{\Gamma})^m$ and $P^{(l)}\to (M_{\Gamma})$, and $B=P^{(l)}\oplus P^{(c)}$ with $P^{(l)}=B^{(l)}$
in $\mo B$, we obtain that there is a monomorphism in $\mo B$ of the form $B\to (M_{\Gamma})^n$ for some positive integer $n$. Therefore, $M_{\Gamma}$
is a faithful module in $\mo B$, and consequently $B=B(\Gamma)$. This finishes the proof of the theorem.

In connection with the final part of the above proof, we mention that, by a recent result proved by Jaworska, Malicki and Skowro\'nski in \cite{JMS1},
an algebra $A$ is a tilted algebra if and only if there exists a sincere module $M$ in $\mo A$ such that for any module $X$ in $\ind A$, we have
$\Hom_A(X,M)=0$ or $\Hom_A(M,\tau_AX)=0$. Moreover, all modules $M$ in a module category $\mo A$ not being the middle of short chains have been
described completely in \cite{JMS2}.
\section{Proof of Theorem \ref{thm10}}
\label{pfthm10}

Let $A$ be a cycle-finite algebra. Clearly, the statement (iv) implies the statements (i), (ii), (iii). Assume $A$ is of infinite representation type.
Then it follows from \cite[Corollary 4.3]{Sk7-5} that there is an idempotent $e$ in $A$ such that $B=A/AeA$ is a tame concealed algebra. Then $\ind B$
admits infinitely many pairwise nonisomorphic directing postprojective (respectively, preinjective) modules. Hence, $\ind B$ has infinitely many pairwise
nonisomorphic rigid modules, and consequently $\ind A$ has infinitely many pairwise nonisomorphic rigid modules. Applying Theorem \ref{thm5} we conclude
that $\ind A$ contains infinitely many pairwise nonisomorphic directing modules, because $A$ is a cycle-finite algebra.
Then $\Gamma_A$ contains a $\tau_A$-orbit containing infinitely many
directing modules, by \cite[Theorem 2.7]{PX} or \cite[Corollary 2]{Sk4}. Hence $\Gamma_A$ contains either an acyclic left stable full
translation subquiver $\mathcal D$ which is closed under predecessors or an acyclic right stable full translation subquiver $\mathcal E$
which is closed under successors. Since $A$ is a cycle-finite algebra, applying \cite[Theorem 2.2]{MPS2} and its dual, we conclude that
$\mathcal D$ (respectively, $\mathcal E$) consists entirely of directing modules, which are obviously also $\tau_A$-rigid and rigid
indecomposable modules. Therefore, any of the statements (i), (ii), (iii) implies the statement (iv).
\section{Examples: infinite cyclic components}
\label{exs-inf}

In this section we present examples illustrating Theorem \ref{thm1}. 
\begin{ex}
Let $K$ be a field and $A=KQ/I$ the bound quiver algebra given by the quiver $Q$ of the form
$$\xymatrix@C=16pt@R=14pt{
&&&9\ar[rd]^{\eta}&&16\ar[ld]_{\psi}\ar[rd]_{l}&&19\ar[rd]^{i}\ar[ll]_{j}\\
&&10\ar[ru]^{\xi}\ar[rd]_{\mu}\ar[dd]_{\pi}&&7\ar[dddd]^{\rho}&&17\ar[d]_{m}&&20\ar[d]^{h}\\
&&&8\ar[ru]_{\nu}&&&18&&21\ar[dd]^{g}\ar[rd]^{f}\\
0\ar[r]^{\theta}&1&2\ar[l]_{\omega}\ar[dd]_{\kappa}&&&&&&&22\ar[ld]^{e}\\
&&&5\ar[ld]_{\beta}&&&&&15\ar[d]^{d}\\
&&3&&6\ar[ld]^{\sigma}\ar[lu]_{\alpha}\ar[rd]_{\varphi}&&12\ar[ld]^{a}\ar[rd]_{b}&&14\ar[ld]^{c}\\
&&&4\ar[lu]^{\gamma}&&11\ar[rd]^{v}&&13\\
&&&&&&23&24\ar[l]_{t}&25\ar[l]_{u}\\
&&&&&&&26\ar[u]_{s}\ar@(r,d)^{r}\\
}$$
and $I$ the ideal in the path algebra $KQ$ of $Q$ over $K$ generated by the elements $\alpha\beta - \sigma\gamma$, $\xi\eta - \mu\nu$,
$\pi\kappa - \xi\eta\rho\alpha\beta$, $\rho\varphi$, $\psi\rho$, $jl$, $dc$, $ed$, $gd$, $hg$, $hf$, $ih$, $av$, $rs$, $st$, $r^2$.
Then $A$ is a cycle-finite algebra and $\Gamma_A$ admits a component $\mathcal C$ of the form
%
$$\includegraphics[scale=0.6]{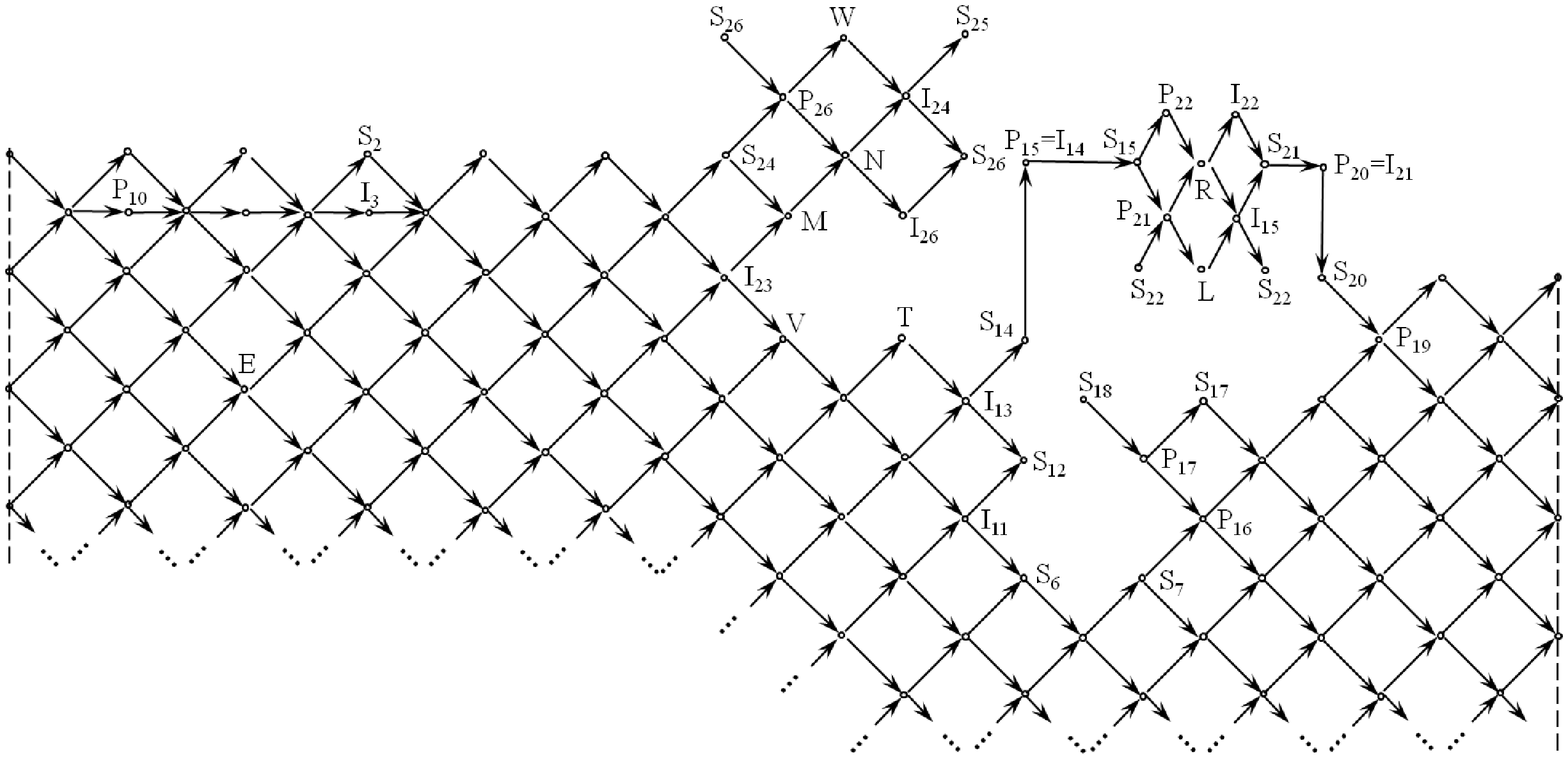}$$
\noindent The cyclic part ${}_c{\mathcal C}$ of $\mathcal C$ consists of one infinite component $\Gamma$ and one finite component $\Gamma'$ described
as follows.
The infinite cyclic component $\Gamma$ is obtained by removing from $\mathcal C$ the modules
$S_{12}, S_{17}, S_{18}, P_{17}$, $S_{24}$, $M$, $S_{26}$, $P_{26}$, $N$, $I_{26}$, $W$, $I_{24}$, $S_{25}$, and the arrows attached to them.
The finite cyclic component $\Gamma'$ is the full translation subquiver of $\mathcal C$ given by the vertices $S_{26}, P_{26}, N, I_{26}, W, I_{24}$.
The maximal cyclic coherent part $\Gamma^{cc}$ of $\Gamma$ is the full translation subquiver of $\mathcal C$ obtained by
removing from $\mathcal C$ the modules $S_{12}$, $I_{13}$, $T$, $S_{14}$, $P_{15}=I_{14}$, $S_{15}$, $P_{21}$, $S_{22}$, $L$, $P_{22}$, $R$, $I_{15}$, $I_{22}$,
$S_{21}$, $P_{20}=I_{21}$, $S_{20}$, $S_{17}$, $P_{17}$, $S_{18}$, $S_{24}$, $M$, $S_{26}$, $P_{26}$, $N$, $I_{26}$, $W$, $I_{24}$, $S_{25}$, and the arrows attached to them.
Further, $\Gamma^{cc}$ is the cyclic part of the maximal almost cyclic coherent full translation subquiver $\Gamma^*$ of $\mathcal C$ obtained by removing from $\mathcal C$
the modules $P_{15}=I_{14}$, $S_{15}$, $P_{21}$, $S_{22}$, $L$, $P_{22}$, $R$, $I_{15}$, $I_{22}$, $S_{21}$, $P_{20}=I_{21}$, $S_{26}$, $P_{26}$, $I_{26}$, $W$, $N$, $I_{24}$
and the arrows attached to them, and shrinking the sectional path $M\to N\to I_{24}\to S_{25}$ to the arrow $M\to S_{25}$.

Let $B=A/\ann_A(\Gamma)$. Then $B=A/\ann_A(\Gamma^*)$, because $\ann_A(\Gamma) = \ann_A(\Gamma^*)$. Observe that $B=KQ_B/I_B$, where $Q_B$ is the full subquiver
of $Q$ given by all vertices of $Q$ except 15, 21, 22, 26, and $I_B=I\cap KQ_B$. We claim that $B$ is a tame generalized multicoil algebra. Consider
the path algebra $C=K\Sigma$ of the full subquiver $\Sigma$ of $Q$ given by the vertices 4, 5, 6, 7, 8, 9. Then $C$ is a hereditary algebra of Euclidean type
$\widetilde{\Bbb D}_5$, and hence a tame concealed algebra. It is known that $\Gamma_C$ admits an infinite family $\mathcal{T}^{C}_{\lambda}$,
${\lambda\in\Lambda(C)}$, of pairwise orthogonal generalized standard stable tubes, having a unique stable tube, say $\mathcal{T}_{1}^C$, of rank
3 with the mouth formed by the modules $S_6=\tau_CS_7$, $S_7=\tau_CE$, $E=\tau_CS_6$, where $E$ is the unique indecomposable $C$-module with the
dimension vector
$\underline{\dim}E = {\begin{smallmatrix}1{\phantom{1}} \\ {\phantom{1}}1 \\ 1{\phantom{1}} \\ {\phantom{1}}{\phantom{1}} \\
1{\phantom{1}} \\ {\phantom{1}}1 \\ 1{\phantom{1}}
\end{smallmatrix}}$, (see \cite[Section 6]{DlRi} and \cite[Theorem XIII 2.9]{SS1}).

Then $B$ is the generalized multicoil enlargement of $C$, obtained by applications of the following admissible operations:
\begin{itemize}
\item two admissible operations of types (ad~1$^*$) with the pivots $S_6$ and $S_{12}$, creating the vertices 11, 12, 13, 14 and the arrows
$\varphi$, $a$, $b$, $c$;
\item two admissible operations of types (ad~1$^*$) with the pivots $E$ and $S_{2}$, creating the vertices 3, 2, 1, 0 and the arrows
$\beta$, $\gamma$, $\kappa$, $\omega$, $\theta$;
\item two admissible operations of types (ad~1) with the pivots $S_7$ and $S_{16}$, creating the vertices 16, 17, 18, 19, 20 and the arrows
$\psi$, $l$, $m$, $j$, $i$;
\item one admissible operation of type (ad~3) with the pivot the radical of $P_{10}$, creating the vertex 10 and the arrows $\xi$, $\mu$, $\pi$;
\item one admissible operation of type (ad~1$^*$) with the pivot $V$ being the unique indecomposable module of dimension $2$ having $S_{11}$ as the socle
and $S_6$ as the top, creating the vertices 23, 24, 25 and the arrows $v$, $t$, $u$.
\end{itemize}
Then the left part $B^{(l)}$ of $B$ is the convex subcategory of $B$ (and of $A$) given by the vertices 0, 1, 2, 3, 4, 5, 6, 7, 8, 9, 11, 12, 13, 14, 23, 24, 25,
and is a tilted algebra of Euclidean type $\widetilde{\Bbb D}_{16}$ with the connecting postprojective component ${\mathcal P}^{B^{(l)}}$ containing all
indecomposable projective $B^{(l)}$-modules.
The right part $B^{(r)}$ of $B$ is the convex subcategory of $B$ (and of $A$) given by the vertices 0, 1, 2, 4, 5, 6, 7, 8, 9, 10, 16, 17, 18, 19, 20,
and is a tilted algebra of Euclidean type $\widetilde{\Bbb D}_{14}$ with the connecting preinjective component ${\mathcal Q}^{B^{(r)}}$ containing all
indecomposable injective $B^{(r)}$-modules. We also note that
the left border $\Delta_l$ of the generalized multicoil $\Gamma^*$ of $\Gamma_B$ is given by the quivers $P_{17}\to S_{17}$ and $S_{20}$, and the right
border $\Delta_r$ of $\Gamma^*$ is given by the quivers $T\to I_{13}\to S_{12}$ and $S_{24}\to M$. Further,
the algebra $B(\Gamma\setminus\Gamma^{cc}) = A/\ann_A(\Gamma\setminus\Gamma^{cc})$ is the disjoint union of three representation-finite convex subcategories of $A$:
$D_1$ given by the vertices 12, 13, 14, 15, 20, 21, 22, $D_2$ given by the vertices 17, 18, and $D_3$ given by the vertices 24, 25, 26.
We note that $D_3$ is the faithful algebra $B(\Gamma')$ of the finite cyclic component $\Gamma'$.
It follows from \cite[Theorems C and F]{MS2} that the Auslander-Reiten quiver $\Gamma_B$ of the generalized multicoil enlargement $B$ of $C$ is of the form
$$\Gamma_B={\mathcal P}^B \cup {\mathcal C}^B \cup {\mathcal Q}^B,$$
where ${\mathcal P}^B = {\mathcal P}^{B^{(l)}}$, ${\mathcal Q}^B = {\mathcal Q}^{B^{(r)}}$, and ${\mathcal C}^B$ is the family
${\mathcal C}^{B}_{\lambda}$, ${\lambda\in\Lambda(C)}$, of pairwise orthogonal generalized multicoils such that ${\mathcal C}^B_1 = \Gamma^*$
and ${\mathcal C}^{B}_{\lambda}={\mathcal T}^{C}_{\lambda}$ for all ${\lambda\in\Lambda(C)}\setminus\{1\}$. Hence $\Gamma_A$ is of the form
$$\Gamma_A={\mathcal P}^A \cup {\mathcal C}^A \cup {\mathcal Q}^A,$$
where ${\mathcal P}^A = {\mathcal P}^{B^{(l)}}$, ${\mathcal Q}^A = {\mathcal Q}^{B^{(r)}}$, and ${\mathcal C}^A$ is the family
${\mathcal C}^{A}_{\lambda}$, ${\lambda\in\Lambda(C)}$, of pairwise orthogonal generalized standard components such that
${\mathcal C}^A_1 = {\mathcal C}$, ${\mathcal C}^{A}_{\lambda}={\mathcal T}^{C}_{\lambda}$ for all ${\lambda\in\Lambda(C)}\setminus\{1\}$.
Moreover, we have
$$\Hom_A({\mathcal C}^{A},{\mathcal P}^{A})=0, \Hom_A({\mathcal Q}^{A},{\mathcal C}^{A})=0, \Hom_A({\mathcal Q}^{A},{\mathcal P}^{A})=0.$$
In particular, $A$ is a cycle-finite algebra with $(\rad^{\infty}_A)^3=0$.
\end{ex}
\begin{ex}
Let $K$ be a field and $B=KQ/J$ the bound quiver algebra given by the quiver $Q$ of the form
\[
\xymatrix@C=14pt@R=16pt{
&&38\cr
31\ar[r]^{\psi_3}&30\ar[r]^{\psi_2}&29\ar[u]_{\kappa_1}&26\ar[l]_{\psi_1}\ar[ld]_{\varphi_2}&32\ar[l]_{\eta_1}\ar[r]^{\eta_2}&33\ar@/_2pc/[lllu]_{\kappa_2}&&&&&&19\ar[lld]_{\varepsilon_3}\cr
&&25&&28\ar[lu]_{\varphi_1}\ar[ld]^{\varphi_3}&34\ar[u]_{\omega_2}\ar[d]^{\omega_1}&&11\ar[rr]_{\delta_3}&&10\ar[rr]_{\delta_2}&&9&18\ar[lu]_{\varepsilon_2}\ar[lddd]^{\varepsilon_1}\cr
&&&27\ar[lu]^{\varphi_4}&&15\ar[r]^{\sigma_4}&14\ar[r]^{\sigma_3}&13\ar[rr]^{\sigma_2}&&12\cr
36\ar[uuu]_{\pi_3}&&37\ar[rrrr]^{\lambda_1}\ar@/_8pc/[rrrrrrrrrrdddd]^{\lambda_2}&&&&17\ar[ld]_{\rho_2}&16\ar[l]_{\xi_2}\ar[d]^{\xi_1}\cr
&&&&&40&&1\ar[ll]_{\rho_1}\ar[ld]_{\alpha_4}&&2\ar[ll]_{\alpha_3}\ar[uu]_{\sigma_1}&&3\ar[ll]_{\alpha_2}\ar[uuu]^{\delta_1}\cr
&&&&&&0&&4\ar[ll]_{\beta_3}&&5\ar[ll]_{\beta_2}&&8\ar[lu]_{\alpha_1}\ar[ll]_{\beta_1}\ar[lld]_{\gamma_1}\cr
&&&&&39&&&6\ar[lll]_{\theta_2}\ar[llu]_{\gamma_3}&&7\ar[ll]_{\gamma_2}\ar[d]_{\mu_1}\cr
35\ar[uuuu]_{\pi_2}\ar[rrrrr]^{\pi_1}&&&&&24\ar[u]^{\theta_1}&23\ar[l]_{\nu_3}&&22\ar[ll]_{\nu_2}\ar[u]_{\nu_1}&&20&&21\ar[ll]_{\mu_2}\cr
}
\]
\vskip 5mm
\noindent and $J$ the ideal in the path algebra $KQ$ of $Q$ over $K$ generated by the elements
$\varphi_1\psi_1$, $\eta_1\varphi_2$, $\omega_1\sigma_4\sigma_3\sigma_2$, $\pi_3\psi_3\psi_2$, $\lambda_2\mu_2$, $\eta_2\kappa_2-\eta_1\psi_1\kappa_1$,
$\psi_2\kappa_1$, $\omega_2\kappa_2$, $\gamma_2\theta_2$, $\pi_1\theta_1$, $\nu_1\theta_2-\nu_2\nu_3\theta_1$, $\alpha_3\rho_1$, $\lambda_1\rho_2$,
$\xi_1\rho_1-\xi_2\rho_2$, $\alpha_1\alpha_2\alpha_3\alpha_4+\beta_1\beta_2\beta_3+\gamma_1\gamma_2\gamma_3$,
$\alpha_1\delta_1$, $\alpha_2\sigma_1$, $\xi_1\alpha_4$, $\varepsilon_1\delta_1$, $\varepsilon_1\alpha_2$, $\varepsilon_3\delta_2$, $\gamma_1\mu_1$, $\nu_1\gamma_3$.
Denote by $P_k, I_k, S_k$ the indecomposable projective module, the indecomposable injective module, and the simple module in $\mo B$
at the vertex $k$ of $Q$.
Then $\Gamma_B$ admits a cyclic component $\mathcal C$ obtained by identification the sectional paths $H_{1}\to H_{2}\to H_{3}$,
$L_{1}\to L_{2}$, $N_1\to N_2$ and the module $S_{21}$
occurring in the following three translation quivers: ${\mathcal C}_1$ of the form
$$\includegraphics[scale=0.8]{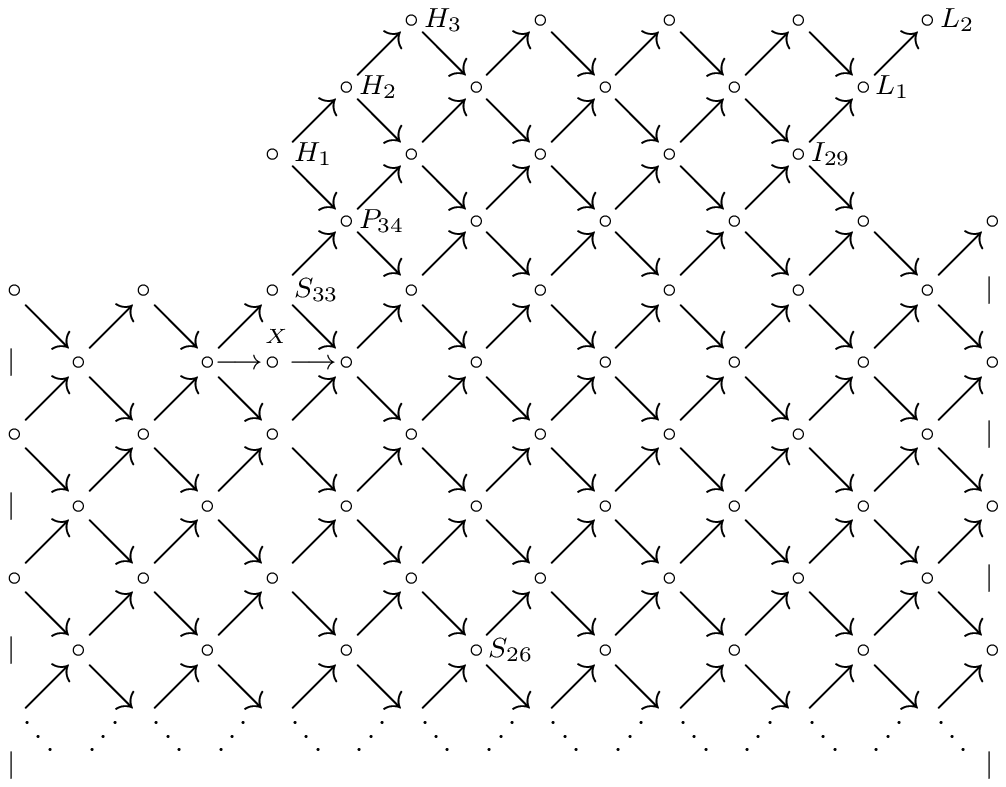}$$
${\mathcal C}_2$ of the form
$$\includegraphics[scale=0.8]{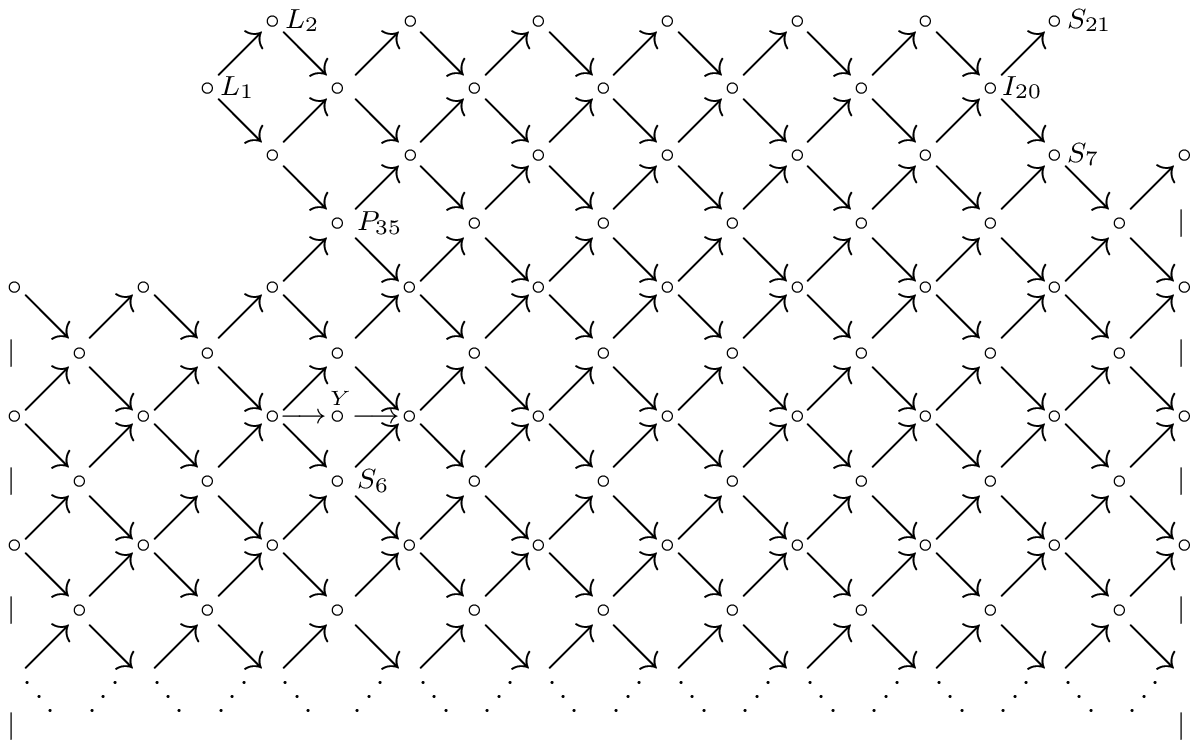}$$
and ${\mathcal C}_3$ of the form
$$\includegraphics[scale=0.8]{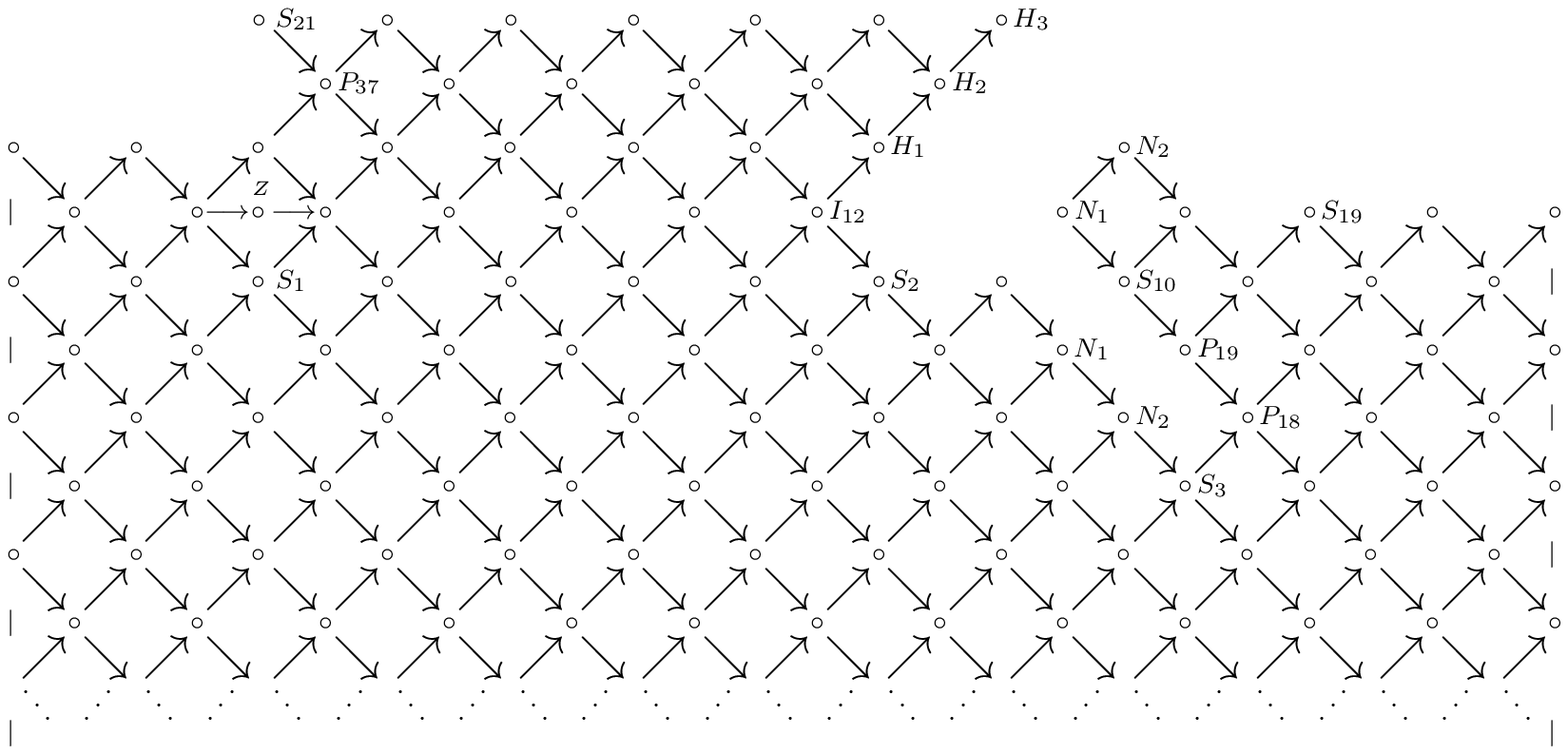}$$
where $X=I_{38}=P_{32}$, $Y=I_{39}=P_{22}$, $Z=I_{40}=P_{16}$, $N_2=I_9$ and  the vertical dashed lines have to be identified in order to obtain the translation quivers ${\mathcal C}_1$, ${\mathcal C}_2$ and ${\mathcal C}_3$.
We claim that $B$ is a generalized multicoil algebra. Denote by $Q_C$ the full subquiver of $Q$ given by the vertices $0, 1, 2, 3, 4, 5, 6, 7, 8$. Consider
the bound quiver algebra $C=KQ_C/J_C$ with $J_C$ the ideal in $KQ_C$ generated by $\alpha_1\alpha_2\alpha_3\alpha_4+\beta_1\beta_2\beta_3+\gamma_1\gamma_2\gamma_3$
and the path algebra $D=KQ_D$ of the full subquiver $Q_D$ of $Q$ given by the vertices $25, 26, 27, 28$. Then $C$ is a canonical algebra of wild type
and $D$ is a canonical algebra of Euclidean type $\widetilde{{\Bbb A}}_3$.
It is known that $\Gamma_C$ admits an infinite family $\mathcal{T}^{C}_{\lambda}$, ${\lambda\in\Lambda(C)}$, of pairwise orthogonal stable tubes,
having a unique stable tube, say $\mathcal{T}_{1}^C$, of rank 4 with the mouth formed by the modules $S_1=\tau_CS_2$, $S_2=\tau_CS_3$, $S_3=\tau_CE$, $E=\tau_CS_1$,
where $E$ is the unique indecomposable $C$-module with the dimension vector
$\underline{\dim}E = {\begin{smallmatrix}{\phantom{1}}0{\phantom{1}}0{\phantom{1}}0{\phantom{1}} \\ 1{\phantom{1}}1{\phantom{1}}1{\phantom{1}}1 \\
{\phantom{1}}{\phantom{1}}1{\phantom{1}}1{\phantom{1}}{\phantom{1}}
\end{smallmatrix}},$ and
a stable tube, say $\mathcal{T}_{2}^C$, of rank 3 with the mouth formed by the modules $S_6=\tau_CS_7$, $S_7=\tau_CF$, $F=\tau_CS_6$,
where $F$ is the unique indecomposable $C$-module with the dimension vector
$\underline{\dim}F = {\begin{smallmatrix}{\phantom{1}}1{\phantom{1}}1{\phantom{1}}1{\phantom{1}} \\ 1{\phantom{1}}1{\phantom{1}}1{\phantom{1}}1 \\
{\phantom{1}}{\phantom{1}}0{\phantom{1}}0{\phantom{1}}{\phantom{1}}
\end{smallmatrix}}$ (see \cite[(3.7)]{Ri1}).
Moreover, $\Gamma_D$ admits an infinite family $\mathcal{T}^{D}_{\mu}$, ${\mu\in\Lambda(D)}$, of pairwise orthogonal stable tubes,
having a stable tube, say $\mathcal{T}_{1}^D$, of rank 2 with the mouth formed by the modules $S_{26}=\tau_DG$, $G=\tau_DS_{26}$,
where $G$ is the unique indecomposable $D$-module with the dimension vector
$\underline{\dim}G = {\begin{smallmatrix}{\phantom{1}}0{\phantom{1}} \\ 1{\phantom{1}}1 \\ {\phantom{1}}1{\phantom{1}}
\end{smallmatrix}}$.
Denote by $C_i=KQ_{C_i}/J_{C_i}$ the bound quiver algebra, where $Q_{C_i}$ is the full subquiver of $Q$ given by the vertices
0, 1, 2, $\ldots, i$, $i\geq 8$ ($C_8=C$), $J_{C_i}=J\cap KQ_{C_i}$, and by $D_j=KQ_{D_j}/J_{D_j}$ the bound quiver algebra, where $Q_{D_j}$ is the full subquiver
of $Q$ given by the vertices 25, 26, 27, $\ldots, j$, $j\geq 28$ ($D_{28}=D$), $J_{D_j}=J\cap KQ_{D_j}$.
Moreover, for each $k\in\{8, 9, \ldots, 39\}$ (respectively, $k\in\{28, 29, \ldots, 33\}$)
and $l\in\{0, 1, \ldots, 39\}$ (respectively, $l\in\{28, 29, \ldots, 33\}$),
we denote by $P_l^{C_k}, I_l^{C_k}, S_l^{C_k}$ (respectively, $P_l^{D_k}, I_l^{D_k}, S_l^{D_k}$) the indecomposable projective module,
the indecomposable injective module, and the simple module in $\mo C_k$ (respectively, in $\mo D_k$) at the vertex $l$ of $Q_{C_k}$
(respectively, of $Q_{D_k}$).
Then $B$ is the generalized multicoil enlargement of $C\times D$, obtained by applications of the following admissible operations:
\begin{itemize}
\item one admissible operation of type (ad~1$^*$) with the pivot $S_3^{C_8}$, creating the vertices 9, 10, 11
and the arrows $\delta_1$, $\delta_2$, $\delta_3$;
\item one admissible operation of type (ad~1$^*$) with the pivot $S_{2}^{C_{11}}$, creating the vertices 12, 13, 14, 15
and the arrows $\sigma_1$, $\sigma_2$, $\sigma_3$, $\sigma_4$;
\item one admissible operation of type (ad~1) with the pivot $S_1^{C_{15}}$, creating the vertices 16, 17
and the arrows $\xi_1$, $\xi_2$;
\item one admissible operation of type (ad~4) with the pivot $S_3^{C_{17}}$ and the finite sectional path $S_{10}^{C_{17}}\to I_{10}^{C_{17}}$,
creating the vertices 18, 19 and the arrows $\varepsilon_1$, $\varepsilon_2$, $\varepsilon_3$;
\item one admissible operation of type (ad~1$^*$) with the pivot $S_7^{C_{19}}$, creating the vertices 20, 21
and the arrows $\mu_1$, $\mu_2$;
\item one admissible operation of type (ad~1) with the pivot $S_6^{C_{21}}$, creating the vertices 22, 23, 24
and the arrows $\nu_1$, $\nu_2$, $\nu_3$;
\item one admissible operation of type (ad~1$^*$) with the pivot $S_{26}^{D_{28}}$, creating the vertices 29, 30, 31
and the arrows $\psi_1$, $\psi_2$, $\psi_2$;
\item one admissible operation of type (ad~1) with the pivot $W$ being the unique indecomposable module of dimension $2$ having $S_{29}^{D_{31}}$
as the socle and $S_{26}^{D_{31}}$ as the top, creating the vertices 32, 33 and the arrows $\eta_1$, $\eta_2$;
\item one admissible operation of type (ad~4) with the pivot $S_{33}^{D_{33}}$ and the finite sectional path
$I_{13}^{C_{24}}\to I_{14}^{C_{24}}\to S_{15}^{C_{24}}$, creating the vertex 34 and the arrows $\omega_1$, $\omega_2$;
\item one admissible operation of type (ad~4) with the pivot $S_{24}^{C_{34}}$ and the finite sectional path $I_{30}^{C_{34}}\to S_{31}^{C_{34}}$,
creating the vertices 35, 36 and the arrows $\pi_1$, $\pi_2$, $\pi_3$;
\item one admissible operation of type (ad~4) with the pivot $S_{17}^{C_{36}}$ and the module $S_{21}^{C_{36}}$,
creating the vertex 37 and the arrows $\lambda_1$, $\lambda_2$;
\item one admissible operation of type (ad~2$^*$) with the pivot $P_{32}^{C_{37}}$,
creating the vertex 38 and the arrows $\kappa_1$, $\kappa_2$.
\item one admissible operation of type (ad~2$^*$) with the pivot $P_{22}^{C_{38}}$,
creating the vertex 39 and the arrows $\theta_1$, $\theta_2$.
\item one admissible operation of type (ad~2$^*$) with the pivot $P_{16}^{C_{39}}$,
creating the vertex 40 and the arrows $\rho_1$, $\rho_2$.
\end{itemize}
Then the left part $B^{(l)}$ of $B$ is the convex subcategory of $B$ being the product $B^{(l)}=B^{(l)}_1\times B^{(l)}_2$,
where $B^{(l)}_1=KQ^{(l)}_1/J^{(l)}_1$ is the branch coextension of the canonical algebra $C$ and $B^{(l)}_2=$ $KQ^{(l)}_2/J^{(l)}_2$
is the branch coextension of the canonical algebra $D$ given by the quivers
\[
\xymatrix@C=14pt@R=16pt{
&&38\cr
31\ar[r]^{\psi_3}&30\ar[r]^{\psi_2}&29\ar[u]_{\kappa_1}&26\ar[l]_{\psi_1}\ar[ld]_{\varphi_2}\cr
Q^{(l)}_2&&25&&28\ar[lu]_{\varphi_1}\ar[ld]^{\varphi_3}&&&11\ar[rr]_{\delta_3}&&10\ar[rr]_{\delta_2}&&9\cr
&&&27\ar[lu]^{\varphi_4}&&15\ar[r]^{\sigma_4}&14\ar[r]^{\sigma_3}&13\ar[rr]^{\sigma_2}&&12\cr
&&&&&&17\ar[ld]_{\rho_2}\cr
&&&&&40&&1\ar[ll]_{\rho_1}\ar[ld]_{\alpha_4}&&2\ar[ll]_{\alpha_3}\ar[uu]_{\sigma_1}&&3\ar[ll]_{\alpha_2}\ar[uuu]^{\delta_1}\cr
&&&&Q^{(l)}_1&&0&&4\ar[ll]_{\beta_3}&&5\ar[ll]_{\beta_2}&&8\ar[lu]_{\alpha_1}\ar[ll]_{\beta_1}\ar[lld]_{\gamma_1}\cr
&&&&&39&&&6\ar[lll]_{\theta_2}\ar[llu]_{\gamma_3}&&7\ar[ll]_{\gamma_2}\ar[d]_{\mu_1}\cr
&&&&&24\ar[u]^{\theta_1}&23\ar[l]_{\nu_3}&&&&20&&21\ar[ll]_{\mu_2}\cr
}
\]
and the ideals $J^{(l)}_1=KQ^{(l)}_1\cap J$ in $KQ^{(l)}_1$ and $J^{(l)}_2=KQ^{(l)}_2\cap J$ in $KQ^{(l)}_2$.
The right part $B^{(r)}$ of $B$ is the convex subcategory of $B$ being the product $B^{(r)}=B^{(r)}_1\times B^{(r)}_2$,
where $B^{(r)}_1=KQ^{(r)}_1/J^{(r)}_1$ is the branch extension of the canonical algebra $C$ and $B^{(r)}_2=KQ^{(r)}_2/J^{(r)}_2$
is the branch extension of the canonical algebra $D$ given by the quivers
\[
\xymatrix@C=14pt@R=16pt{
31\ar[r]^{\psi_3}&30&&26\ar[ld]_{\varphi_2}&32\ar[l]_{\eta_1}\ar[r]^{\eta_2}&33&&&&&&19\ar[lld]_{\varepsilon_3}\cr
&Q^{(r)}_2&25&&28\ar[lu]_{\varphi_1}\ar[ld]^{\varphi_3}&34\ar[u]_{\omega_2}\ar[d]^{\omega_1}&&11\ar[rr]_{\delta_3}&&10&&&18\ar[lu]_{\varepsilon_2}\ar[lddd]^{\varepsilon_1}\cr
&&&27\ar[lu]^{\varphi_4}&&15\ar[r]^{\sigma_4}&14\ar[r]^{\sigma_3}&13\cr
36\ar[uuu]_{\pi_3}&&37\ar[rrrr]^{\lambda_1}\ar@/_8pc/[rrrrrrrrrrdddd]^{\lambda_2}&&&&17&16\ar[l]_{\xi_2}\ar[d]^{\xi_1}\cr
&&&&&&&1\ar[ld]_{\alpha_4}&&2\ar[ll]_{\alpha_3}&&3\ar[ll]_{\alpha_2}\cr
&&&&&Q^{(r)}_1&0&&4\ar[ll]_{\beta_3}&&5\ar[ll]_{\beta_2}&&8\ar[lu]_{\alpha_1}\ar[ll]_{\beta_1}\ar[lld]_{\gamma_1}\cr
&&&&&&&&6\ar[llu]_{\gamma_3}&&7\ar[ll]_{\gamma_2}\cr
35\ar[uuuu]_{\pi_2}\ar[rrrrr]^{\pi_1}&&&&&24&23\ar[l]_{\nu_3}&&22\ar[ll]_{\nu_2}\ar[u]_{\nu_1}&&&&21\cr
}
\]
\vskip 5mm
\noindent and the ideals $J^{(r)}_1=KQ^{(r)}_1\cap J$ in $KQ^{(r)}_1$ and $J^{(r)}_2=KQ^{(r)}_2\cap J$ in $KQ^{(r)}_2$.
It follows from \cite[Theorems C and F]{MS2} that the Auslander-Reiten quiver $\Gamma_B$ of the generalized multicoil enlargement $B$ of $C\times D$
is of the form
$$\Gamma_B={\mathcal P}^B \cup {\mathcal C}^B \cup {\mathcal Q}^B,$$
where ${\mathcal P}^B, {\mathcal C}^B, {\mathcal Q}^B$ are of the following families of components:
\begin{itemize}
\item ${\mathcal C}^B$ is a family of pairwise orthogonal generalized multicoils consisting of the faithful cyclic component ${\mathcal C}$ (described above),
the family ${\mathcal T}^C_{\lambda}$, $\lambda\in\Lambda(C)\setminus\{1,2\}$, of stable tubes of $\Gamma_C$, and the family ${\mathcal T}^D_{\mu}$,
$\mu\in\Lambda(D)\setminus\{1\}$, of stable tubes of $\Gamma_D$;
\item ${\mathcal P}^B={\mathcal P}^{B^{(l)}}$ and consists of the unique postprojective component ${\mathcal P}(A^{(l)}_1)$ of the wild concealed algebra
$A^{(l)}_1$ being the convex subcategory of $B^{(l)}_1$ given by all object of $B^{(l)}_1$ except 8, the unique postprojective component
${\mathcal P}(B^{(l)}_2)={\mathcal P}^{B^{(l)}_2}$ of the tilted algebra $B^{(l)}_2$ of Euclidean type $\widetilde{\Bbb A}_7$, one component with the stable part
${\Bbb Z}{\Bbb A}_{\infty}$ containing the indecomposable projective $B^{(l)}_1$-module at the vertex 8, and infinitely many regular components of the form
${\Bbb Z}{\Bbb A}_{\infty}$;
\item ${\mathcal Q}^B={\mathcal Q}^{B^{(r)}}$ and consists of the unique preinjective component ${\mathcal Q}(A^{(r)}_1)$ of the wild concealed algebra
$A^{(r)}_1$ being the convex subcategory of $B^{(r)}_1$ given by all object of $B^{(r)}_1$ except 0, the unique preinjective component
${\mathcal Q}(B^{(r)}_2)={\mathcal Q}^{B^{(r)}_2}$ of the tilted algebra $B^{(r)}_2$ of Euclidean type $\widetilde{\Bbb A}_9$, one component with the stable part
${\Bbb Z}{\Bbb A}_{\infty}$ containing the indecomposable injective $B^{(r)}_1$-module at the vertex 0, and infinitely many regular components of the form
${\Bbb Z}{\Bbb A}_{\infty}$.
\end{itemize}
Moreover, we have
\[ \Hom_{B}({\mathcal C}^B,{\mathcal P}^B) = 0, \Hom_{B}({\mathcal Q}^B,{\mathcal C}^B)=0, \Hom_{B}({\mathcal Q}^B,{\mathcal P}^B) = 0. \]
We also note that $B$ is not a cycle-finite algebra, because $\Gamma_B$ contains regular components of the form ${\Bbb Z}{\Bbb A}_{\infty}$ (see \cite[Lemma 3]{Sk4}).

Finally, we mention that the cyclic component $\mathcal C$ of $\Gamma_B$ is the cyclic generalized multicoil obtained from the stable tubes
${\mathcal T}^C_1$, ${\mathcal T}^C_2$ of $\Gamma_C$ and the stable tube ${\mathcal T}^D_1$ of $\Gamma_D$ by the 14 translation quiver admissible operations
\cite[Section 2]{MS1} corresponding to the 14 admissible algebra operations leading from $C\times D$ to $B$, described above. We also point that the cyclic
component $\mathcal C$ has a M\"obius strip configuration obtained by identifying in ${\mathcal C}_3$ two sectional paths $N_1\to N_2$.
\end{ex}
\section{Examples: finite cyclic components}
\label{exs-fin}

In this section we present examples illustrating Theorem \ref{thm2} and showing faithful almost acyclic Auslander-Reiten components of new types.
\begin{ex} \label{ex6-1}
Let $K$ be a field, $n\geq 7$ a natural number, and $A_n=KQ_n/I_n$ the bound quiver algebra given by the quiver $Q_n$ of the form
\[
\xymatrix@C=23pt@R=18pt{
&&&2\ar[rrd]^{\gamma}\cr
0\ar@(l,u)^{\varepsilon}&1\ar[l]_{\eta}\ar[rru]^{\beta}\ar[rd]_{\varrho}&&&&5\ar[r]^{\sigma_{6}}&6\ar[r]^{\hspace{-1mm}\sigma_{7}}&
\cdots\ar[r]^{\hspace{-3mm}\sigma_{n-1}}&{n-1}\ar[r]^{\hspace{2mm}\sigma_n}&{n}\cr
&&3\ar[rr]_{\delta}&&4\ar[ru]_{\omega}\cr
}
\]
and $I_n$ the ideal in the path algebra $KQ_n$ of $Q_n$ over $K$ generated by the elements $\varepsilon^2, \eta\varepsilon$ and $\beta\gamma-\varrho\delta\omega$.
Then the category $\mo A_n$ is equivalent to the category $\rep_K(Q_n,I_n)$ of the $K$-linear representations of the bound quiver $(Q_n,I_n)$.
Consider the indecomposable module $M_n$ in $\mo A_n$ corresponding to the indecomposable representation in $\rep_K(Q_n,I_n)$ of the form
\[
\xymatrix@C=20pt@R=18pt{
&&&K\ar[rrd]^{1}\cr
K^2\ar@(l,u)^{\left[\begin{smallmatrix}0&1\\0&0\end{smallmatrix}\right]}&K\ar[l]_{\left[\begin{smallmatrix}0\\1\end{smallmatrix}\right]}\ar[rru]^{1}\ar[rd]_{1}&&&&K\ar[r]^{1}&K\ar[r]^{1}&\cdots\ar[r]^{1}&K\ar[r]^{1}&{K}\cr
&&K\ar[rr]_{1}&&K\ar[ru]_{1}\cr
}
\]
We note that $M_n$ is a faithful $A_n$-module, and hence $B(M_n)=A_n$. Let $\Omega_n$ be the full subquiver of $Q_n$ given by the vertices $2, 3, 4, 5, 6, \ldots, n-1, n$
and the arrows $\gamma, \delta, \omega, \sigma_6, \sigma_7, \ldots,$ $\sigma_{n-1}, \sigma_n$, and $H_n=K\Omega_n$ the associated path algebra. Then $H_n$ is
a hereditary algebra. Observe that $H_7, H_8, H_9$ are hereditary algebras of Dynkin types ${\Bbb E}_6$, ${\Bbb E}_7$, ${\Bbb E}_8$ (respectively), $H_{10}$ is
a hereditary algebra of Euclidean type $\widetilde{{\Bbb E}}_8$, and, for $n\geq 11$, $H_n$ is a hereditary algebra of wild type. For each $i\in\{0, 1, \ldots, n-1, n\}$,
we denote by $P_i, I_i, S_i$ the indecomposable projective module, the indecomposable injective module, the simple module in $\mo A_n$ at the vertex $i$ of $Q_n$.
Moreover, for each $j\in\{2, 3, 4, \ldots, n-1, n\}$, we denote by $I^*_j$ the indecomposable injective module in $\mo H_n$ at the vertex $j$ of $\Omega_n$.
Further, let $\Lambda = K[\varepsilon]/(\varepsilon^2)$. Then $P_0$ is the indecomposable projective module in $\mo \Lambda$ and $S_0$ is its top. Finally, observe
that $A_n$ is the one-point extension algebra
\[
\left[\begin{matrix}\Lambda\times H_n&0 \\ S_0\oplus I^*_n&K\end{matrix}\right]
\]
of $\Lambda\times H_n$ by the module $S_0\oplus I^*_n$, with the extension vertex $1$.
Since $I^*_n$ is the indecomposable injective module in $\mo H_n$ and $H_n$ is a hereditary algebra, we conclude that
$$\Hom_{\Lambda\times H_n}(S_0\oplus I^*_n,\tau_{H_n}X)=\Hom_{H_n}(I^*_n,\tau_{H_n}X)=0$$
for any module $X$ in $\ind H_n$. Then, applying \cite[Corollary XV.1.7]{SS2} (see also \cite[Lemma 5.6]{SY0}), we conclude that every almost split sequence
in $\mo H_n$ is an almost split sequence in $\mo A_n$.
This implies that the Auslander-Reiten quiver $\Gamma_{H_n}$ of $H_n$ is a full translation subquiver of the Auslander-Reiten quiver $\Gamma_{A_n}$ of $A_n$.
In particular, we obtain that the preinjective component $Q(H_n)$ of $\Gamma_{H_n}$, containing the indecomposable injective modules $I^*_j$ , $j\in\{2, 3, 4,$ $\ldots, n-1, n\}$,
is a full translation subquiver of a component $\mathcal{C}_n$ of $\Gamma_{A_n}$ which is closed under predecessors.
Then the direct calculation shows that $\mathcal{C}_n$ is a component of the form
%
%
$$\includegraphics[scale=0.6]{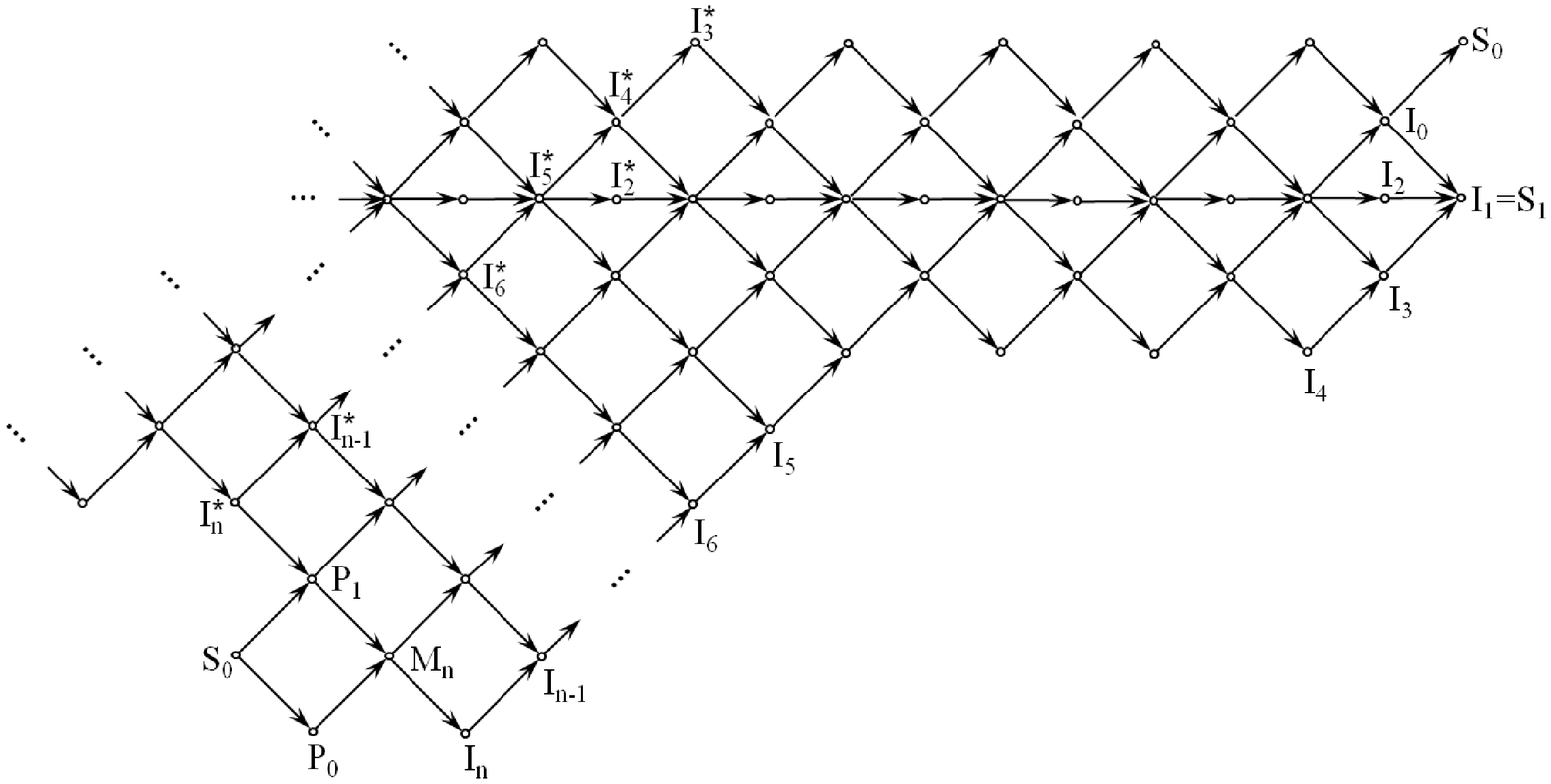}$$
Observe that $\mathcal{C}_n$ is an almost acyclic component of $\Gamma_{A_n}$, contains the faithful module $M_n$, and is closed under successors in $\ind A_n$.
Hence $\mathcal{C}_n$ is a faithful, almost acyclic, generalized standard component of $\Gamma_A$. Then it follows from \cite[Theorem 3.1]{RS2} that $A_n$ is
a generalized double tilted algebra and $\mathcal{C}_n$ is its unique connecting component. We also note that $\mathcal{C}_n$ admits a unique multisection
$\Delta=\Delta_n$ consisting of all indecomposable modules in $\mathcal{C}_n$ which lie on oriented cycles passing through the simple module $S_0$.
Moreover, we have $\Delta'_l=\Delta=\Delta'_r$, and hence $\Delta=\Delta_c$. Further, the left part $\Delta_l$ of $\Delta$ coincides with $\tau_{A_n}\Delta''_r$
and consists of the indecomposable modules $I^*_j$, for $j\in\{2, 3, 4, \ldots, n-1, n\}$. Similarly, the right part $\Delta_r$ of $\Delta$ coincides with
$\tau_{A_n}^{-1}\Delta''_l$ and consists of the indecomposable injective modules $I_1, I_2, I_3, I_4$. Therefore, the left tilted part $A^{(l)}_n$ of $A_n$ is the
hereditary algebra $H_n$ and the right tilted part $A^{(r)}_n$ of $A_n$ is the path algebra $K\Sigma$ of the quiver $\Sigma$ of the form
$$\xymatrix@C=20pt@R=10pt{
2&1\ar[l]_{\beta}\ar[r]^{\varrho}&3\ar[r]^{\delta}&4.\\
}$$
Observe now that $\Delta=\Delta_c$ is a cycle-finite finite component of $_c\Gamma_{A_n}$ containing the faithful indecomposable module $M_n$,
because $\mathcal{C}_n$ is a generalized standard component of $\Gamma_A$ closed under successors in $\ind A_n$. In particular, we conclude that $\Delta$
is the cyclic component $\Gamma(M_n)$ of the module $M_n$ and $B(\Gamma(M_n))=B(M_n)=A_n$ is a generalized double tilted algebra. We also mention that
$A_7, A_8, A_9$ are of finite representation type with $\Gamma_{A_7}={\mathcal C}_7$, $\Gamma_{A_8}={\mathcal C}_8$, $\Gamma_{A_9}={\mathcal C}_9$,
and hence are cycle-finite algebras. Further, $A_{10}$ is a cycle-finite algebra of infinite representation type whose Auslander-Reiten quiver has the disjoint
union decomposition
$$\Gamma_{A_{10}}={\mathcal P}(H_{10})\cup {\mathcal T}^{H_{10}}\cup {\mathcal C}_{10},$$
where ${\mathcal P}(H_{10})$ is the postprojective component and ${\mathcal T}^{H_{10}}$ an infinite family of pairwise orthogonal generalized standard stable tubes
of $\Gamma_{H_{10}}$. On the other hand, the algebras $A_n$, for $n\geq 11$, are not cycle-finite because their Auslander-Reiten quivers admit regular components of
$\Gamma_{H_n}$ being of the form ${\Bbb Z}{\Bbb A}_{\infty}$, and hence consisting of indecomposable modules lying on infinite cycles (see \cite[Theorem]{Sk4}).
More precisely, for $n\geq 11$, the Auslander-Reiten quiver of $A_n$ has the disjoint
union decomposition
$$\Gamma_{A_{n}}={\mathcal P}(H_{n})\cup {\mathcal R}(H_{n})\cup {\mathcal C}_{n},$$
where ${\mathcal P}(H_{n})$ is the postprojective component and ${\mathcal R}(H_{n})$ is an infinite family of regular components of the form ${\Bbb Z}{\Bbb A}_{\infty}$
in $\Gamma_{H_{n}}$. We also mention that the algebras $A_n$, for $n\geq 7$, are of infinite global dimension, because the simple module $S_0$ is
of infinite projective dimension.
\end{ex}
\begin{ex} \label{ex6-2}
Let $K$ be a field, $m, n\geq 8$ natural numbers, and $B_{m,n}=KQ_{m,n}/I_{m,n}$ the bound quiver algebra given by the quiver $Q_{m,n}$ of the form
\[
\xymatrix@C=12pt@R=18pt{
&&3\ar[lld]_{\varrho}&&&&0\ar[rd]^{\beta}&&&&3'\ar[lld]_{\psi}\cr
6\ar[d]_{\alpha_7}&&&&2\ar[llu]_{\xi}\ar[ld]_{\eta}\ar[r]^{\sigma}&1\ar[ru]^{\alpha}\ar[rr]_{\gamma}&&1'\ar[r]^{\delta}&2'&&&&6'\ar[llu]_{\varphi}\ar[ld]_{\lambda}\cr
7\ar[d]_{\alpha_8}&5\ar[lu]_{\mu}&&4\ar[ll]_{\omega}&&&&&&4'\ar[lu]_{\nu}&&5'\ar[ll]_{\theta}&7'\ar[u]_{\beta_7}\cr
\vdots\ar[d]_{\alpha_{m-1}}&&&&&&&&&&&&\vdots\ar[u]_{\beta_8}\cr
m-1\ar[rr]^{\alpha_m}&&m&&&&&&&&n'\ar[rr]^{\beta_n}&&(n-1)'\ar[u]_{\beta_{n-1}}\cr
}
\]
and $I_{m,n}$ the ideal in the path algebra $KQ_{m,n}$ of $Q_{m,n}$ over $K$ generated by the elements
$\alpha\beta$, $\sigma\alpha$, $\beta\delta$, $\sigma\gamma\delta$, $\xi\varrho-\eta\omega\mu$, $\varphi\psi-\lambda\theta\nu$.
Then the category $\mo B_{m,n}$ is equivalent to the category $\rep_K(Q_{m,n},I_{m,n})$ of the $K$-linear representations of the bound quiver $(Q_{m,n},I_{m,n})$.
Consider the indecomposable module $M_m$ in $\mo B_{m,n}$ corresponding to the indecomposable representation in $\rep_K(Q_{m,n},I_{m,n})$ of the form
\[
\xymatrix@C=12pt@R=18pt{
&&K\ar[lld]_{1}&&&&K\ar[rd]^{1}&&&&0\ar[lld]\cr
K\ar[d]_{1}&&&&K\ar[llu]_{1}\ar[ld]_{1}\ar[r]^{1}&K\ar[ru]^{0}\ar[rr]_{1}&&K\ar[r]&0&&&&0\ar[llu]\ar[ld]\cr
K\ar[d]_{1}&K\ar[lu]_{1}&&K\ar[ll]_{1}&&&&&&0\ar[lu]&&0\ar[ll]&0\ar[u]\cr
\vdots\ar[d]_{1}&&&&&&&&&&&&\vdots\ar[u]\cr
K\ar[rr]^{1}&&K&&&&&&&&0\ar[rr]&&0\ar[u]\cr
}
\]
and the indecomposable module $N_n$ in $\mo B_{m,n}$ corresponding to the indecomposable representation in $\rep_K(Q_{m,n},I_{m,n})$ of the form
\[
\xymatrix@C=12pt@R=18pt{
&&0\ar[lld]&&&&K\ar[rd]^{0}&&&&K\ar[lld]_{1}\cr
0\ar[d]&&&&0\ar[llu]\ar[ld]\ar[r]&K\ar[ru]^{1}\ar[rr]_{1}&&K\ar[r]^{1}&K&&&&K\ar[llu]_{1}\ar[ld]_{1}\cr
0\ar[d]&0\ar[lu]&&0\ar[ll]&&&&&&K\ar[lu]_{1}&&K\ar[ll]_{1}&K\ar[u]_{1}\cr
\vdots\ar[d]&&&&&&&&&&&&\vdots\ar[u]_{1}\cr
0\ar[rr]&&0&&&&&&&&K\ar[rr]^{1}&&K\ar[u]_{1}\cr
}
\]
We note that $M_m\oplus N_n$ is a faithful $B_{m,n}$-module.

Let $\Omega_m$ be the subquiver of $Q_{m,n}$ given by the vertices $3, 4, 5, 6, 7, \ldots, m-1, m$ and the arrows
$\varrho$, $\omega$, $\mu$, $\alpha_7$, $\alpha_8, \ldots, \alpha_{m-1}, \alpha_m$, and $H_m=K\Omega_m$ the path algebra of $\Omega_m$ over $K$.
Similarly, let $\Omega'_n$ be the subquiver of $Q_{m,n}$ given by the vertices $3', 4', 5', 6', 7', \ldots, (n-1)', n'$ and the arrows
$\varphi$, $\theta$, $\lambda$, $\beta_7$, $\beta_8, \ldots, \beta_{n-1}, \beta_n$, and $H'_n=K\Omega'_n$ the path algebra of $\Omega'_n$ over $K$.
Then $H_m$ and $H'_n$ are hereditary algebras. Moreover, $H_8$ and $H'_8$ are of Dynkin type ${\Bbb E}_6$, $H_9$ and $H'_9$ are of Dynkin type ${\Bbb E}_7$,
$H_{10}$ and $H'_{10}$ are of Dynkin type ${\Bbb E}_8$, $H_{11}$ and $H'_{11}$ are of Euclidean type $\widetilde{{\Bbb E}}_8$, and $H_m$ and $H'_n$,
for $m, n\geq 12$, are of wild type. For each $i\in\{3, 4, \ldots, m-1, m\}$, we denote by $I_i^*$ the indecomposable injective $H_m$-module at the vertex $i$.
Similarly, for each $j'\in\{3', 4', \ldots, (n-1)', n'\}$, we denote by $P_{j'}^*$ the indecomposable projective $H'_n$-module at the vertex $j'$.
Furthermore, for each vertex $i$ of $Q_{m,n}$, we denote by $P_i, I_i, S_i$ the indecomposable projective module, the indecomposable injective module,
and the simple module in $\mo B_{m,n}$ at the vertex $i$. Finally, we denote by $\Sigma$ the subquiver of $Q_{m,n}$ given by the vertices $0, 1, 1'$ and the arrows
$\alpha, \beta, \gamma$, and $\Lambda=K\Sigma/J$ the bound quiver algebra with $J$ the ideal in the path algebra $K\Sigma$ of $\Sigma$ over $K$ generated by
$\alpha\beta$. We denote by $R$ and $T$ the indecomposable modules in $\mo\Lambda$ corresponding to the representations
\[
\xymatrix@C=20pt@R=18pt{
&0\ar[rd]&&&&K^2\ar[rd]^{\left[\begin{smallmatrix}0&1\end{smallmatrix}\right]}\cr
K\ar[rr]_{1}\ar[ru]&&K&{\rm and}&K\ar[rr]_{1}\ar[ru]^{\left[\begin{smallmatrix}1\\0\end{smallmatrix}\right]}&&K\cr
}
\]
in $\rep_K(\Sigma,J)$, respectively. Moreover, denote by $\overline{P}_1$ the indecomposable projective $\Lambda$-module at the vertex $1$ and
by $\overline{I}_{1'}$ the indecomposable injective $\Lambda$-module at the vertex $1'$, and observe that $P_0$ is the indecomposable projective
$\Lambda$-module at the vertex $0$ and $I_0$ is the indecomposable injective $\Lambda$-module at the vertex $0$.

We claim that $B_{m,n}$ is a generalized double tilted algebra and the indecomposable modules $M_m$ and $N_n$ belong to a cycle-finite cyclic component
$\Gamma_{m,n}$, and hence $B(\Gamma_{m,n})=B_{m,n}$. More precisely, we will show that $\Gamma_{m,n}$ is the cyclic part of the almost acyclic generalized
standard component ${\mathcal C}_{m,n}$ of $\Gamma_{B_{m,n}}$ obtained by identification the modules $R, T$ and $S_0$ occurring in the following two
translation quivers: ${\mathcal C}^-_m$ of the form
%
%
$$\includegraphics[scale=0.63]{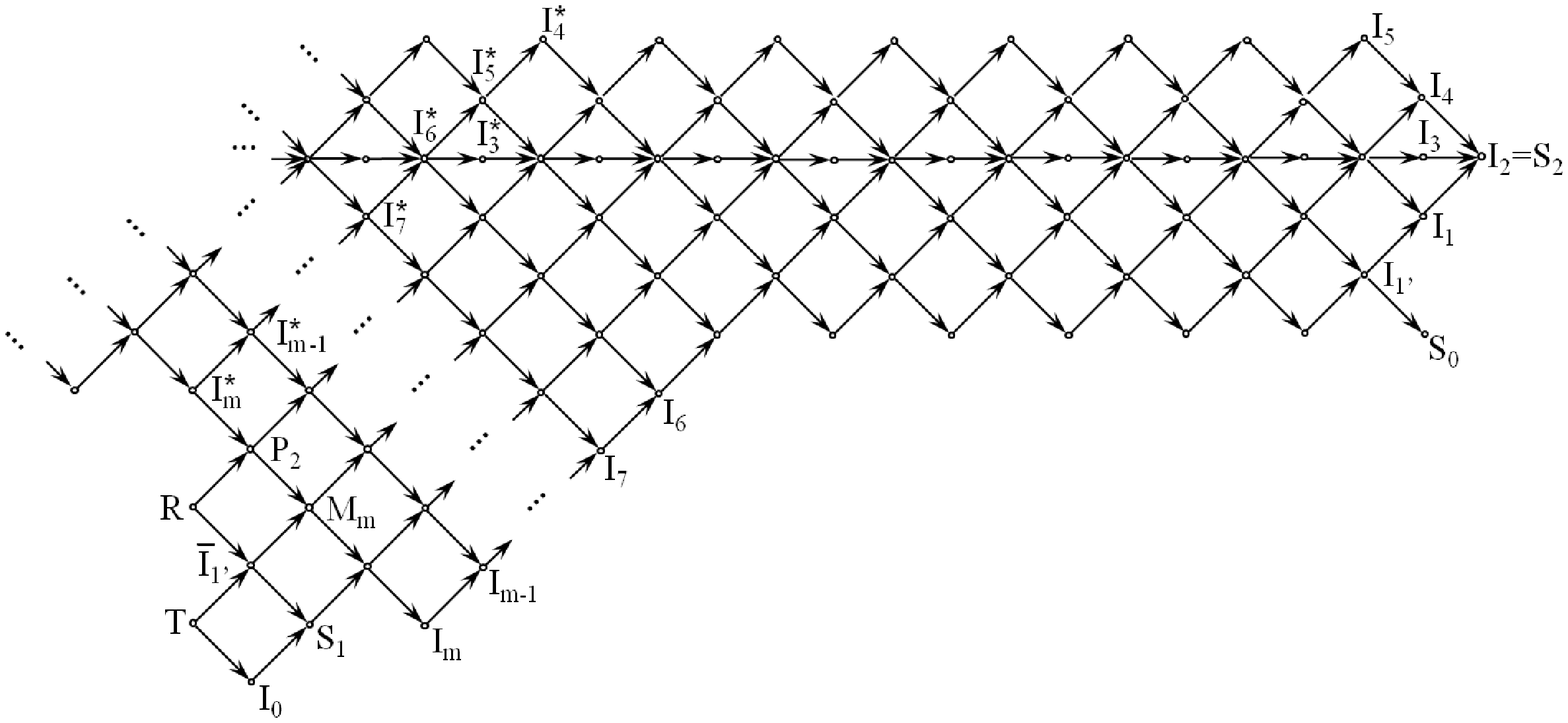}$$
and ${\mathcal C}^+_n$ of the form
%
$$\hskip -3mm\includegraphics[scale=0.63]{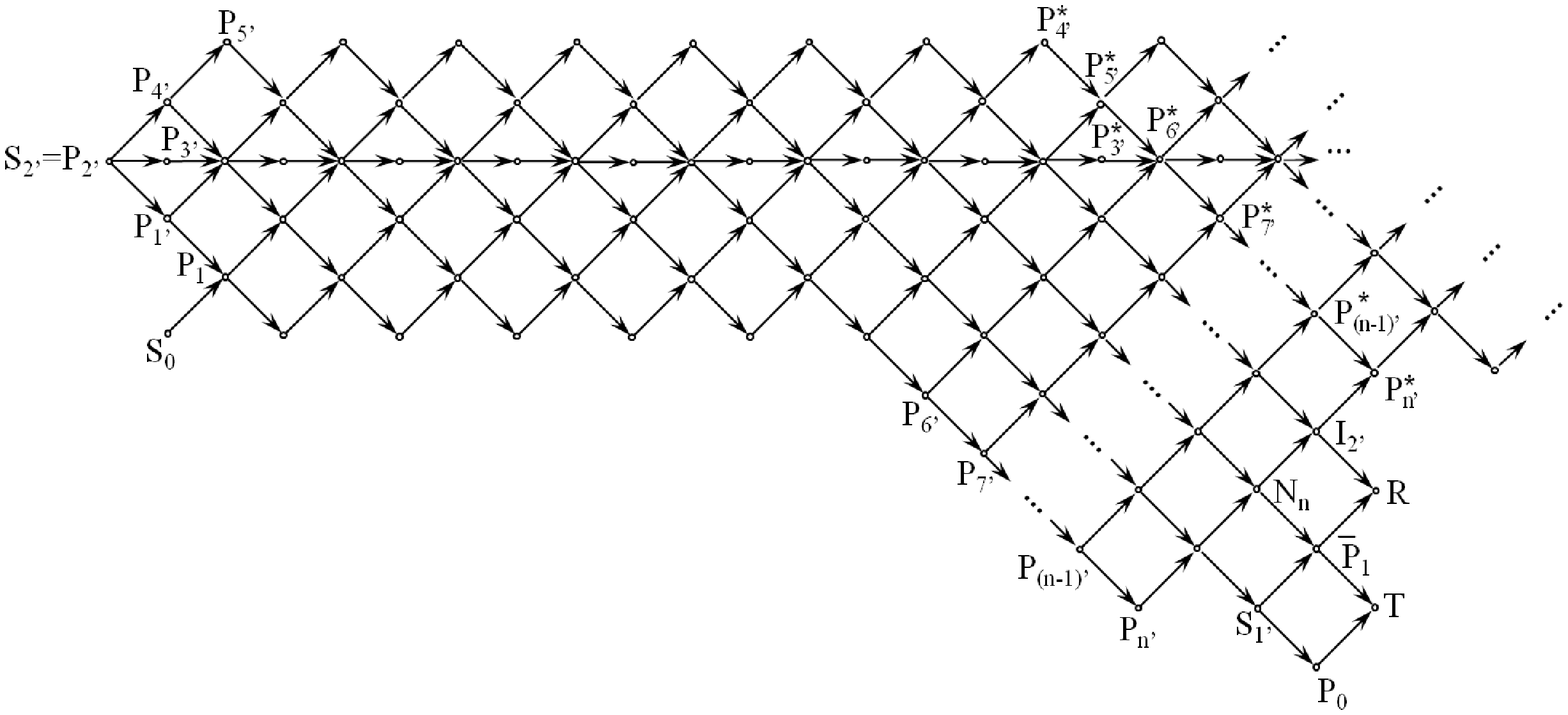}$$
Let $Q_m$ be the subquiver of $Q_{m,n}$ given by the vertices $1'$, $0$, $1$, $2$, $3$, $4$, $5$, $6$, $7, \ldots, m-1, m$ and the arrows $\alpha$, $\beta$, $\gamma$, $\xi$,
$\varrho$, $\eta$, $\omega$, $\mu$, $\alpha_7$, $\alpha_8, \ldots, \alpha_{m-1}, \alpha_m$, $J_m$ the ideal in the path algebra $KQ_m$ of $Q_m$ over $K$ generated
by $\alpha\beta$, $\sigma\alpha$, $\xi\varrho-\eta\omega\mu$, and $C_m=KQ_m/J_m$ the associated bound quiver algebra. Then $C_m$ is the one-point extension algebra
\[
\left[\begin{matrix}\Lambda\times H_m&0 \\ R\oplus I^*_m&K\end{matrix}\right]
\]
of $\Lambda\times H_m$ by the module $R\oplus I^*_m$, with the extension vertex $2$.
Since $I^*_m$ is the indecomposable injective module over the hereditary algebra $H_m$, we conclude that
$$\Hom_{\Lambda\times H_m}(R\oplus I^*_m,\tau_{H_m}X)=\Hom_{H_m}(I^*_m,\tau_{H_m}X)=0$$
for any indecomposable module $X$ in $\mo H_m$. Then, applying \cite[Corollary XV.1.7]{SS2} (or \cite[Lemma 5.6]{SY0}), we conclude that every almost split sequence
in $\mo H_m$ is an almost split sequence in $\mo C_m$.
This implies that the Auslander-Reiten quiver $\Gamma_{H_m}$ of $H_m$ is a full translation subquiver of the Auslander-Reiten quiver $\Gamma_{C_m}$ of $C_m$.
Moreover, a direct calculation shows that the component ${\mathcal C}_m$ of $\Gamma_{C_m}$, containing the indecomposable injective $H_m$-modules $I^*_i$,
$i\in\{3, 4, 5, 6, 7, \ldots, m-1, m\}$, is the translation quiver obtained from the translation quiver $\mathcal{C}^-_m$ and the translation quiver below
\[
\xymatrix@C=10pt@R=10pt{
S_0\ar[rd]&&R\ar[rd]\cr
&\overline{P}_1\ar[ru]\ar[rd]&&\overline{I}_{1'}\ar[rd]\cr
S_{1'}\ar[ru]\ar[rd]&&T\ar[ru]\ar[rd]&&S_1\cr
&P_0\ar[ru]&&I_{0}\ar[ru]\cr
}
\]
by identifying the common modules $R$, $\overline{I}_{1'}$, $S_1$, $T$, $I_0$ and $S_0$. We observe that the Auslander-Reiten quiver $\Gamma_{C_m}$ consists
of the component ${\mathcal C}_m$ and the components of $\Gamma_{H_m}$ different from the preinjective component.
We also note that the indecomposable module $M_m$ is a unique sincere module in $\ind C_m$, and $M_m$ is not a faithful module in $\mo C_m$.

Dually, let $Q'_n$ be the subquiver of $Q_{m,n}$ given by the vertices $1$, $0$, $1'$, $2'$, $3'$, $4'$, $5'$, $6'$, $7', \ldots, (n-1)', n$ and the arrows
$\alpha$, $\beta$, $\gamma$, $\delta$, $\psi$, $\varphi$, $\nu$, $\theta$, $\lambda$, $\beta_7$, $\beta_8, \ldots, \beta_{n-1}, \beta_n$, $J'_n$ the ideal in
the path algebra $KQ'_n$ of $Q'_n$ over $K$ generated by $\alpha\beta$, $\beta\delta$, $\varphi\psi-\lambda\theta\nu$, and $C'_n=KQ'_n/J'_n$ the associated
bound quiver algebra. Then $C'_n$ is the one-point coextension algebra
\[
\left[\begin{matrix} K&0 \\ \Hom_K(R\oplus P^*_{n'},K) & \Lambda\times H'_n   \end{matrix}\right]
\]
of $\Lambda\times H'_n$ by the module $R\oplus P^*_{n'}$, with the coextension vertex $2'$.
Since $P^*_{n'}$ is the indecomposable projective module over the hereditary algebra $H'_n$, we conclude that
$$\Hom_{\Lambda\times H'_n}(\tau^{-1}_{H'_n}Y,R\oplus P^*_{n'})=\Hom_{H'_n}(\tau^{-1}_{H'_n}Y,P^*_{n'})=0$$
for any indecomposable module $Y$ in $\mo H'_n$. Then, applying the dual of \cite[Corollary XV.1.7]{SS2} (or \cite[Lemma 5.6]{SY0}), we conclude that every
almost split sequence in $\mo H'_n$ is an almost split sequence in $\mo C'_n$.
This implies that the Auslander-Reiten quiver $\Gamma_{H'_n}$ of $H'_n$ is a full translation subquiver of the Auslander-Reiten quiver $\Gamma_{C'_n}$ of $C'_n$.
Moreover, a direct calculation shows that the component ${\mathcal C}'_n$ of $\Gamma_{C'_n}$, containing the indecomposable projective $H'_n$-modules $P^*_{j'}$,
$j'\in\{3', 4', 5', 6', 7', \ldots, (n-1)', n'\}$, is the translation quiver obtained from the translation quiver $\mathcal{C}^+_n$ and the translation quiver below
\[
\xymatrix@C=10pt@R=10pt{
&&R\ar[rd]&&S_0\cr
&\overline{P}_1\ar[ru]\ar[rd]&&\overline{I}_{1'}\ar[rd]\ar[ru]\cr
S_{1'}\ar[ru]\ar[rd]&&T\ar[ru]\ar[rd]&&S_1\cr
&P_0\ar[ru]&&I_{0}\ar[ru]\cr
}
\]
by identifying the common modules $S_{1'}$, $P_0$, $\overline{P}_{1}$, $R$, $T$ and $S_0$. We observe that the Auslander-Reiten quiver $\Gamma_{C'_n}$ consists
of the component ${\mathcal C}'_n$ and the components of $\Gamma_{H'_n}$ different from the postprojective component.
We also note that the indecomposable module $N_n$ is a unique sincere module in $\ind C'_n$, and $N_n$ is not a faithful module in $\mo C'_n$.

Further, we observe that the algebra $B_{m,n}=KQ_{m,n}/I_{m,n}$ is the one-point extension algebra
\[
\left[\begin{matrix}C'_n\times H_m&0 \\ R\oplus I^*_m&K\end{matrix}\right]
\]
of $C'_n\times H_m$ by the module $R\oplus I^*_m$, with the extension vertex $2$.
It follows from the structure of the Auslander-Reiten quiver $\Gamma_{C'_n}$ of $C'_n$ that, for any indecomposable module $Z$ in $\mo C'_n$ nonisomorphic to
the simple module $S_0$, we have
$$\Hom_{C'_n\times H_m}(R\oplus I^*_m,\tau_{C'_n}Z)=\Hom_{C'_n}(R,\tau_{C'_n}Z)=0.$$
Then, applying \cite[Corollary XV.1.7]{SS2} (or \cite[Lemma 5.6]{SY0}) again, we conclude that every almost split sequence in $\mo C'_n$ with the right
term nonisomorphic to $S_0$ is an almost split sequence in $\mo B_{m,n}$.
This shows that the translation quiver obtained from $\Gamma_{C'_n}$ by removing the module $S_0$ and two arrows attached to it is a full translation subquiver
of $\Gamma_{B_{m,n}}$. In particular, we conclude that the almost split sequence in $\mo C'_n$ with the left term $S_0$ is an almost split sequence in $\mo B_{m,n}$.

Finally, we observe that the algebra $B_{m,n}=KQ_{m,n}/I_{m,n}$ is also the one-point coextension algebra
\[
\left[\begin{matrix} K&0 \\ \Hom_K(R\oplus P^*_{n'},K) & C_m\times H'_n   \end{matrix}\right]
\]
of $C_m\times H'_n$ by the module $R\oplus P^*_{n'}$, with the coextension vertex $2'$.
It follows also from the structure of the Auslander-Reiten quiver $\Gamma_{C_m}$ of $C_m$ that, for any indecomposable module $Z$ in $\mo C_m$ nonisomorphic to
the simple module $S_0$, we have
$$\Hom_{C_m\times H'_n}(\tau_{C_m}^{-1}Z,R\oplus P^*_{n'})=\Hom_{C_m}(\tau_{C_m}^{-1}Z,R)=0.$$
Then, applying the dual of \cite[Corollary XV.1.7]{SS2} (or \cite[Lemma 5.6]{SY0}) again, we conclude that every almost split sequence in $\mo C_m$ with the left
term nonisomorphic to $S_0$ is an almost split sequence in $\mo B_{m,n}$.
This shows that the translation quiver obtained from $\Gamma_{C_m}$ by removing the module $S_0$ and two arrows attached to it is a full translation subquiver
of $\Gamma_{B_{m,n}}$. In particular, we obtain that the almost split sequence in $\mo C_m$ with the right term $S_0$ is also an almost split sequence in $\mo B_{m,n}$.

Summing up, we proved that $\Gamma_{B_{m,n}}$ contains the component ${\mathcal C}_{m,n}$ of the required form, containing the preinjective component ${\mathcal Q}(H_m)$
of $\Gamma_{H_m}$ as a full translation subquiver closed under predecessors and the postprojective component ${\mathcal P}(H'_n)$ of $\Gamma_{H'_n}$ as a full
translation subquiver closed under successors. Moreover, the Auslander-Reiten quiver $\Gamma_{B_{m,n}}$ of $B_{m,n}$ has a disjoint union decomposition
\[\Gamma_{B_{m,n}} = {\mathcal P}_{m,n} \cup {\mathcal C}_{m,n} \cup {\mathcal Q}_{m,n} \]
such that
\begin{itemize}
\item ${\mathcal P}_{m,n}$ is empty for $m\in\{8,9,10\}$;
\item ${\mathcal P}_{11,n}$ consists of the postprojective component ${\mathcal P}(H_{11})$ of Euclidean type $\widetilde{{\Bbb E}}_8$ and
an infinite family ${\mathcal T}^{H_{11}}$ of pairwise orthogonal generalized standard stable tubes in $\Gamma_{H_{11}}$;
\item ${\mathcal P}_{m,n}$, for $m\geq 12$, consists of the postprojective component ${\mathcal P}(H_{m})$ of wild type and an infinite family of regular
components of the form ${\Bbb Z}{\Bbb A}_{\infty}$ in $\Gamma_{H_{m}}$;
\item ${\mathcal Q}_{m,n}$ is empty for $n\in\{8,9,10\}$;
\item ${\mathcal Q}_{m,11}$ consists of the preinjective component ${\mathcal Q}(H'_{11})$ of Euclidean type $\widetilde{{\Bbb E}}_8$ and
an infinite family ${\mathcal T}^{H'_{11}}$ of pairwise orthogonal generalized standard stable tubes in $\Gamma_{H'_{11}}$;
\item ${\mathcal Q}_{m,n}$, for $n\geq 12$, consists of the preinjective component ${\mathcal Q}(H'_{n})$ of wild type and an infinite family of regular
components of the form ${\Bbb Z}{\Bbb A}_{\infty}$ in $\Gamma_{H'_{n}}$.
\end{itemize}

Finally, observe that ${\mathcal C}_{m,n}$ is an almost acyclic component of $\Gamma_{B_{m,n}}$ whose cyclic part $\Gamma_{m,n}$ is connected and consists
of all indecomposable modules in ${\mathcal C}_{m,n}$ which lie on oriented cycles passing through the simple module $S_0$. In fact, $\Gamma_{m,n}$ is the
unique multisection $\Delta$ of ${\mathcal C}_{m,n}$, and so $\Delta_c=\Gamma_{m,n}$. Further, $\Gamma_{m,n}$ contains the indecomposable modules $M_m$ and $N_n$.
Since $M_m\oplus N_n$ is a faithful module in $\mo B_{m,n}$, we conclude that $\Gamma_{m,n}$ is a faithful cyclic component of $\Gamma_{B_{m,n}}$, and hence
$B_{m,n}=B(\Gamma_{m,n})=B_{m,n}/\ann_{B_{m,n}}(\Gamma_{m,n})$. Observe also that $B_{m,n}=\su(\Gamma_{m,n})$.
In particular, ${\mathcal C}_{m,n}$ is a faithful component of $\Gamma_{B_{m,n}}$. Moreover,
the left part $\Delta_l$ of $\Delta$ coincides with $\tau_{B_{m,n}}\Delta''_r$ and consists of the indecomposable modules $I_i^*$, $i\in\{3, 4, \ldots, m-1\}$,
and the indecomposable modules $P_{1'}$, $\tau^{-1}_{B_{m,n}}P_{2'}$, $\tau^{-1}_{B_{m,n}}P_{3'}$, $\tau^{-1}_{B_{m,n}}P_{4'}$, $\tau^{-1}_{B_{m,n}}P_{5'}$.
Similarly, the right part $\Delta_r$ of $\Delta$ coincides with $\tau^{-1}_{B_{m,n}}\Delta''_l$ and consists of the indecomposable modules $P_{j'}^*$,
$j'\in\{3', 4', \ldots, (n-1)'\}$, and the indecomposable modules $I_{1}$, $\tau_{B_{m,n}}I_{2}$, $\tau_{B_{m,n}}I_{3}$, $\tau_{B_{m,n}}I_{4}$, $\tau_{B_{m,n}}I_{5}$.
Observe also that ${\mathcal Q}(H_m)$ is a generalized standard component of $\Gamma_{H_m}$, ${\mathcal P}(H'_n)$ is a generalized standard component of
$\Gamma_{H'_n}$, and $\Hom_{B_{m,n}}(P,Q)=0$ for any indecomposable modules $P\in{\mathcal P}(H'_n)$ and $Q\in{\mathcal Q}(H_m)$. This shows that ${\mathcal C}_{m,n}$
is a generalized standard component of $\Gamma_{B_{m,n}}$. Then it follows from \cite[Theorem 3.1]{RS2} that $B_{m,n}$ is a generalized double tilted algebra.
Moreover, the left tilted part $B_{m,n}^{(l)}$ is the product $H_m\times H'$ of $H_m$ and the path algebra $H'=K\Omega'$ of the quiver $\Omega'$ of the form
\[
\xymatrix@C=12pt@R=18pt{
&&&3'\ar[lld]_{\psi}\cr
1'\ar[r]^{\delta}&2'\cr
&&4'\ar[lu]_{\nu}&&5'\ar[ll]_{\theta}\cr
}
\]
and Dynkin type ${\Bbb D}_5$, while the right tilted part $B_{m,n}^{(r)}$ is the product $H'_n\times H$ of $H'_n$ and the path algebra $H=K\Omega$ of the quiver $\Omega$ of the form
\[
\xymatrix@C=12pt@R=18pt{
&3\cr
&&&2\ar[llu]_{\xi}\ar[ld]_{\eta}\ar[r]^{\sigma}&1\cr
5&&4\ar[ll]_{\omega}\cr
}
\]
and Dynkin type ${\Bbb D}_5$. In particular, we obtain that $B_{m,n}$ is a tame generalized double tilted algebra (equivalently, cycle-finite algebra) if and only if
$m, n \in \{8, 9, 10, 11\}$. Clearly, $B_{m,n}$ is of finite representation type if and only if $m, n \in \{8, 9, 10\}$. Finally, we note that the algebras $B_{m,n}$,
for all $m, n\geq 8$, are of global dimension three, with the simple module $S_1$ having the projective dimension three.
\end{ex}
%
%

\end{document}